\documentclass[11pt]{article}
\usepackage[utf8]{inputenc}
\usepackage{enumerate, amsthm, amssymb,xcolor,comment,subcaption}
\usepackage[a4paper, footskip=1cm, headheight = 16pt, top=3.7cm, bottom=3cm,  right=2.5cm,  left=2.5cm ]{geometry}
\usepackage[colorlinks,linkcolor=blue,citecolor=red]{hyperref}
\usepackage{geometry}

\newtheorem{question}{Question}[section]
\newtheorem{lemma}[question]{Lemma}
\newtheorem{theorem}[question]{Theorem}

\usepackage[T1]{fontenc}

\usepackage[leqno]{amsmath}

\usepackage{todonotes}

\makeatletter
\newcommand{\leqnomode}{\tagsleft@true}
\newcommand{\reqnomode}{\tagsleft@false}
\makeatother

\usepackage{latexsym}
\usepackage{amsfonts}

\usepackage{tikz}
\usepackage{float}
\usepackage{lmodern}

\def\dd{\hbox{-}}
\def\cc{\hbox{-}\cdots\hbox{-}}
\def\ll{,\ldots,}

\usepackage{marvosym}               

\DeclareMathOperator{\tw}{tw}
\DeclareMathOperator{\width}{width}

\DeclareMathOperator{\Hub}{Hub}

\DeclareMathOperator{\hdim}{hdim}

\DeclareMathOperator{\Core}{Core}
\DeclareMathOperator{\Proj}{Proj}

\newcounter{tbox}

\newcommand{\sta}[1]{\vspace*{0.3cm}\refstepcounter{tbox}\noindent{ \parbox{\textwidth}{(\thetbox) \emph{#1}}}\vspace*{0.3cm}}

\newcommand{\mylongtitle}[1]{%
  \ifodd\value{page}%
    \protect\parbox{0.97\linewidth}{#1}\hfill%
  \else%
    \hfill\protect\parbox{0.97\linewidth}{#1}%
  \fi%
}

\makeatletter
\newcommand{\otherlabel}[2]{\protected@edef\@currentlabel{#2}\label{#1}}
\makeatother

\mathchardef\mh="2D

\title{Tree Independence Number IV. \\ Even-hole-free Graphs}

\author{
Maria Chudnovsky\thanks{$^{\dagger}$ Princeton University, Princeton, NJ, USA. 
Supported by  NSF-EPSRC Grant DMS-2120644 and by AFOSR grant FA9550-22-1-0083.}
\and
Peter Gartland\thanks{University of California, Santa Barbara, USA. 
}
\and
Sepehr Hajebi\thanks{University of Waterloo, Waterloo, Ontario, Canada. 
We acknowledge the support of the Natural Sciences and Engineering Research Council of Canada (NSERC), [funding reference number RGPIN-2020-03912].
Cette recherche a \'et\'e financ\'ee par le Conseil de recherches en sciences naturelles et en g\'enie du Canada (CRSNG), [num\'ero de r\'ef\'erence RGPIN-2020-03912]. This project was funded in part by the Government of Ontario. This research was conducted while Spirkl was an Alfred P. Sloan Fellow.}
\and
Daniel Lokshtanov\footnotemark[2]
\and
Sophie Spirkl\footnotemark[3]
}

\allowdisplaybreaks

\begin{document}

\maketitle

\begin{abstract}
We prove that the tree independence number of every even-hole-free graph is at most polylogarithmic in its number of vertices. More explicitly, we prove that there exists a constant $c>0$ such that for every integer $n>1$ every $n$-vertex even-hole-free graph has a tree decomposition where each bag has stability (independence) number at most $c \log^{10} n$.  This implies that the \textsc{Maximum Weight Independent Set} problem, as well as several other natural algorithmic problems that are known to be \textsf{NP}-hard in general, can be solved in quasi-polynomial time if the input graph is even-hole-free.
\end{abstract}

\section{Introduction}
A graph $G$ {\em even-hole-free} if $G$ does not contain a cycle of even length as an induced subgraph.
Here an {\em induced subgraph} of a graph $G$ is a graph that can be obtained from $G$ by deleting vertices (and all edges incident to the deleted vertices), the {\em length} of a cycle (or path) is the number of edges in it, and
a {\em hole} in $G$ in an induced subgraph that is a cycle of length at least four. 
The class of even-hole-free graphs has attracted much attention due to its somewhat tractable, yet quite rich structure~\cite{bisimplicialnew,decompPart1,Kristinasurvey}.  In particular the structure of even-hole-free graphs has some similarities~\cite{Kristinasurvey} with the structure of perfect graphs, which by the strong perfect graph theorem~\cite{perfect} are precisely the graphs that are odd-hole-free and whose complement is odd-hole-free. 
In addition to their structure, much effort was put into designing efficient algorithms for even-hole-free graphs (to solve problems that are NP-hard in general). This is discussed in the survey \cite{Kristinasurvey}, while \cite{AAKST, CdSHV, CGP, Le} provide examples of more recent work. We now consider some of the problems that have received the most attention on even-hole free graphs.  

A vertex set $S$ of a graph $G$ is {\em stable} (or {\em independent}) if no two vertices of $S$ have an edge between them. A {\em clique} is a set $S$ if vertices such that every pair of vertices of $S$ has an edge between them. A {\em proper $k$-coloring} of $G$ is a partition of the vertex set of $G$ into (at most) $k$ independent sets. In the {\sc Maximum Weight Independent Set} ({\sc Maximum Weight Clique}) problem, the input is a graph $G$ and a weight function that assigns to each vertex an integer weight. The task is to find a stable set (clique) $S$ in $G$ of maximum weight. In the {\sc $k$-Coloring} problem, the input is a graph $G$, the task is to determine whether $G$ has a proper $k$-coloring. Finally, in the {\sc Coloring} problem, the input is a graph $G$ and the task determine the minimum $k$ such that $G$ has a proper $k$-coloring. All of the above mentioned problems are known to be \textsf{NP}-hard~\cite{GJNPC, Karp72}
On even-hole-free graphs, {\sc Maximum Weight Clique} is known to be polynomial-time solvable~\cite{Kristinasurvey}. The questions of whether or not there exist polynomial time algorithms for {\sc Maximum Weight Independent Set} and {\sc Coloring} remain open. This is in stark contrast to perfect graphs, for which polynomial time algorithms for these problems have been known since 1981~\cite{GLS81}. 
%

This discrepancy is somewhat surprising, to explain why (and to state our main result) we need to define tree decompositions and treewidth. For a graph $G = (V(G),E(G))$, a \emph{tree decomposition} $(T, \chi)$ of $G$ consists of a tree $T$ and a map $\chi: V(T) \to 2^{V(G)}$ with the following properties: 
\begin{itemize}\setlength\itemsep{-.7pt}
    
    \item For every $v_1v_2 \in E(G)$, there exists $t \in V(T)$ with $v_1, v_2 \in \chi(t)$.
    
    \item For every $v \in V(G)$, the subgraph of $T$ induced by $\{t \in V(T) \mid v \in \chi(t)\}$ is non-empty and connected.
\end{itemize}

The \emph{width} of a tree decomposition $(T, \chi)$, denoted by $\width(T, \chi)$, is $\max_{t \in V(T)} |\chi(t)|-1$. The \emph{treewidth} of $G$, denoted by $\tw(G)$, is the minimum width of a tree decomposition of $G$. Bounded treewidth is a useful graph property from both a structural \cite{RS-GMXVI} and an algorithmic \cite{Bodlaender1988DynamicTreewidth} perspective.
For example {\sc Maximum Weight Independent Set}, {\sc Maximum Weight Clique}, {\sc $k$-Coloring} (for every fixed $k$) and a host of other problems are known to admit $O(c^tn)$ time algorithms on graphs of treewidth at most $t$~\cite{cyganBook,dfbook,fgbook}, while {\sc Coloring} is known to admit $O(t^{t+ O(1)}n)$ time algorithms on graphs of treewidth at most $t$.

From the perspective of tree decompositions and treewidth, perfect graphs appear much more intractable than even-hole-free graphs. On one hand there exist triangle-free (a triangle is a clique on $3$ vertices) perfect graphs whose treewidth is {\em linear} in the number of vertices in the graph: the complete bipartite graph $K_{t,t}$, consisting of two stable sets $L$ and $R$ of size $t$ with an edge connecting every vertex in $L$ with every vertex in $R$, is an example. On the other hand, triangle-free even-hole-free graphs have constant treewidth \cite{CdSHV}. 
Sintiari and Trotignon~\cite{mainconj} give a construction of arbitrarily large $K_4$-free ($K_t$ is the clique on $t$ vertices) even-hole-free graphs whose treewidth is {\em logaritmic} in the number of vertices.
%
This led Sintiari and Trotignon~\cite{mainconj} to conjecture that for every $t$ there exists a constant $c_t$ such that every $n$-vertex $K_t$-free and even-hole-free graph has treewidth at most $c_t \log n$. This conjecture was very recently confirmed by Chudnovsky et al.~\cite{TWXV}. The logarithmic treewidth bound of Chudnovsky et al.~\cite{TWXV} immediately implies that $k$-{\sc Coloring} can be solved in polynomial time for every fixed $k$ on even-hole-free graphs.

An early step of the proof of Chudnovsky et al.~\cite{TWXV} is to prove that even-hole-free graphs admit ``dominated balanced separators''. More precisely they show that there exists a constant $c$ such that every even-hole-free graph contains a vertex set $S$ of size at most $c$ such that every connected component of $G-N[S]$ has at most $n/2$ vertices. Here $N[S]$ is the {\em closed neighborhood} of $S$, namely the set of all vertices in $S$ and all vertices with at least one neighbor in $S$.
A fairly direct consequence of this result (based on an argument of Chudnovsky et al.~\cite{qptases}) is that for every $\epsilon > 0$ there exists a $(1+\epsilon)$-approximation algorithm for {\sc Maximum Weight Independent Set} which runs in quasi-polynomial time. Here an algorithm runs in {\em quasi-polynomial} time if the running time is upper bounded by $2^{O(\log^c n)}$ for some constant $c$. Nevertheless, this is fully consistent with  {\sc Maximum Weight Independent Set} on even-hole-free graphs being \textsf{NP}-hard, and so the complexity of  {\sc Maximum Weight Independent Set} and {\sc Coloring} on even-hole-free graphs remain open. 

\smallskip
\noindent
{\bf Our results.} We prove a structural result which implies that  {\sc Maximum Weight Independent Set}, as well as a host of other problems, admit quasi-polynomial time algorithms on even-hole-free graphs. 
To state our main result we need one more notion, that of tree independence number. This is a relatively new width parameter, defined by Dallard, Milani\v{c} and \v{S}torgel \cite{dms2}, in the second of a series of papers~\cite{dfgkm, dms4, dms1,dms2,dms3} aiming to  identify graphs whose large treewidth can be completely explained by the presence of a large clique.
The {\em tree independence number} of a tree decomposition $(T, \chi)$ is the maximum over all $t \in V(T)$ of the maximum stable set size of the subgraph $G[\chi(t)]$ of $G$ induced by $\chi(t)$. The tree independence number of a graph $G$ is the minimum tree independence number of a tree decomposition of $G$. We are now ready to state our main result.

\begin{theorem} \label{main}
  There exists a constant $c$ such that for every integer $n>1$ every
  $n$-vertex even-hole-free graph has tree independence number 
  at most $c \log^{10}  n$.
\end{theorem}

Since the only construction of even-hole-free graphs with large treewidth known to date is the construction of \cite{mainconj},  where all graphs have clique number at most four and have treewidth logarithmic in the number of vertices,  we do not know if the bound of Theorem~\ref{main} is asymptotically tight, or whether the exponent of $\log n$ can be reduced.
Dallard et al.~\cite{dfgkm} gave an algorithm that takes as input a graph $G$ and integer $k$, runs in time $2^{O(k^2)}n^{O(k)}$ and either outputs a tree decomposition of $G$ with independence number at most $8k$, or determines that the tree-independence number of $G$ is larger than $k$. Using this algorithm, Theorem~\ref{main} can be made constructive in the sense that there exists an algorithm which takes as input an even-hole-gree graph, runs in time $2^{O(\log^{20}n)}$ and computes a tree decomposition of $G$ with tree independence $O(\log^{10} n)$.

Theorem~\ref{main} implies the main result of~\cite{TWXV} with a $O(\log^{10} n)$ instead of logarithmic bound on the treewidth. Indeed, let $G$ be a $K_t$-free even-hole free graph, and let $(T, \chi)$ be the tree decomposition obtained from Theorem~\ref{main}. We claim that this decomposition has width $O(\log^{10} n)$. This follows from the fact that every even-hole-free graph on at least $2 \alpha t$ vertices either has a clique of size at least $t$ or a stable set of size at least $\alpha$~\cite{bisimplicialnew}. Thus every bag of the decomposition must have size at most $c \log^{10} n \cdot 2t$. 

A number of problems can be solved efficiently when a tree decomposition of the input graph of low independence number is given as input. This is discussed in more detail in \cite{TI2} and \cite{lima2024tree}. We will not repeat the discussion here, but only list some of the problems that can be solved in quasi-polynomial time in the class of even-hole-free graphs as a direct consequence of Theorem~\ref{main}, the above mentioned approximation algorithm of Dallard et al.~\cite{dfgkm}, and existing algorithms when a tree decomposition of the input graph of low independence number (or low width) is given as input.





\begin{theorem}\label{mainApplications}
For every integer $k \geq 0$, the following problems admit quasi-polynomial time time algorithms on even-hole free graphs:
\begin{itemize}\setlength\itemsep{-.7pt}
  \item \textsc{Maximum Weight Independent Set},
  \item \textsc{Weighted Feedback Vertex Set}, 
  \item \textsc{Weighted Odd cycle transversal},
  \item \textsc{Maximum Weight Induced Path},
  \item  \textsc{Maximum Weight Induced Matching},
  \item \textsc{Maximum Weight Induced Subgraph of Treewidth at most $k$},
  \item \textsc{$k$-coloring}. 
\end{itemize}
\end{theorem}

Resolving the complexity status of {\sc Maximum Weight Independent Set} on even-hole-free graphs has been stated as an open problem a number of times~\cite{adlerRank,CdSHV, CGP,husicFPT, Le}, and the complexity of \textsc{Feedback Vertex Set} on even-hole-free graphs has been posed at least once~\cite{fvsEven1}.  
Even though Theorem~\ref{mainApplications} does not fully resolve these open problems, it offers a partial resolution in the following sense. If an NP-hard problem has a quasi-polynomial-time algorithm then every problem in NP has a quasi-polynomial-time algorithm. Thus Theorem~\ref{mainApplications} implies that, unless a highly unexpected complexity theoretic collapse occurs, none of the above problems are \textsf{NP}-hard in even-hole-free graphs. It remains an intriguing and challenging open problem to design {\em polynomial} time algorithms for the problems in Theorem~\ref{mainApplications} 
(with the exception of \textsc{$k$-coloring}, for which a polynomial time algorithm was recently found~\cite{TWXV})
on even-hole-free graphs.
Polynomial time algorithms for these problems would require significant new ideas. Indeed, even if the exponent of $\log n$ in Theorem~\ref{main} were reduced to 1, the resulting algorithms would still take quasi-polynomial (rather than polynomial) time.

It is worth noting that \textsc{Feedback Vertex Set}, \textsc{Maximum Weight Induced Path}, and \textsc{Maximum Weight Induced Matching} are all \textsf{NP}-hard on bipartite graphs, and therefore on perfect graphs. This partially confirms the intuition that even-hole-free graphs should be more algorithmically tractable than perfect graphs.  

All of the results of Theorem~\ref{mainApplications} (except for $k$-{\sc Coloring}) can be derived from the following theorem.  The theorem uses Counting Monadic Second Order Logic (CMSO$_2$), which is a useful formalism to express properties of graphs and vertex and edge sets~\cite{bpt92,lima2024tree}. We refer the reader to Lima et al.~\cite{lima2024tree} for an introduction to CMSO$_2$ logic.

\begin{theorem}\label{thm:alg_general}
For every integer $\ell$ and CMSO$_2$ formula $\phi$, there exists an algorithm that takes as input an even-hole-free graph $G$  and a weight function $w : V(G) \rightarrow \mathbb{N}$, runs in time $(f(\phi,\ell)n)^{O(\ell \log^{10} n)}$ and outputs a maximum weight vertex subset $S$ such that $G[S]$ has treewidth at most $\ell$ and $G[S] \models \phi$.
\end{theorem}

Theorem~\ref{thm:alg_general} is obtained from Theorem~\ref{main} in (exactly) the same way as Chudnovsky et al.~\cite{TI2} obtain their algorithmic consequences (Theorem 8.2 of~\cite{TI2}) from their bound on the tree independence number (Theorem 1.2 of~\cite{TI2}) of $3$PC-free graphs.
The algorithm for \textsc{Odd Cycle Transversal} in Theorem~\ref{mainApplications} follows from Theorem~\ref{thm:alg_general} because on even-hole-free graphs, \textsc{Odd Cycle Transversal} and \textsc{Feedback Vertex Set} are the same problem. 

The quasi-polynomial time algorithm of Theorem~\ref{mainApplications} for $k$-{\sc Coloring} follows from the fact that $K_{k+1}$ is not $k$-colorable, the $O(k^{t+O(1)}n)$ time algorithm for $k$-{\sc Coloring} on graphs of treewidth $t$~\cite{cyganBook}, and the $O(\log^{10} n)$ treewidth bound on $K_t$-free even-hole free graphs which follows from Theorem~\ref{main})

The list of problems in Theorem~\ref{mainApplications} for which Theorem~\ref{thm:alg_general} implies a quasi-polynomial-time algorithm is by no means exhaustive. For an example Theorem~\ref{thm:alg_general} also yields a quasi-polynomial time algorithm to recognize even-hole-free graphs (since we can encode in CMSO$_2$ that $G[S]$ is a cycle of even length). We refer the reader to \cite{TI2} and \cite{lima2024tree}.

\smallskip
\noindent
{\bf Comparison with the Algorithmic Consequences of~\cite{TWXV}:}
It is worth comparing the algorithmic consequences of Theorem~\ref{main} with those of the logarithmic treewidth bound for $K_t$-free graphs in~\cite{TWXV}. 
The logarithmic treewidth bound for $K_t$-free even-hole-free graphs in~\cite{TWXV} typically leads to {\em polynomial} time algorithms for problems on $K_t$-free even-hole-free graphs. With the exception of $k$-{\sc Coloring}, for which this leads to a polynomial time algorithm on even-hole-free graphs, for other problems the results on  $K_t$-free even-hole-free graphs only lead to polynomial (or quasi-polynomial) time approximation schemes on even-hole-free graphs.
On the other hand, Theorem~\ref{main} readily leads to quasi-polynomial time {\em exact} algorithms on even-hole-free graphs (and hence shows that the considered problems are unlikely to be \textsf{NP}-hard), but does not appear to give any meaningful polynomial time algorithms (neither exact nor approximation).


\smallskip
\noindent
{\bf Organization of the paper.} In Section~\ref{sec:prelim} we define the notations and basic definitions used in the paper. In Section~\ref{sec:outline} we give a brief outline of the proof of Theorem~\ref{main}. The remainder of the paper concerns the proof of Theorem~\ref{main}.

\section{Preliminaries}\label{sec:prelim}
All graphs in this paper are finite and simple and all logarithms are base $2$. We begin with some standard definitions (see, for example, \cite{TWXV}). 
Let $G = (V(G),E(G))$ be a graph. In this paper, we use induced subgraphs and their vertex sets interchangeably. For a graph $G$ and vertex set $S$, the graph $G \setminus S$ is the graph obtained from $G$ by deleting all vertices in $S$ and all edges incident to at least one vertex in $S$. The subgraph of $G$ {\em induced by} $S$ is denoted by $G[S]$ and defined as $G[S] = G \setminus (V(G) \setminus S)$. For graphs $G, H$ we say that $G$ {\em contains $H$} if $H$ is isomorphic to an induced subgraph of $G$. We say that $G$ is {\em $H$-free} if $G$ does not contain $H$. For a set $\mathcal{H}$ of graphs, $G$ is {\em $\mathcal{H}$-free} if $G$ is $H$-free for every $H \in \mathcal{H}$.



Let $v \in V(G)$. Let $X \subseteq V(G)$. We denote by $N_G(X)$ the set of all vertices in $V(G) \setminus X$ with at least one neighbor in $X$. We also define $N_G[X]=N_G(X) \cup X$.  When $X = \{v\}$, we write $N_G(v)$ for $N_G(\{v\})$ and $N_G[v]$ for $N_G[\{v\}]$. If there is no danger of confusion, we omit the subscript $G$. If $H$ is an induced subgraph of $G$, then $N_H(X)=N(X) \cap H$ and $N_H[X]=N_H(X) \cup X$. Let $Y \subseteq V(G)$ be disjoint from $X$. We say $X$ is \textit{complete} to $Y$ if every vertex in $X$ is adjacent to every vertex in $Y$ in $G$, and $X$ is \emph{anticomplete}
to $Y$ if there are no edges between $X$ and $Y$. 

    
    


A {\em path} in a graph is an {\em induced} subgraph that is a path. Given a path $P$ with ends $a,b$, the {\em interior} of $P$, denoted by $P^*$, is the set $P \setminus \{a,b\}$. 
The {\em length} of a path or a hole is the number of edges in it. A hole is \emph{even} if its length is even. A graph is {\em even-hole-free} it contains no even holes.




The {\em stability (or independence) number} $\alpha(G)$ of $G$ is the maximum size of a stable set in $G$. A related parameter, the {\em clique cover number} $\kappa(G)$ of $G$, is the smallest number of cliques whose union equals $V(G)$. 


Next, we define a slight generalization of even-hole-free graphs; we need the following definitions (see, for example, \cite{TWXV}). 
A \emph{theta} is a graph consisting of three internally vertex-disjoint paths $P_1 = a \dd \cdots \dd b$, $P_2 = a \dd \cdots \dd b$, and
$P_3 = a \dd \cdots \dd b$, each of length at least 2, such that $P_1^*, P_2^*, P_3^*$ are pairwise anticomplete. We call $a$ and $b$ the
{\em ends} of the theta and $P_1, P_2, P_3$ the \emph{paths} of the theta.
A {\em near-prism} is a graph consisting of two triangles $\{a_1,a_2,a_3\}$ and  $\{b_1, b_2, b_3\}$, and three paths $P_i$ from $a_i$ to $b_i$
for $1 \leq i \leq 3$, and such that  $P_i \cup P_j$ is a hole for every $i,j \in \{1,2,3\}$. It follows 
that $P_1^*, P_2^*, P_3^*$ are pairwise disjoint and anticomplete to each other,
$|\{a_1,a_2,a_3\} \cap \{b_1,b_2,b_3\}| \leq 1$, and if $|\{a_1,a_2,a_3\} \cap \{b_1,b_2,b_3\}| = 1$, then two of the paths have length at least 2. 
Moreover, the only edges between $P_i$ and $P_j$ are 
$a_ia_j$ and $b_ib_j$.
A {\em prism} is a near-prism whose triangles are disjoint. 
A \emph{wheel}  in $G$ is a pair $(H, x)$ where $H$ is a hole and $x$ is a vertex with at least three neighbors in $H$. The vertex $x$ is called the {\em center} of the wheel $(H,x)$. A wheel $(H,x)$ is {\em even} if $x$ has an even number of neighbors on $H$.
Let $\mathcal{C}$ be the class of ($C_4$, theta, prism, even wheel)-free graphs (these are sometimes called ``$C_4$-free odd-signable graphs''). Every even-hole-free graph belongs to $\mathcal{C}$.
For every integer $t \geq 1$, let $\mathcal{C}_t$ be the class of all graphs in $\mathcal{C}$ with no clique of size $t$. 

A wheel $(H,x)$ is {\em proper} if $\alpha(N(x) \cap H)$ is at least three. 
A {\em pyramid} is a graph consisting of a vertex $a$ and a triangle $\{b_1, b_2, b_3\}$, and three paths $P_i$ from $a$ to $b_i$ for $1 \leq i \leq 3$, such that  $P_i \cup P_j$ is a hole for every $i,j \in \{1,2,3\}$.
It follows that $P_1 \setminus a, P_2 \setminus a , P_3 \setminus a$ are pairwise disjoint, and the only edges between them are of the form $b_ib_j$. It also follows that at most one of $P_1, P_2, P_3$ has length exactly 1. We say that $a$ is the {\em apex} of the pyramid and that $b_1b_2b_3$ is the {\em base} of the pyramid.

\section{Proof Outline}\label{sec:outline}
We give here the main ideas of the proof. We will not concern ourselves with the exponent of $\log n$ in our bound and simply show that the tree independence number of $G$ is polylogarithmic in $n$. 
The high level scaffolding of the proof is quite similar to the proof of~\cite{TWXV} that $K_t$-free even-hole-free graphs have logarithmic treewidth. However most of the individual high level pieces of the proof differ substantially from the corresponding step in~\cite{TWXV} and require significant new ideas. 

The proof of~\cite{TWXV} that every $K_t$-free graph has treewidth $O(\log n)$ goes as follows: first it is proved that even-hole-free graphs admit ``dominated balanced separators''. In particular there exists a constant $c$ such that every even-hole-free graph contains a vertex set $S$ of size at most $c$ such that every connected component of $G-N[S]$ has at most $n/2$ vertices. This fact is then used to show that, in order to obtain the treewidth bound, it is sufficient to show that every pair of non-adjacent vertices $a$, $b$ in a $K_t$-free even-hole free graph can be separated from each other by a set of size at most $c_t \log n$ for a constant $c_t$ depending only on $t$. The most technical part of the argument is then to show the existence of such small $a$-$b$ separators.

The first step of the proof of~\cite{TWXV}, that even-hole-free graphs admit dominated balanced separators, does not use the assumption that $G$ is $K_t$-free and hence we can use it here. 
Using the dominated balanced separator bound we then show an analog of the second step of~\cite{TWXV} tailored to tree independence number rather than treewidth: we show that in order to obtain the tree independence number bound, it is sufficient to prove that every pair of non-adjacent vertices $a$, $b$ in an even-hole-free graph can be separated from each other by a set with independence number $O(\log n^{O(1)})$.
Here we cannot re-use the proof from~\cite{TWXV} because that proof crucially uses the assumption that $G$ is $K_t$-free. However we can directly apply a similar step from~\cite{TI2}. Thus, ``all'' that remains is to prove the following statement:

\begin{theorem} \label{banana}
  There exists a constant $c$ with the following property. Let $G$ be an even-hole-free graph  with $|V(G)|=n$, and let $a,b \in V(G)$ be non-adjacent. Then there is a set $X \subseteq V(G) \setminus \{a,b\}$ with $\kappa(X) \leq  c \log^8 n$ and such that every component of $G \setminus X$ contains at most one of $a,b$.
\end{theorem}

In Theorem~\ref{banana}, the clique cover number $\kappa$ of the separator $X$ is bounded rather than the independence number $\alpha$ because that is what naturally comes out of the proof. Obtaining a separator $X$ with polylogarithmic independence number would also have been sufficient. Let us now not worry too much about the exponent $8$ in the statement of Theorem~\ref{banana}, and simply aim for a polylogarithmic bound on $\kappa$.

Towards this goal we employ a recent tool of Korchemna et al.~\cite{KLSSX}, who provide a ``max flow-min cut'' like theorem for clique separators. Applying this theorem to even-hole-free graphs (and using that even-hole-free graphs are $K_{2,2}$-free and therefore have only polynomially many maximal cliques~\cite{Farber}) we get one of two outcomes. Either we get an $a$-$b$ separator $X$ with a polylogaritmic upper bound on $\kappa(X)$, this is the desired outcome, or we get a set of $f$ paths $P_1, \ldots, P_f$ from $a$ to $b$ such that no clique of $G-\{a,b\}$ intersects more than $O(\log n)$ of the paths. Here $f$ can be chosen to be an arbitrarily large polylogarithmic function of $n$ (and the larger we choose $f$, the worse bound we get on $\kappa(X)$). 

Observe that the graph $G'$ induced by all of these $a$-$b$ paths $P_1, \ldots, P_f$ is $K_t$-free, where $t=O(\log n)$. Indeed, every vertex of a clique $K$ in $G'$ must be on some $a$-$b$ path (by the definition of $G'$), and no path can contain more than $2$ vertices of $K$ (since the path is induced), so $K$ must intersect at least $|K|/2$ paths. It follows that $G'$ is $K_t$-free where $t=O(\log n)$.

If only the constant $c_t$ in the $c_t \log n$ bound on the size of an $a$-$b$ separator from~\cite{TWXV} depended polynomially on $t$ we would be done! Indeed, we could then have chosen $f$ to be so large that $f \geq (c_t \cdot \log n) \cdot \log^2 n$, and get an $a$-$b$ separator $X$ in $G'$ of size at most $c_t \cdot \log n$. But then some vertex of $X$ intersects at least $\log^2 n$ of the paths, contradicting that no clique intersects more than $O(\log n)$ of them. 

Unfortunately the constant $c_t$ in the proof of~\cite{TWXV} does not depend polynomially on $t$, and it does not look like an easy task to improve the dependence to a polynomial in $t$. However, the argument above works out just fine even if $c_t$ is not a constant but rather a polynomial in $t$ and $\log n$. Hence, in order to prove Theorem~\ref{banana} it is sufficient to prove that, for every $K_t$-free graph, every pair $a, b$ of vertices can be separated by a set $X$ of size polynomial in $t$ and $\log n$. This is precisely the approach that we take, this is the most involved part of our arguments. 

\paragraph{Separating Two Vertices}
We now sketch the main ideas of the proof of Theorem~\ref{bananaclique}: that for every $K_t$-free graph every pair $a$ and $b$ of vertices can be separated by a set $X$ of size polynomial in $t$ and $\log n$. We start by sketching how we would have liked the proof to work, point to where this approach breaks, and then outline how the proof actually works. 

We wish to separate $a$ from $b$. We will concentrate on the neighborhood $N(a)$ of $a$. Let $D$ be the component of $G-N[a]$ that contains $b$, we can safely ignore the vertices in $N(a) \setminus N(D)$ and focus on the vertices in $N(D)$. As long as $N(D)$ contains a clique that covers at least $0.1\%$ of $N(D)$ we can add this clique to $X$ since we only can do this step $O(\log n)$ many times. After this step, ``many'' (i.e., at least $99\%$) of vertex triples $x_1$, $x_2$, $x_3$ in $N(a) \cap N(D)$ are stable (this follows from, e.g.,~\cite{bisimplicialnew}). 

For a stable triple $x_1, x_2, x_3$ let $D'$ be an inclusion minimal connected subset of $D$ that contains neighbors of $x_1$, $x_2$ and $x_3$, and let $H = \{a,x_1,x_2,x_3\} \cup D'$. A simple case analysis shows that $\{a,x_1,x_2,x_3\} \cup D'$ is either a pyramid or a wheel. 
%
We show in Section~\ref{pyramids} that if $H$ is a pyramid, then either there is a clique $K$ in $D$ that separates at least two vertices of $\{x_1, x_2, x_3\}$ from $b$ in $(\{x_1, x_2, x_3\} \cup D)$, or there is a clique $K$ in $D$ that has an ``almost as good'' separation effect (the precise formulation of this ``almost as good'' effect is cumbersome, and we skip it here). Additionally there are two exceptional cases for which we are not able to obtain this outcome: $H$ could grow to a {\em loaded pyramid} or an {\em extended near-prism} (see Section~\ref{pyramids} for definitions). 



Suppose that at least $1\%$ of the stable triples $x_1, x_2, x_3$ in $N(a) \cap N(D)$ there is a clique $K_{x_1,x_2,x_3}$ that separates at least two of them from $b$ in $(\{x_1, x_2, x_3\} \cup D)$ (or does the morally equivalent job). 
One of the main structural insights in this paper is that in this case we can conclude that there is a single set $K$ in $D$ such that $\kappa(K)$ is constant and no component of $D - K$ sees more than $99\%$ of $N(D)$. 
This kind of local-to-global transition is usually very hard to force when one works with families of graphs defined by forbidden induced subgraphs. Our arguments here only rely on $G$ being $C_4$-free, so we expect for this technique to be applicable in other contexts in the future. 

Whenever we get a $K$ as above we win -- we can just add it to our separator $X$, and again we will only do this $O(\log n)$ many times before $|N(D)|$ drops to $0$ and $a$ is separated from $b$.  
A similar outcome can be derived using structural arguments from~\cite{bisimplicialnew} when $H$ is a pyramid that grows to an extended near-prism for a sufficiently large proportion of stable triples in $N(D)$.

The problem with this approach is that it gets stuck whenever $99\%$ of the stable triples $\{x_1, x_2, x_3\}$ satisfy that $H$ is a wheel or grows to a loaded pyramid.
%
When this problem occurs a large fraction of the vertices of $N(D)$ are {\em hubs}. We say that a vertex $v$ of $G$ is a {\em hub} if $v$ is the center of a proper wheel, or the ``corner'' of a loaded pyramid (again,  see Section~\ref{pyramids}) and we denote by $\Hub(G)$ the set of all hubs of $G$. We remark that our definition of hubs is not precisely the same as the definition of hubs in~\cite{TWXV}, although the role hubs play in~\cite{TWXV} and here are similar. 
When {\em none} of the vertices in $N(a)$ (and therefore $N(D)$) are hubs an argument quite similar to the one outlined above works, and we are able to obtain an $a$-$b$ separator $X$ with polylogarithmic $\kappa$. We call this the ``hub-free'' case. 

The remainder of the proof then consists of reducing the general case to the hub-free case. The reduction is based on the ``central bag'' method, developed in \cite{TWI} and \cite{TWIII} and also used in \cite{TWXV}. 
%
Since $G$ is $C_4$-free it follows that $N(a) \cap N(b)$ is a clique, and we may add  $N(a) \cap N(b)$ to $X$ at the cost of increasing $|X|$ by $t$ (recall that $G$ is $K_t$-free). From now on we assume that $N(a) \cap N(b)$ is empty. 

Since every even-hole-free graph has a vertex whose neighborhood is the union of two cliques, it follows that every induced subgraph of $G$ has average degree upper bounded by $O(t)$.
Thus, the set $\Hub(G)$ contains a stable set $S_1$ of size at least $\Omega(\frac{|\Hub(G)|}{t})$ such that the degree of each vertex in $S_1$ in $G[\Hub(G)]$ is at most $O(t)$. 
We show that for every hub $v$ in a graph in ${\cal C}$, $G - N[v]$ is disconnected. When $v$ is a wheel center this was already known, for loaded pyramid corners we prove it in Section~\ref{pyramidcutsets} (we skip this proof in the overview). 

%

For each $v \in S_1$, since $v \notin N(a) \cap N(b)$ there is a component $D_v$ of $G \setminus N[v]$ that contains at least one of $\{a,b\}$.
We claim that $N[D_v]$ must contain both $\{a, b\}$ or we are already done! Indeed, suppose  $N[D_v]$ contains $b$ but not $a$. Then a set $X$ that separates $v$ from $b$ also separates $a$ from $b$, but $v$ only has $O(t)$ hubs in its neighborhood. So adding these $t$ hubs to $X$ leaves us with the task of separating $v$ from $b$, but now $v$ has no hubs in its neighborhood and we are in the hub-free case (and therefore done). 

Thus, if we are not done yet, then for every vertex $v$ in $S_1$ there is a component $D_v$ such that $\{a,b\} \subseteq N[D_v]$. We define the {\em central bag} to be $\beta = \cap_{v \in S_1} (N[D_v] \cup \{v\})$ (the actual definition of the central bag is subtly different). Observe that both $a$ and $b$ are in the central bag $\beta$.

There are now two key observations behind the central bag method. The first (and easy) one is that no vertex $v \in S_1$ can be a hub in $\beta$. 
Indeed $v$ cannot be a hub in $N[D_v]$ since $N[D_v] - N(v) = D_v$ is connected, contradicting that the neighborhood of every hub is a separator. Since $\beta \subseteq D_v$ it follows that $v$ cannot be a hub in $\beta$ either. When the more nuanced definition is used, the proof is slightly more involved.

This observation means that the number of hubs in $\beta$ is smaller by a linear fraction than the number of hubs in $G$. Thus, by induction on $\log(|\Hub(G)|)$, we can find an $a$-$b$ separator $Y$ in the central bag $\beta$ of polylogarithmic size. If we can grow $Y$ to an $a$-$b$ separator $X$ in $G$ incurring an {\em additive} polylogarithmic cost, then the induction goes through and we are able to upper bound the total size of $X$ by $(\log n)^{O(1)}$. In particular the depth of the induction is logaritmic in $n$, so if $|X|$ grows by an additive term of $(\log n)^{O(1)}$ in each inductive step this is ok, but if $X$ grows by a factor $1.01 $ in each step then $|X|$ ends up being polynomial in $n$. 

The second (and more involved) component of the central bag method is to show that $Y$ can indeed be grown to $X$ as prescribed above. 
This incurs a cost which is proportional to (essentially) $|Y \cap S_1|$, because for each vertex $s$ in $|Y \cap S_1|$ we add all of its $O(t)$ hub neighbors to $X$ and then separate $s$ from either $a$ or $b$ using the hub-free case.

Unfortunately the inductive step which gives us $Y$ does not give us any guarantees on the size of  $|Y \cap S_1|$, so in the worst case  $Y \cap S_1$ could be almost as big as $Y$ itself. Then the cost of turning $Y$ into a separator $X$ in $G$ would incur at least a constant {\em multiplicative} cost, which would be too expensive. We therefore apply an additional ``pivot'' step where $Y$ is changed so that $|Y \cap S_1|$ is small. This pivot step again relies on the hub-free case as well as a second application of the local-global transition mentioned above. This concludes the proof outline. 

We note that most of the  results are proved for the slightly more general class of graphs $\mathcal{C}$, rather than even-hole-free graphs. However, Section~\ref{strips} deals with even-hole-free graphs only. It is very likely that the proofs there do in fact work
in the more general setting, but we did not verify the details. Theorem~\ref{bisimplicialnew} is another fact we need that has only been proved for even-hole-free graphs (and not for $C_4$-free odd-signable graphs); once again it is likely to generalize, but we  have not checked it carefully. These are the two reasons for the fact that our main theorem applies to even-hole-free graphs only.

\subsection{Organization of the Proof}
The pieces of the proof appear in a different order than in the outline. Specifically each piece is proved before it is used. 
%
%
In Section~\ref{pyramids} we prove Theorem~\ref{pyramid_separate}, which allows us to generate clique separators from pyramids that do not grow to an extended near-prisms.
In Sections \ref{wheelcutsets} and \ref{pyramidcutsets} we define hubs and prove Theorem~\ref{hubstarcutset}, that every for every hub $v$, $N[v]$ separates the graph (in a particular way).
In Section~\ref{strips} we prove Theorem~\ref{treestructplus} which allows us to decompose graphs that contain pyramids that {\em do} grow to an extended near-prisms.
In Section~\ref{sec:localglobal} we prove Theorem~\ref{localglobal}: that if we have sufficiently many clique separators of the type that are the output of Theorem~\ref{pyramid_separate}, then there is a bounded $\kappa$ size set whose removal substantially separates the graph. 
%
In Section~\ref{sec:dangerous} we prove Theorem~\ref{dangerous} which shows how to deal with the case where we have a vertex $a$, and many of the stable triples $x_1, x_2, x_3$ in the neighborhood of the component $D$ of $G-N(a)$ that contains $b$ are contained in a pyramid that grows to an extended near-prism. 
%
In Section~\ref{sec:sepab_nohubs} we handle the hub-free case and prove Theorem~\ref{ablogn}, that if $a$ does not have any hub neighbors, then $a$ and $b$ can be separated with $O(\log n)$ cliques. 
%
Section~\ref{sec:centralbag_banana} contains all of the elements needed for the central bag method, with the exception of the ``pivot'' step where the separator $Y$ is changed to (mostly) avoid $S_1$.
In Section~\ref{sec:boundhubs} we prove Theorem~\ref{boundhubs}, which does the aforementioned pivot step. 
In Section~\ref{sec:bananaclique} we apply the central bag method, putting together the results from Sections~\ref{sec:sepab_nohubs},~\ref{sec:centralbag_banana} and~\ref{sec:boundhubs} to prove Theorem~\ref{bananaclique}, that every pair of non-adjacent vertices in a $K_t$-free even-hole-free graph can be separated by $(t \log n)^{O(1)}$ cliques. 
In Section~\ref{sec:banana} we prove Theorem~\ref{sec:banana}, that every pair of non-adjacent vertices in an even-hole-free graph can be separated by $(\log n)^{O(1)}$ cliques. 
Finally, in Section~\ref{sec:domsep} we use Theorem~\ref{sec:banana} to prove Theorem~\ref{main}.

\section{Jumps on pyramids}
\label{pyramids}


The goal of this section is to prove Theorem~\ref{pyramid_separate},
which asserts the existence of well-structured cutsets that separate the neighbors of the apex of a pyramid.
Theorem~\ref{pyramid_separate} is then used to produce the ``local cutsets'' in
Theorem~\ref{localglobal}.

We start with some definitions.
Let $G \in \mathcal{C}$ and let
 $\Sigma$ be a pyramid in $G$  with apex $a$, base $b_1b_2b_3$ and
paths $P_1,P_2,P_3$.
We say that $X \subseteq \Sigma$ is {\em local (in $\Sigma$)} if $X \subseteq P_i$  for some $i \in \{1,2,3\}$, or $X \subseteq \{b_1,b_2,b_3\}$. 
Let $P=p_1 \dd\cdots \dd p_k$ be a path with $P \cap \Sigma = \emptyset$. $P$ is a {\em corner path for $b_1$} if $p_1$ is adjacent to $b_2,b_3$,
$p_k$ has a neighbor in $P_1 \setminus b_1$,
and there are no other edges from $\Sigma \setminus b_1$ to $P$.
A corner path for $b_2$ and $b_3$ is defined similarly.
We say that $P$ is a {\em corner path for $\Sigma$} if $P$ is a corner path
for $b_1$, $b_2$ or $b_3$. If $v \in G \setminus \Sigma$, and $v$ is not a corner path for $\Sigma$, and $N_{\Sigma}(v)$ is not local, we say that $v$ is {\em major (for $\Sigma$)}.

A \textit{loaded pyramid} in a graph $G$ is a pair $\Pi=(\Sigma,P)$ where $\Sigma$ is pyramid with apex $a$, base
$b_1b_2b_3$ and paths $P_1, P_2, P_3$, $a$ is adjacent to $b_2$ (hence $|P_2|=2$), and $P=p_1 \dd \cdots \dd p_k$ is a path such that
\begin{itemize}
\item $p_1$ is adjacent to $b_2$;
\item $p_k$ has a neighbor in $P_1^*$;
\item $P_3$ is anticomplete to $P$ (and in particular $a$ is anticomplete to $P$);
  \item $b_2$ is anticomplete to $P \setminus p_1$; and
  \item $P_1 \setminus b_1$ is anticomplete to $P \setminus p_k$.
    \end{itemize}
In this case, we say $b_2$
is a \textit{loaded pyramid corner}, and denote the loaded pyramid
as a pair $(\Pi, b_2)$. We also use the notation $\Pi$ to denote the vertex set $P\cup \Sigma$.

Recall that a  \emph{wheel} $(H, x)$ in $G$ is a pair where $H$ is a hole and $x$ is a vertex with at least three neighbors in $H$; the vertex $x$ is called the \emph{center} of the wheel.
A wheel $(H,x)$ is {\em proper} if $\alpha(N_H(x)) \geq 3$.
We say that a vertex $v$ of $G$ is a {\em hub} if $v$ is
a proper wheel center or a loaded  pyramid corner, and we 
denote by $\Hub(G)$ the set of all hubs of $G$.

An {\em extended  near-prism}, defined in \cite{bisimplicialnew},
is a graph obtained from a  near-prism by adding one extra edge, as follows. 
Let $P_1,P_2,P_3$ be as in the definition
of a near-prism, and let $a\in P_1^*$ and $b\in P_2^*$; and add an edge $ab$. (It is important that $a,b$ do not belong 
to the triangles.) If the two triangles of the extended near-prism are
disjoint, we also call it an {\em extended prism}.
We call $ab$ the {\em cross-edge} of the extended near-prism (or of the extended prism). 
We start with  two lemmas:

\begin{theorem}
  \label{major}
Let  $G \in \mathcal{C}$.
  Let $\Sigma$ be a pyramid in $G$ with  apex $a$, base $b_1b_2b_3$ and
  paths $P_1,P_2,P_3$. Let $p$ be a major vertex for $\Sigma$. Then
  one of the following holds: 
  \begin{enumerate}
    \item $p$ is adjacent to $a$ and at least two of the neighbors of $a$ in
      $\Sigma$;
    \item $p$ is adjacent to $a$, and $p$ is a hub;
    \item $\Sigma \cup p$ is an extended prism whose cross-edge contains $a$; or
\item $(\Sigma, p)$ is a loaded pyramid, and so one of $b_1,b_2,b_3$ is a loaded pyramid corner.
  \end{enumerate}
  \end{theorem}

\begin{proof}
   We may assume that the last three outcomes of Theorem \ref{major}
  do not hold.
First we prove that $p$ is adjacent to $a$.  Suppose not.
  Assume first that  $p$ has a neighbor in each of
  $P_1,P_2,P_3$. Since
  $p$ is not a corner path for $\Sigma$ and $N_\Sigma(p)$ is not local, it follows that $p$ has neighbors in the interiors of at least two of $P_1, P_2, P_3$. Consequently, there is a theta
  with ends $a$ and $p$ whose paths are subpaths of $P_1,P_2,P_3$,
  a contradiction.
 Thus we may assume that  $p$ is anticomplete
 to at least one of $P_1, P_2,P_3$,
 say $p$ is anticomplete to $P_3$.
 \\
 \\
  \sta{If $p$ is non-adjacent to $a$, then for $i=1,2$, we have $N_{P_i}(p) \neq \{b_i\}$. \label{notbi}}

Suppose $N_{P_1}(p)=\{b_1\}$. Since $p$ is major, $p$ has a neighbor in $P_2 \setminus b_2$.
Since $p$ is non-adjacent to $a$, and $\Sigma \cup p$ is not a loaded pyramid, it follows that
 $b_1$ is non-adjacent to $a$. But now we get a theta with ends $b_1, a$ and paths
 $b_1 \dd P_1 \dd a$, $b_1 \dd p \dd P_2 \dd a$ and $b_1 \dd b_3 \dd P_3 \dd a$, a contradiction.
This proves \eqref{notbi}.
\\
\\
By \eqref{notbi} and since $p$ is major and we have assumed that $p$ is non-adjacent to $a$, $p$ has both a neighbor in $P_1 \setminus b_1$,
and a neighbor in $P_2 \setminus b_2$.
Let $H$ be the hole formed by $P_1$ and $P_2$.
\\
\\
 \sta{If $p$ is non-adjacent to $a$, then $p$ has exactly three neighbors in $H$, and two of them are
 consecutive. \label{3nbrs}}

Suppose that $p$ has two non-adjacent neighbors in
$P_1$. Then there exists a path $P_1'$ from $p$ to $a$, and a path $P_1''$ from
$p$ to $b_1$, both with interior in $P_1$ and such that
$P_1' \setminus p$ is anticomplete to $P_1'' \setminus p$.
Now we get a theta with ends $p,a$ and paths $p \dd P_1' \dd a$,
$p \dd P_1'' \dd b_1 \dd b_3 \dd P_3 \dd a$ and a path from $p$ to
$a$ with interior in $P_2 \setminus b_2$, a contradiction. This proves that
$p$ has either one or two consecutive neighbors in
$P_1$, and the same for $P_2$. Since $(H,p)$ is not an even wheel, and $H \cup p$ is not
a theta, \eqref{3nbrs} follows.
\\
\\
  We may assume that 
 $N_{P_1}(p)=\{t\}$,  $N_{P_2}(p)=\{q,r\}$, where
  $P_2$ traverses $a,q,r,b_2$ is this order. By \eqref{notbi}, $t \neq b_1$.
  It follows from \eqref{3nbrs} that $q$ is adjacent to $r$.
 Since $\Sigma \cup p$ is not an extended prism, it follows that  $t$ is non-adjacent to
 $a$. But now  there is a theta with ends
 $t,a$ and paths $t \dd P_1 \dd a$, $t \dd p \dd q \dd P_2 \dd a$ and 
 $t \dd P_1 \dd b_1 \dd b_3 \dd P_3 \dd a$, a contradiction.
 This proves that $p$ is adjacent to $a$.

To complete to proof of \ref{major}, assume that $p$ is anticomplete to $N_{P_1 \cup P_2}(a)$.
Since $p$ is major, $p$ has a neighbor in $H \setminus a$. Since
$p \not \in \Hub(G)$ and $H \cup p$ is not a theta, it follows that
$p$ has exactly two neighbors in $H \setminus a$, and they are adjacent. Since $p$ is not a corner path
for $b_3$, it follows that $N_H(p) \neq \{a,b_1,b_2\}$, and so we may assume that $N_H(p) \subseteq P_1$. Since $p$ is major, we deduce that $p$ has a neighbor
in $P_3 \setminus a$. Let $H'$ be the hole $P_1 \cup P_3$.
Since $p  \not \in \Hub(G)$, it follows that
$(H',p)$ is not a proper wheel. Consequently, $N_{P_3}(p)=\{a\} \cup N_{P_3}(a)$.
But now $(H',p)$ is an even wheel, a contradiction.
 \end{proof}

\begin{theorem}\label{pyramid2}
Let $G \in \mathcal{C}$.
  Let $\Sigma$ be a pyramid in $G$ with apex $a$,
    base $b_1b_2b_3$ and
  paths $P_1,P_2,P_3$. Assume that $N_G(a) \subseteq \Sigma$.
    Let $P$ be a path in $G \setminus \Sigma$.
    Then one of the following holds.
    
\begin{enumerate}
\item $N_{\Sigma}(P)$ is local in $\Sigma$;
\item $P$ contains a major vertex for $\Sigma$;
\item $P$ contains a  corner path for $\Sigma$;
\item There exist distinct $i,j \in \{1,2,3\}$ and a subpath $Q=q_1 \dd \cdots \dd q_m$ of $P$ such that
  \begin{itemize}
  \item $N_{\Sigma}(q_1) \subseteq P_i$;
    \item  $q_1$ has a unique neighbor in $P_i$ and  $N_{P_i}(q_1)=N_{P_i}(a)$;
              \item $N_{\Sigma}(q_m) \subseteq P_j$; 
    \item $q_m$ has exactly two neighbors $x,y$ in $P_j$; $x$ is adjacent to  $y$, and $a \not \in \{x,y\}$; and
    \item there are no other edges between $\Sigma$ and $Q$;
      \end{itemize}
    In particular, $a$ is contained in the cross-edge of an extended prism.
\item There is an $i \in \{1,2, 3\}$ and a subpath $Q=q_1 \dd \cdots \dd q_m$ of $P$ such that $(\Sigma, P)$ is a
  loaded pyramid with loaded pyramid corner $b_i$.
  \end{enumerate}
\end{theorem}

\begin{proof}
  Let $P=p_1 \dd \cdots \dd p_k$. Suppose for a contradiction that none of the outcomes hold. We may assume that $N_{\Sigma}(P)$ is not local, and that $N_{\Sigma}(X)$ is local for every proper subpath $X$ of $P$. 
  Since no vertex of $P$ is major, and no subpath  of $P$ is  a corner path,
  it follows that $k \geq 2$. 

  Since $N_{\Sigma}(P) \not \subseteq \{b_1,b_2,b_3\}$, by symmetry, we may assume that
  \begin{itemize}
      \item 
$N_{\Sigma}(p_1)  \subseteq P_1$ and $p_1$ has a neighbor in $P_1 \setminus b_1$; 
    \item $p_k$ has a neighbor in $P_2 \setminus a$, and either $N_{\Sigma}(p_k) \subseteq P_2$, or $N_{\Sigma}(p_k) \subseteq \{b_1,b_2,b_3\}$; and
    \item $\Sigma \setminus b_1$ is anticomplete to $P^*$. 
  \end{itemize}
   We first show: 
   \\
   \\
     \sta{We have $N_{\Sigma}(p_k) \subseteq P_2 \cup \{b_1\}$. \label{3.5}}
 
 Suppose  not. Then $N_{\Sigma}(p_k) \subseteq \{b_1,b_2,b_3\}$. Since
 $p_k$ has a neighbor in $P_2 \setminus a$, it follows that 
 $p_k$ is adjacent to $b_2$. Since $P \cup (P_1 \setminus b_1)$ contains a path from
 $p_k$ to $a$, and since $P$ does not contain a corner  path for $b_1$, it
 follows that   $p_k$ is
 non-adjacent to $b_3$. This proves \eqref{3.5}.
 \\
 \\
 \sta{$N_{P_2}(p_k)  \neq \{b_2\}$.
      \label{onlyb2}}

 Suppose that $N_{P_2}(p_k) = \{b_2\}$.
By \eqref{3.5} $p_k$ is non-adjacent to $b_3$.
    If $b_2$ is non-adjacent to $a$, then there is a theta in $G$ with ends $b_2, a$ and paths $P_2$, $b_2$-$b_3$-$P_3$-$a$ and $b_2$-$p_k$-$P$-$p_1$-$P_1$-$a$, a contradiction; so $b_2$ is adjacent to $a$. Now since $a$ has no neighbor in $P$,  $(\Sigma, P)$ is a loaded pyramid with base $b_1b_2b_3$, apex $a$, and paths $P_1, P_2, P_3$, and the fifth outcome holds, a
    contradiction.
        This proves \eqref{onlyb2}.
     \\
     \\
     \sta{There is no edge from $b_1$ to $P \setminus p_1$.\label{notb1nbrs}}

     Suppose for a contradiction that $b_1$ has a neighbor in $P \setminus p_1$. Since $N_{\Sigma}(P \setminus p_1)$ is local, it follows from \eqref{3.5} that
     $N_{\Sigma}(p_k) \subseteq \{b_1, b_2\}$, contrary to \eqref{onlyb2}, 
          and \eqref{notb1nbrs} follows. 
 \\
 \\
 By \eqref{notb1nbrs}, it follows that $\Sigma$ is anticomplete to $P^*$ and that
 $N_\Sigma(p_k) \subseteq P_2$. Traversing $P_1$ from $b_1$ to $a$, let $x_1$ be the first neighbor of $p_1$. Traversing $P_2$ from $b_2$ to $a$, let $x_2$ be the first neighbor of $p_k$. Then $H = p_1$-$P$-$p_k$-$x_2$-$P_2$-$b_2$-$b_1$-$P_1$-$x_1$-$p_1$ is a hole in $G$ (since it contains at least the four distinct vertices $b_1, b_2, p_1, p_k$). For $i = 1, 2$, let $z_i$ be the neighbor of $x_i$ in $P$ (and thus $z_i \in \{p_1, p_k\}$).
\\
\\
\sta{For $i = 1,2$, either there is a vertex $y_i$ in $P_i$ with $x_i$ adjacent to $y_i$ and $N(z_i) \cap P_i = \{x_i, y_i\}$;
    or $x_i$ is the only neighbor of $z_i$ in $P_i$ and $x_i$ is adjacent to $a$. \label{nbrsofa}}

Since $p_1$ has a neighbor in $P_1 \setminus b_1$,
and since by \eqref{3.5} and \eqref{onlyb2} $p_k$ has a neighbor in
$P_2 \setminus b_2$, it follows that
for $i = 1, 2$, there is a path from every vertex of $P$ to $a$ with interior in $(P \cup P_i) \setminus b_i$.

If $x_i$ is the only neighbor of $z_i$ in $P_i$, and $x_i$ is non-adjacent to $a$, then we find a theta with ends $x_i, a$ in $G$ and paths $x_i$-$P_i$-$a$, $x_i$-$P_i$-$b_i$-$b_3$-$P_3$-$a$, and a path whose interior is contained in $(P \cup P_{3-i}) \setminus b_{3-i}$ given by the claim of the previous paragraph applied to $z_{3-i}$,  a contradiction. 

Thus we may assume that $z_i$ has two non-adjacent neighbors in $P_i$. Let $y_i$ be the neighbor of $z_i$ along $P_i$ closest to $a$. Since $y_i \neq a$,
there is a theta with ends $z_i, a$ and paths $z_i$-$y_i$-$P_i$-$a$, $z_i$-$x_i$-$b_i$-$b_3$-$P_3$-$a$, and a path whose interior is contained in $(P \cup P_{3-i}) \setminus b_{3-i}$ given by the claim of the first  paragraph applied to $z_{3-i}$,  a contradiction. 
Now \eqref{nbrsofa} follows. 
\\
\\
If the first outcome of \eqref{nbrsofa} holds for both $i=1$ and $i=2$, we get a prism with triangles  $x_1y_1z_1$ and
$x_2y_2z_2$ and paths $y_1 \dd P_1 \dd a \dd P_2 \dd y_2$, $P$ and
$x_1 \dd P_1 \dd b_1 \dd b_2 \dd P_2 \dd x_2$, a contradiction.
If the second outcome of \eqref{nbrsofa} holds for both $i=1$ and $i=2$, then
we get a theta with ends $x_1,x_2$ and 
paths $x_1 \dd P_1 \dd a \dd P_2 \dd x_2$, $P$ and
$x_1 \dd P_1 \dd b_1 \dd b_2 \dd P_2 \dd x_2$, a contradiction.
Thus we may assume that the first outcome holds for $i=1$ and the third outcome holds for $i=2$. But now the fourth outcome of  Theorem \ref{pyramid2}  holds, a contradiction.
\end{proof}

We need an additional definition: given a graph $G$ $x, y, z \in V(G)$, a path $P$ from $x$ to $y$, and a non-empty set $A \subseteq P$, we define the \emph{$(P, y)$-last vertex of $A$} to be the vertex of $A$ which is closest to $y$ along $P$.

We now prove the main result of this section.

\begin{theorem}\label{pyramid_separate}
Let $G \in \mathcal{C}$.
  Let $\Sigma$ be a pyramid in $G$ with apex $a$, base $b_1b_2b_3$ and
  paths $P_1,P_2,P_3$. 
   For each $i$, let $Q_i$ be a subpath of $P_i$, with ends $a$ and $x_i$, such that 
    all internal vertices of $Q_i$ have degree two in $G$.
    Let $b \in V(G) \setminus (Q_1 \cup Q_2 \cup Q_3)$.
  Assume that
  \begin{itemize}
\item $N_G(a)=N_{\Sigma}(a)$;
\item $\{x_1,x_2,x_3\}$ is a stable set;
    \item $b$ is non-adjacent to $x_1, x_2, x_3$;
    \item $a$ does not belong to a cross-edge of  an extended near-prism in $G$; and
    \item $\Hub(G) \cap \{x_1,x_2,x_3\}=\emptyset$.
  \end{itemize}
   Let  $D=G \setminus (Q_1 \cup Q_2 \cup Q_3)$.
  Then there is a clique $K \subseteq D$  and   distinct  $i,j  \in \{1,2,3\}$ such that one of the following holds:
  (For $i \in \{1,2,3\}$ let
   $D_i$ be the  union of components of $D \setminus K$ such that
  $N(x_i) \cap D_i \neq \emptyset$.)
 
  \begin{enumerate}
  \item $b \not \in K \cup D_i \cup D_j$; or
\item   We have
  $b \not \in N[D_i]$. Moreover, there is a set $D_j' = D_j'(x_1x_2x_3)$ of vertices such that:
  \begin{itemize}
      \item $D_j'$ is not a clique; 
      \item $D_j'$ is complete
  to $K$; 
      \item There is a vertex $q= q(x_1x_2x_3)$ with the following property. Either $b \in K$ and $q = b$; or there exists $k \in \{1, 2, 3\} \setminus \{i\}$ such that $x_k$ is complete to $K \cup D_j'$ and $q = x_k$; 
      \item For every $v \in D_j'$, there is a path $P$ in $D \cup \{x_j\} $ from $b$ to $x_j$ such that $P \cap N(q) \neq \emptyset$ and the $(P, x_j)$-last vertex in $P \cap N(q)$ is $v$;
      and
      \item For every path $P$ in $D \cup \{x_j\} $ from $b$ to $x_j$, we have $P \cap N(q) \neq \emptyset$, and the $(P, x_j)$-last vertex of $P \cap N(q)$ is complete to $K$. 
\end{itemize}
\end{enumerate}

  \end{theorem}

\begin{proof}
  Let $B_1,C_1,B_2,C_2,B_3,C_3$ be pairwise disjoint subsets of $G \setminus a$ with the following properties:
  \begin{itemize}
  \item the sets $B_1,B_2,B_3$ are all pairwise complete to each other;
  \item the sets $C_1,C_2,C_3$ are pairwise anticomplete to each other;
  \item for distinct $i, j \in \{1, 2, 3\}$, the set $B_i$ is anticomplete to $C_j$: 
  \item for every $i$,  every vertex of $B_i$ is an end of  a path 
    to $x_i$ with interior in $C_i$;
  \item for every $i$, one of the following holds: 
  \begin{itemize}
      \item $x_i \in C_i$ and every vertex of $C_i \setminus Q_i$ is in the interior of a path 
    from some vertex of $B_i$ to $x_i$; or
    \item $x_i = b_i$, and $B_i = \{x_i\}$, and $C_i = Q_i^*$. 
  \end{itemize}
  \item For every $i$, we have $b_i \in B_i$ and $P_i \setminus \{b_i, a\} \subseteq C_i$.
  \end{itemize}
  Subject to these properties, we choose the sets with 
  $W=\{a\} \cup \bigcup_{i=1}^3 (B_i \cup C_i)$ maximal.
  \\
  \\
  \sta{Let $D$ be a component of $G \setminus W$.
    Then either $N(D) \subseteq B_1 \cup B_2 \cup B_3$, or there exists $i$
    such that $N(D) \subseteq B_i \cup C_i$. \label{localcomps}}

  Suppose not and let $D$ be a component of $G \setminus W$ violating the statement. Then there exist $i \in \{1, 2, 3\}$ and a path
  $P = p_1 \dd \cdots \dd p_k$ in $D$ such that $p_1$ has a neighbor in
  $C_i$ and $p_k$ has a neighbor in $W \setminus (B_i \cup C_i \cup \{a\})$. We may assume that $P$ is chosen with $k$ as small as possible
  and that $i=1$. Let $c_1 \in C_1$ be a neighbor of $p_1$.
  Let $P_1'$ be  a path from $b_1' \in B_1$ to $a$ with interior in
  $C_1$ and such that $c_1 \in P_1'$.
  Let $c_2 \in B_2 \cup C_2 \cup B_3 \cup C_3$ be a neighbor
  of $p_k$; choose $c_2 \not \in B_2 \cup B_3$ if possible.
  We may assume that $c_2 \in B_2 \cup C_2$.
  Let $P_2'$ be  a path from $b_2' \in B_2$ to $a$ with interior in
  $C_2$ and such that $c_2 \in P_2'$. Let $P_3'=P_3$.
  Now $\Sigma'=P_1' \cup P_2' \cup P_3' \cup \{a\}$  is  a pyramid with
  apex $a$ an base $b_1'b_2'b_3$. We apply Theorem~\ref{pyramid2} to
  $\Sigma'$ and $P$. 
  Since $P$ is not local for $\Sigma'$ and $N_G(a)=N_{\Sigma}(a)$ and $N_G(a) \cap \Hub(G) = \emptyset$ (because each neighbor of $a$ ethre has degree 2 in $G$ or is in $\{x_1, x_2, x_3\}$),
  and $a$ is not contained in a cross-edge of an extended near-prism in $G$, it follows that
  one of the following holds:
  \begin{itemize}
  \item $P$ contains a major vertex for $\Sigma'$; or
\item $P$ contains a  corner path for $\Sigma'$;
\end{itemize}
  Theorem~\ref{major} implies that $P$ does not contain a major vertex
  for $\Sigma'$, and therefore $P$ contains a corner path for $\Sigma'$.
  By the minimality of $k$, it follows that $P$ is a corner path
  for $\Sigma'$; consequently $N_{\Sigma' \setminus \{b_1'\}}(p_k)=\{b_2',b_3\}$,
  $N_{\Sigma' \setminus \{b_1'\}}(p_1) \subseteq B_1 \cup C_1$, and there are no
  other edges between $P$ and $\Sigma' \setminus b_1'$. In particular, $p_1$ has a neighbor in $C_1$, and so $x_1 \in C_1$. 

  We claim that $P$ is anticomplete to $C_2 \cup C_3$ and $P \setminus p_k$ is anticomplete to $B_2 \cup B_3$.
  By the minimality of $k$, it follows that $P \setminus p_k$ is anticomplete to $C_2 \cup C_3 \cup B_2 \cup B_3$. From the choice of $c_2$, it follows that $p_k$ is anticomplete to
  $C_2 \cup C_3$. This proves the claim.

  Now let $i \in \{2,3\}$  and let $\Sigma''$ be obtained for $\Sigma'$
  by replacing the path $P_i'$ by an arbitrary path from
  some $b_i'' \in B_i$ to $a$ with interior in $C_i$. Then $P$ is not
  local for $\Sigma''$. Applying Theorems~\ref{pyramid2} and \ref{major}
  to $\Sigma''$ and $P$, we deduce that $P$ is a corner path for
  $\Sigma''$. It follows that $p_k$ is adjacent to $b_i''$.
  Since $b_i''$ was chosen arbitrarily, we conclude that $p_k$
  is complete to $B_2 \cup B_3$. But now, we can replace $B_1$ by $B_1 \cup p_k$ and $C_1$ by $C_1 \cup (P \setminus p_k)$, contradicting the maximality of $W$.
  This proves~\eqref{localcomps}.
  \\
  \\
   Let $F$ be the union of
  the components $D$  of $G \setminus W$ with $N(D) \subseteq B_1 \cup B_2 \cup B_3$. 
  Let $F_i$ be the union of the components $D$ of $G \setminus (W \cup F)$ such that
  $N(D) \subseteq B_i \cup C_i$. By \eqref{localcomps}, $G \setminus W= F_1 \cup F_2 \cup F_3 \cup F$ and the sets $F_1,F_2,F_3,F$ are pairwise disjoint and anticomplete to each other. We may assume that $b \in B_3 \cup C_3 \cup F_3 \cup F$. If $K=B_1 \cup B_2$ is a clique and $x_1, x_2 \not\in B_1 \cup B_2$, then  for $i \in \{1,2\}$, we have
  $D_i=(C_i \cup F_i) \setminus Q_i$ and outcome (2)(a) holds. Thus we may assume that either $B_2$ is not a clique or $B_2 = \{x_2\}$. 

    \sta{We may assume that $B_2 = \{x_2\}$ or $B_3 = \{x_3\}$. \label{b2x2}}

  Suppose not; by symmetry, we may assume also that $B_1 \neq \{x_1\}$. Then $B_2$ is not a clique. Since $G$ is $C_4$-free, it follows that $B_1 \cup B_3$ is a clique. Let $K=B_1 \cup B_3$.  Now $D_1 = (C_1 \cup F_1) \setminus Q_1$
  and $D_2 \subseteq (C_2 \cup F_2 \cup B_2 \cup F) \setminus Q_2$ and $D_3=(C_3 \cup F_3) \setminus Q_3$. 
  
  If $b \in C_3 \cup F_3$, then outcome (2)(a) of the theorem holds with
  $i=1$ and $j=2$, and if $b \in F$, then outcome (2)(a) of the theorem holds with  $i=1$ and $j=3$. Thus we may assume that $b \in B_3$. Let $i=1$
  and $j=2$. Then $b$ is anticomplete to $D_1$.
    Since every path from $b$ to $x_2$ with interior in $D$   contains exactly one vertex
  of $B_2$ and exactly one vertex in $N(b)$, and $B_2$ is not a clique, and for every vertex $v$ in $B_2$ there is a path from $v$ to $x_2$ with interior in $C_2$, it follows that  outcome (2)(b) holds with $D_j' = B_2$ and $q = b$. This proves \eqref{b2x2}. 

  \vspace*{0.3cm}

    Since $b$ is non-adjacent to $a$, it follows that if $B_3 = \{x_3\}$, then $C_3 = Q_3^*$ and $F_3 = \emptyset$, and so $b \in F$ and in particular, $b \in B_2 \cup C_2 \cup F_2 \cup F$; so there is symmetry between $2$ and $3$ in this case. 

    Therefore, we may assume that $B_2 = \{x_2\}$. Since $\{x_1,x_2,x_3\}$ is a stable set, it follows that $x_1, x_3 \not\in B_1 \cup B_3$. Since $b$ is non-adjacent to $x_2$, it follows that $b \in C_3 \cup F_3 \cup F$.

    \sta{$B_1 \cup B_3$ is not a clique. \label{b1b3}}

    Suppose that $B_1 \cup B_3$ is a clique, and let $K = B_1 \cup B_3$. Then $D_1 = (C_1 \cup F_1) \setminus Q_1$ and $D_3 = (C_3 \cup F_3) \setminus Q_3$ and $D_2 \subseteq F$. If $b \in F$, then outcome (2)(a) of the theorem holds with $i = 1$ and $j = 3$. If $b \in C_3 \cup F_3$, then outcome (2)(a) of the theorem holds with $i = 1$ and $j = 2$. This proves \eqref{b1b3}. 
    \sta{$B_1$ is a clique. \label{case1hopefully}}

    Suppose not. It follows that $B_3$ is a clique. Let $K = B_3$. Then $D_3 = (C_3 \cup F_3) \setminus Q_3$ and $D_1, D_2 \subseteq (C_1 \cup B_1 \cup F_2 \cup F) \setminus (Q_1 \cup Q_2)$. If $b \in C_3 \cup F_3$, then outcome (2)(a) of the theorem holds with $i = 1$ and $j = 2$. It follows that $b \in F$. Let $R$ be the component of $F$ containing $b$. 
    Let $M = N(R) \cap B_1$. If $M$ is a clique, then outcome (2)(a) of the theorem holds with $K = B_3 \cup M$ and $i = 1$ and $j = 3$. 

    It follows that $M$ is not a clique. Now outcome (2)(b) of the theorem holds with $i = 3$ and $j = 1$ as well as $K = B_3$ and $D_1' = M$ and $q = x_2$. We show that the last two bullets of (2)(b) hold:
    \begin{itemize}
        \item Let $v \in M$. Let $P$ be a path from $b$ to $v$ in $R \cup \{v\}$. Let $Q$ be a path from $v$ to $x_1$ with $Q^* \subseteq C_1$. Then $P' = b \dd P \dd v \dd Q \dd x_1$ is a path from $b$ to $x_1$ in $D \cup \{x_1\}$. Then, $v$ is the $(P', x_1)$-last vertex in $P' \cap N(q)$, and the second-to-last bullet of (2)(b) holds.  
        \item Let $P$ be a path from $b$ to $x_1$ in $D \cup \{x_1\}$. Traversing $P$ from $b$ to $x_1$, let $v$ be the last vertex of $P$ which is not in $C_1 \cup F_1$. It follows that $v \in B_1$, and so $v$ is complete to $K$, and the last bullet holds. 
    \end{itemize}

\sta{$B_3$ is a clique. \label{b3clique}}

 Suppose not. It follows that $B_1$ is a clique. Let $K = B_1$. Then $D_1 = (C_1 \cup F_1) \setminus Q_1$ and $D_2, D_3 \subseteq (C_3 \cup B_3 \cup F_2 \cup F) \setminus (Q_2 \cup Q_3)$. Suppose first that $b \in F$. Then there is symmetry between $1$ and $3$, and the result follows from \eqref{case1hopefully}. It follows that $b \in F_3 \cup C_3$. Let $R$ be the component of $(F_3 \cup C_3) \setminus Q_3$ containing $b$. Then $M = N_D(R) \subseteq B_3$. If $M$ is a clique, then outcome (2)(a) holds with $K = B_1 \cup M$ and $i = 1$ and $j=2$. So we may assume that $M$ is not a clique.  Now we let $i = 1$, $j = 2$, $q = x_2$, and $D_2' = M$. It follows that outcome (2)(b) holds. This proves \eqref{b3clique}.

 \vspace*{0.3cm} 
 
 Together, \eqref{b1b3}, \eqref{case1hopefully}, and \eqref{b3clique} yield a contradiction; this concludes the proof.  
\end{proof}

\section{Star cutsets from wheels} \label{wheelcutsets}

The following well-known definitions appear, for example, in \cite{TWI}. A \emph{cutset} $C \subseteq V(G)$ of $G$ is a set of vertices such that $G \setminus C$ is disconnected. A {\em star cutset} in a graph $G$ is a cutset $S\subseteq V(G)$ such that  either $S=\emptyset$ or for some $x\in S$, $S\subseteq N[x]$.

Let $G$ be a graph and let $X,Y,Z \subseteq V(G)$. We say that $X$
{\em separates} $Y$ from $Z$ if no component of $G \setminus X$
meets both $Y$ and $Z$.
Recall that a \emph{wheel} $(H, x)$ of $G$ consists of a hole $H$ and a vertex $x$ that has at least three neighbors in $H$, and
a wheel is {\em proper} if $\alpha (N(x) \cap H) \geq 3$.
A \emph{sector} of $(H,x)$ is a path $P$ of $H$ whose ends are distinct and adjacent to $x$, and such that $x$ is anticomplete to $P^*$. A sector $P$ is a \emph{long sector} if $P^*$ is non-empty.
A wheel $(H, x)$ is a \emph{universal wheel} if $x$ is complete to $H$.
The following result was observed in \cite{TWI} based on results of \cite{Addario-Berry2008BisimplicialGraphs, daSilva2013Decomposition2-joins, daSilva2007TriangulatedGraphs} and stated in this form in \cite{TWXV}; it shows that proper wheels force star cutsets in graphs in $\mathcal{C}$.

\begin{theorem}[Abrishami, Chudnovsky, Vu\v{s}kovi\'c \cite{TWI}; see also \cite{TWXV}]\label{wheelstarcutset}
  Let $G \in \mathcal{C}$   and let $(H,v)$ be an proper  wheel in $G$.
  Then there is no component  $D$ of $G \setminus N[v]$
such that $H \subseteq N[D]$. 
\end{theorem}

In particular, we need the following: 
\begin{theorem}[Addario-Berry, Chudnovsky, Havet, Reed, Seymour \cite{Addario-Berry2008BisimplicialGraphs}, da Silva, Vu\v{s}kovi\'c \cite{daSilva2013Decomposition2-joins}]
  Let $G \in \mathcal{C}$ and let $(H,x)$ be a proper wheel in $G$ that is not
  a universal wheel.
    Let $x_1$ and $x_2$ be the endpoints of a long sector $Q$ of $(H, x)$. Let $W$ be the set of all vertices $h$ in $H \cap N(x)$
such that the subpath of $H \setminus \{x_1\}$ from $x_2$ to $h$ contains an even number of neighbors of $x$, and let $Z = H \setminus (Q \cup N(x))$. Let $N' = N(x)\setminus W$. Then, $N' \cup \{x\}$ is a cutset of $G$ that separates $Q^*$ from $W \cup Z$.
\label{lemma:proper_wheel_forcer}
\end{theorem}

\section{Star cutsets from loaded  pyramids}
\label{pyramidcutsets}

The main theorem of this section is the following. 

\begin{theorem}\label{loadedcorner}
  Let $G \in \mathcal{C}$. 
  Suppose that $G$ contains a loaded pyramid $\Pi=(\Sigma,P)$ with $a, b_1,b_2,b_3,P_1,P_2,P_3$ as in the definition. Then there is no connected component $D$ of $G\setminus N[b_2]$ with $\Pi \subseteq N[D]$.
\end{theorem}

In order to prove  Theorem \ref{loadedcorner}, we first prove the following:

\begin{theorem}\label{mainloadedpyr}
  Let $G \in \mathcal{C}$.  Suppose that $G$ contains a loaded pyramid $\Pi=(\Sigma,P)$ with $a, b_1,b_2,b_3,P_1,P_2,P_3$ as in the definition. Moreover, assume that $D$ is a connected component of $G\setminus N[b_2]$ such that neither $N[D]\cap (\Pi\setminus P_3)$ nor $N[D]\cap (P_3\setminus \{a\})$ is empty. Then one of the following holds.
  \begin{itemize}
      \item We have $N[D]\cap (\Pi\setminus P_3)=\{b_1\}$ and $N[D]\cap (P_3\setminus \{a\})=\{b_3\}$; or
      \item There is a proper wheel $(H,b_2)$ in $G$ with two long sectors $\Gamma_1$ and $\Gamma_3$
        such that $\Gamma_1^*$ contains the neighbor of $a$ in $P_1$ and $\Gamma_3=P_3$.
  \end{itemize}
\end{theorem}

In order to prove Theorem \ref{mainloadedpyr}, first we need to verify its assertion for two special types of loaded pyramids, as follows. For a loaded pyramid $(\Sigma,P)$ in a graph $G$, we say $(\Sigma,P)$ is of \textit{type 1} if
$N_{P_1}(p_k) \subseteq N_{P_1}[b_1]$. We 
say that $(\Sigma,P)$ is of \textit{type 2} if $b_1$ is anticomplete to $P$ and $p_k$ has exactly two neighbors $x,y \in P_1$, and $x$ is adjacent to
$y$. 

\begin{theorem}\label{loadedpyr1}
  Let $G \in \mathcal{C}$. Suppose that $G$ contains a loaded pyramid $\Pi=(\Sigma,P)$ of type 1 with $a,b_1,b_2,b_3,P_1,P_2,P_3,P$ as in the definition. Moreover, assume that $D$ is a connected component of $G\setminus N[b_2]$ such that neither $N[D]\cap (\Pi\setminus P_3)$ nor $N[D]\cap (P_3\setminus \{a\})$ is empty. Then one of the following holds.
  \begin{itemize}
      \item We have $N[D]\cap (\Pi\setminus P_3)=\{b_1\}$ and $N[D]\cap (P_3\setminus \{a\})=\{b_3\}$; or
      \item There is a proper wheel $(H,b_2)$ in $G$ with two long sectors $\Gamma_1$ and $\Gamma_3$
         such that $\Gamma_1^*$ contains the neighbor of $a$ in $P_1$ and $\Gamma_3=P_3$.
  \end{itemize}
\end{theorem}
\begin{proof}
Suppose not. Then since the first bullet of Theorem \ref{loadedpyr1} does not hold, there exists a shortest  path $Q=q_1\dd \cdots \dd q_t$ in $G\setminus N[b_2]$ with $t\geq 1$ where
  \begin{itemize}
      \item $q_1$ has a neighbor in $P_3\setminus \{a\}$;
      \item $q_t$ has a neighbor in $\Pi\setminus P_3$; and 
      \item either $q_1$ has a neighbor in $P_3^*$ or $q_t$ has a neighbor in $P_1^* \cup P$.
  \end{itemize}
  From the minimality of $Q$, it follows that $Q^*$ is anticomplete to $(\Sigma \cup P)
\setminus \{a,b_1,b_3\}$.
Let $c$ be the neighbor of $b_1$ in $P_1$. Then
$c \neq a$,  and since $G$ is
$C_4$-free, $c$ is non-adjacent to $a$.
\\
\\
\sta{\label{lp11} $N(q_1)\cap P_3^*\neq \emptyset$.}

For otherwise we have $N(q_1)\cap (P_3\setminus \{a\})=\{b_3\}$. Consequently $q_t$ has a neighbor in $P_1^* \cup P$, and so by the minimality of $Q$, $Q\setminus \{q_1\}$ is anticomplete to $P_3\setminus \{a\}$. Traversing $Q$ starting at $q_1$, let $q$ be the the first vertex with a neighbor in $P'=(P_1\setminus b_1)\cup P$, and traversing the path $P'$ starting at $a$, let $x$ be the first vertex adjacent to $q$. Then $H=a \dd P' \dd x \dd q \dd Q \dd q_1 \dd b_3 \dd P_3$ is a hole in $G$. Note that
$\{a,b_3\}\subseteq N(b_2)\cap H\subseteq \{a,b_3,p_1\}$. Since $H\cup \{b_2\}$ is not a theta in $G$, it follows that $N(b_2)\cap H= \{a,b_3,p_1\}$. But then $(H,b_2)$ is a wheel in $G$ satisfying the second bullet of Theorem \ref{loadedpyr1}, a contradiction. This proves \eqref{lp11}.
\\
\\
\sta{\label{lp13} $N(b_1)\cap (Q\setminus q_t)=\emptyset$, and $q_t$ has a neighbor in $P_1^* \cup P$.}

Suppose not. Traversing $Q$ starting at $q_t$, let $q$ be the last vertex adjacent to $b_1$. Since \eqref{lp13} does not hold, it follows that $q$ has no neighbor in $P_1^* \cup P$. Note that by \eqref{lp11}, $q \dd Q \dd q_1 \cup (P_3\setminus b_3)$ is connected, and so contains an induced path $R$ from $q$ to $a$. But then $P_1\cup R\cup \{b_2\}$ is a theta with ends $a,b_1$  in $G$, which is impossible. This proves \eqref{lp13}.
\\
\\
\sta{\label{lp12} $N(Q)\cap (P\cup \{b_1\})\neq \emptyset$.}
Suppose not. Then $q_t$ has a neighbor in $P_1 \setminus b_1$. Traversing
$P_1\setminus b_1$ starting at $a$, let $x$ be the last vertex with a neighbor in $Q$. It follows that $x \neq a$. By \eqref{lp13} $x$ is adjacent to $q_t$ and anticomplete to $Q \setminus q_t$.
Traversing $P_3$ starting at $a$, let $y$ be the last vertex adjacent to $q_1$; hence $y\neq a$. Then both $H_1=(P_1\setminus b_1)\cup P$ and
$H_2=c \dd P_1 \dd x  \dd q_t \dd Q \dd q_1 \dd y \dd P_3 \dd b_3 \dd b_2 \dd p_1 \dd P \dd p_k \dd c$ are holes in $G$. Note that $b_1$ has at least two
non-adjacent neighbors in $H_1$ and $N(b_1)\cap H_2=(N(b_1)\cap H_1)\cup \{b_3\}$. So one of $(H_1,b_1)$ and $(H_2,b_1)$ is an even wheel in $G$, a contradiction. This proves \eqref{lp12}.
\\
\\
\sta{\label{lp14} $N(q_t)\cap P \neq \emptyset$.}
 
Suppose for a contradiction that $q_t$ is anticomplete to $P$. Recall that by the minimality of $t$, $Q\setminus q_t$ is also anticomplete to $P$. Thus, $Q$ is anticomplete to $P$, and so by \eqref{lp12}, we have $N(b_1)\cap Q\neq\emptyset$. This, together with \eqref{lp13}, implies that $q_t$ is adjacent to $b_1$. Since $q_t \dd b_1 \dd b_2 \dd a \dd q_t$ is not a $C_4$ in $G$,
$q_t$ is nonadjacent to $a$. Moreover, note that by \eqref{lp13}, it follows that $q_t$ has a neighbor in $P_1^*$. Traversing $P_1\setminus b_1$ starting at $c$, let $x$ be the last vertex adjacent to $q_t$. Also, by \eqref{lp11}, $(P_3\setminus b_3)\cup Q$ is connected, and so contains a path $R$ from $q_t$ to $a$ (note that $R$ has more than one edge, as $q_t$ is not adjacent to $a$). Now, if $x\neq c$, then we get a theta with ends $a,q_t$ and paths
$a \dd P_1 \dd x \dd q_t$, $R$ and $q_t \dd b_1 \dd b_2 \dd a$,
a contradiction.
Therefore, we have $x=c$. But now we get a theta with ends $c,a$ and paths
$c \dd P_1 \dd a$, $c \dd p_k \dd P \dd p_1 \dd b_2 \dd a$ and
$c \dd q_t \dd R \dd a$, again a contradiction.
This proves \eqref{lp14}.
\\
\\
In particular, \eqref{lp14} implies that $b_3$ is anticomplete to $Q \setminus q_1$. Henceforth, we denote the hole $(P_1\setminus b_1)\cup P$ in $G$ by $K$. Also, in view of
\eqref{lp14}, let $i \in \{1, \dots, k\}$ be minimum such that $q_t$ is adjacent to $p_i$.
Denote the neighbor of $a$ in $P_1$ by $w$. Then $w \neq c$.
\\
\\
\sta{\label{lp15} $|N(q_t)\cap K|\geq 2$.}

Suppose not. Then $p_i$ is the unique neighbor of $q_t$ in $K$.
By \eqref{lp11}, $(P_3\setminus b_3)\cup Q$ is connected, and so contains a path $R$ from $q_t$ to $a$. But now $K\cup R$ is a theta with ends $a,p_i$ in $G$, which is impossible. This proves \eqref{lp15}.
\\
\\
\sta{\label{lp16} $q_t$ has a neighbor in $K \setminus (N[a] \cup \{p_1\})$.}

Suppose not. Then $i=1$. Since $q_t \dd p_1 \dd b_2 \dd a \dd q_t$ is not a $C_4$
in $G$, it follows that $q_t$ is non-adjacent to $a$. Now $N_K(q_t)=\{p_1,w\}$
and so $K \cup q_t$ is a theta, a contradiction. This proves~\eqref{lp16}.
\\
\\
In view of \eqref{lp16} 
there is a path $P'$ from $q_t$ to $b_1$ with interior in $K \setminus \{b_2, a, p_1, w \}$. It follows from \eqref{lp13} and the minimality of $t$ that $q_t$ is
the only vertex of  $P'^*$ with a neighbor in $Q \setminus q_t$.
\\
\\
\sta{\label{lp18} $N(a) \cap Q = \emptyset$.}
 
Suppose not.  By \eqref{lp11}, $(P_3\setminus a)\cup \{q_1\}$ is connected, and so contains an induced path $R$ from $b_3$ to $q_1$.
 Now $H_1=Q\cup P' \cup R$ and $H_2=Q\cup b_2 \dd p_1 \dd P \dd p_i \cup R$
 are holes in $G$. Also, we have $N(a)\cap H_1=N(a)\cap (Q \cup R)$ and $N(a)\cap H_2=(N(a)\cap H_1)\cup \{b_2\}$. If $|N(a)\cap (Q \cup R)|\geq 3$, then either $(H_1,a)$ or $(H_2,a)$ is an even wheel in $G$, which is impossible. Also, if $|N(a)\cap (Q \cup R)|=1$, then $H_2\cup \{a\}$ is a theta in $G$, which is impossible. It follows that $|N(a)\cap (Q \cup R)|=2$.
 Since  $H_1\cup \{a\}$ is not a theta in $G$,
 the two neighbors of $a$ in $H_1$ are adjacent and contained in $(Q \cup R)\setminus \{b_3\}$.
But now
$H_1\cup \{a,b_2\}$ is a prism in $G$, which is impossible. This proves \eqref{lp18}.
\\
\\
If $q_1$ has exactly one neighbor $x$ in $P_3$, then by \eqref{lp11},
$x \in P_3^*$ and, using \eqref{lp18}  we get a theta with ends $x,b_2$ and paths
$x \dd P_3 \dd a \dd b_2$, $x \dd q_1 \dd Q \dd q_t \dd p_i \dd P \dd b_2$,
and $x \dd P_3 \dd b_3 \dd b_2$, a contradiction.
Thus, traversing $P_3$ starting at $b_3$, we may assume $y$ and $z$ to be the first and the last neighbor of $q_1$ in $P_3$, respectively, and $y \neq z$.
If $z$ is non-adjacent to $y$, using \eqref{lp18} we get a theta with ends $b_2,q_1$ and paths
$q_1 \dd z  \dd P_3 \dd a \dd b_2$, $q_1 \dd Q \dd q_t \dd p_i \dd P \dd b_2$,
and $q_1 \dd y  \dd P_3 \dd b_3 \dd b_2$, a contradiction.
So $z$ and $y$ are adjacent.
But now, again using \eqref{lp18},  we get a near-prism with triangles $b_1b_2b_3$ and $q_1zy$
and paths $b_1 \dd P' \dd q_t \dd Q \dd q_1$, $b_2 \dd a \dd P_3 \dd z$ and
$b_3 \dd P_3 \dd y$, a contradiction.
\end{proof}

\begin{theorem}\label{loadedpyr2}
  Let $G \in \mathcal{C}$. Suppose that $G$ contains a loaded pyramid $\Pi=(\Sigma,P)$ of type 2 with $a, b_1,b_2,b_3,P_1,P_2,P_3,P$ as in the definition.
  Moreover, assume that $D$ is a connected component of $G\setminus N[b_2]$ such that neither $N[D]\cap (\Pi\setminus P_3)$ nor $N[D]\cap (P_3\setminus \{a\})$ is empty. Then one of the following holds.
  \begin{itemize}
      \item We have $N[D]\cap (\Pi\setminus P_3)=\{b_1\}$ and $N[D]\cap (P_3\setminus \{a\})=\{b_3\}$; or
      \item There is a proper wheel $(H,b_2)$ in $G$ with two long sectors $\Gamma_1$ and $\Gamma_3$ such that $\Gamma_1^*$ contains the neighbor of $a$ in $P_1$ and $\Gamma_3=P_3$.
  \end{itemize}
\end{theorem}
\begin{proof}
  Suppose not.
  Assume that $P_1$ traverses $b_1,x,y,a$ in this order, where $x$ and $y$ are the two neighbors of $q_t$ in $P_1$. Then  $b_1 \neq x$.
  Let $P'_1=b_1 \dd P_1 \dd x$,
  $P''_1=y \dd P_1 \dd a$. We write $P = p_1 \dd \cdots \dd p_k$ where $p_1$ is adjacent to $b_2$, and $p_k$ has a neighbor in $P_1^*$. 

 Since the first bullet of Theorem \ref{loadedpyr2} does not hold, there exists a shortest  path $Q=q_1\dd \cdots \dd q_t$ in $G\setminus N[b_2]$ with $t\geq 1$ where
  \begin{itemize}
      \item $q_1$ has a neighbor in $P_3\setminus \{a\}$;
      \item $q_t$ has a neighbor in $\Pi\setminus P_3$; and 
      \item either $q_1$ has a neighbor in $P_3^*$ or $q_t$ has a neighbor in $P_1^* \cup P$.
  \end{itemize}
  From the minimality of $Q$, it follows that $Q^*$ is anticomplete to $(\Sigma \cup P)
\setminus \{a,b_1,b_3\}$.
Let $c$ be the neighbor of $b_1$ in $P_1$. Then $c \neq a$, and since
$G$ is $C_4$-free, $c$ is non-adjacent to $a$.
\\
\\
\sta{$q_t$ has a neighbor in $P \cup P_1'$. \label{qtPP1'}}

Suppose not. Then $N_{P_1}(q_t) \subseteq P_1''$, and there is a path  $R$ from 
$y$ to $b_3$ with $R \subseteq P_1'' \cup Q \cup P_3$. Now we get a prism with
triangles $p_kxy$ and $b_1b_2b_3$ and paths $P$, $P_1'$ and $R$, a contradiction. This proves \eqref{qtPP1'}.
\\
\\
\sta{$q_1$ has a neighbor in $P_3^*$. \label{q1P3*}}

Suppose not. Then $N_{P_3}(q_1) \subseteq \{b_3,a\}$. Suppose first that either $a$ has a neighbor in $Q$, or $q_t$ has a neighbor in $(P_1 \setminus b_1) \cup (P \setminus p_1)$. Then there is a path
$R$  from $q_1$ to $a$ with $R^* \subseteq Q \cup (P_1 \setminus b_1) \cup (P \setminus p_1)$.
Now we get a theta
with ends $b_3,a$ and paths $P_3$, $b_3 \dd b_2 \dd a$ and $b_3 \dd q_1 \dd R \dd a$, a contradiction. 

It follows that $a$ is anticomplete to $Q$, and $N_{P_1 \cup P}(q_t) \subseteq \{p_1, b_1\}$. From the choice of $Q$, and since $G$ is $C_4$-free, it follows that $N_{P_1 \cup P}(q_t) = \{p_1\}$. Now $H = b_3 \dd Q \dd p_1 \dd P \dd y \dd P_1'' \dd a \dd P_3 \dd b$ is a hole, and so $(H, b_2)$ is a proper wheel satisfying the second outcome of the theorem. This proves~\eqref{q1P3*}.
\\
\\
It follows from \eqref{q1P3*} and the minimality of $t$ that
$b_1$ is anticomplete to $Q \setminus q_t$. Let $v$ be the neighbor of
$q_1$ in $P_3$ closest to $b_3$. Let $v'$ be the neighbor of
$q_1$ in $P_3$ closest to $a$.
\\
\\
\sta{$q_t$ has a neighbor in $P_1^* \cup P$. \label{qtPP1*}}

  Suppose not. Then $N_{P_1 \cup P}(q_t) \subseteq \{b_1,a\}$.
  By \eqref{q1P3*} there is a path $R$ from $q_t$ to $a$ with
  $R^* \subseteq Q \cup (P_3 \setminus b_3)$.
  Now we get a theta with ends $b_1,a$ and paths $P_1$,
  $b_1 \dd b_2 \dd a$ and $b_1 \dd q_t \dd R \dd a$, a contradiction.
  This proves \eqref{qtPP1*}.
  \\
  \\
  It follows from \eqref{qtPP1*} and the minimality of $t$ that
$b_3$ is anticomplete to $Q \setminus q_1$.
\\
\\
\sta{$a$ is anticomplete to $Q$. \label{aantiQ}}

Suppose not. Let  $R_1$ be a path  from $q_1$ to $b_1$
with $R_1^* \subseteq Q \cup P_1' \cup P$. If possible, choose $R_1$ so that
$p_1 \not \in R_1$. Observe that $Q \subseteq R_1$ and $a$ is anticomplete to
$R_1 \setminus Q$.
  Let $H_1$ be the hole $b_1 \dd R_1 \dd q_1 \dd v \dd P_3 \dd b_3 \dd b_1$.
  By \eqref{qtPP1*} there is a path $R_2$ from $b_2$ to $q_1$
  with $R_2^* \subseteq Q \cup P \cup (P_1' \setminus b_1)$.
  Observe that $Q \subseteq R_2$ and $a$ is anticomplete to $R_2 \setminus Q$.
  Let $H_2$ be the hole $b_2 \dd R_2 \dd q_1 \dd v \dd P_3 \dd b_3 \dd b_2$.
  Now $N_{H_2}(a) = N_{H_1}(a) \cup \{b_2\}$.
  Since neither of $(H_1,a)$, $(H_2,a)$ is a theta or an even wheel in $G$,
  it follows that $a$ has exactly two neighbors $uw$ in $Q$,
  and $u$ is adjacent to $w$. We may assume that $Q$ traverses $q_1,u,w,q_t$
  in this order.
  If $p_1 \not \in R_1$, we get a prism with triangles
  $auw$ and $b_1b_2b_3$ and paths $ab_2$, $u \dd Q \dd q_1 \dd v \dd P_3 \dd b_3$ and $w \dd R_1 \dd b_1$, a contradiction. This proves that $p_1 \in R_1$,
  and therefore $N_{P \cup P_1'}(q_t) = \{p_1\}$.
  Now let $H_3$ be the hole
  $b_1 \dd P_1' \dd x \dd p_k \dd P \dd p_1 \dd q_t \dd Q \dd w  \dd a  \dd P_3 \dd b_3 \dd b_1$. Then $N_{H_3}(b_2)=\{b_1,b_3,p_1,a\}$ and so
  $(H_3,b_2)$ is an even wheel, a contradiction.
  This proves \eqref{aantiQ}.
  \\
  \\
  \sta{$N(q_t) \cap (P_1 \cup P)$ is a clique of size at least 2.
    \label{qtnbrs}}
Suppose not. Since $G$ is $C_4$-free, $q_t$ is non-adjacent to at least one of $b_1$ and the neighbor $p_1$ of $b_2$ in $P$. 
  Suppose that there are two paths $R_1,R_2$ from $q_t$ to $a$,
  both with interior in $P_1 \cup P \cup b_2$ and such that
  $R_1^*$ is anticomplete to $R_2^*$. Then we get a theta with ends $q_t,a$
  and paths $R_1,R_2$ and $q_t \dd Q \dd v' \dd P_3 \dd a$, a contradiction.
  It follows that $q_t$ has a unique neighbor $w$ in $P_1 \cup P \cup b_2$.
  By \eqref{qtPP1'}, $w$ is non-adjacent to $a$. Now
  there are two path $R_1, R_2$ from $w$ to $a$ with interior in $P_1 \cup P \cup b_2$ and such that
  $R_1^*$ is anticomplete to $R_2^*$, and we get a theta with ends
  $w,a$ and paths $R_1,R_2$ and $w \dd q_t \dd Q \dd v' \dd P_3 \dd a$, a contradiction. This proves~\eqref{qtnbrs}.
  \\
  \\
  If follows from \eqref{qtPP1'} and \eqref{qtnbrs} that
  $N_{P_1''}(q_t) \subseteq \{y\}$.
  \\
  \\
  \sta{$q_t$ is adjacent to $y$. \label{qty}}
  Suppose not. Let $N_{P_1 \cup P}(q_t)=\{u,w\}$,
  and we may assume that $P_1' \cup P$ traverses $b_1,u,w,p_1$ in this order.
  Then there is a prism  with triangles $uwq_t$ and $b_1b_2b_3$,
  two of whose paths have interior in $P_1' \cup P$, and the third one is
  $q_t \dd Q \dd q_1 \dd v \dd P_3  \dd b_3$, a contradiction. This proves~\eqref{qty}.
  \\
  \\  
  \sta{$q_t$ is non-adjacent to $x$. \label{qtnotx}}

  Suppose $q_t$ is adjacent to $x$. Since we do not get a prism with triangles
  $xyq_t$ and $b_1b_2b_3$ and paths $x \dd P_1 \dd b_1$, $y \dd P_1 \dd a \dd b_2$ ad $q_t \dd Q \dd q_1 \dd v \dd P_3 \dd b_3$, it follows that $v = v'$ is adjacent to $a$.
  But now we get a theta with ends $b_2,v$ and paths $b_2 \dd a \dd v$,
  $b_2 \dd p_1 \dd P \dd p_k \dd x \dd q_t \dd Q \dd q_1 \dd v$ ($x$ is omitted if $p_k$ is adjacent to $q_t$) and
  $b_2 \dd b_3 \dd P_3 \dd v$. This  proves \eqref{qtnotx}.
  \\
  \\
  By \eqref{qtnbrs}, \eqref{qty}, and \eqref{qtnotx}, we deduce that
  $N_{P \cup P_1}(q_t)=\{y,p_k\}$.
  \\
  \\
%
Suppose that $y$ is non-adjacent to $a$. Then $G$ contains a theta with ends $a, y$ and paths $P_1''$, $y \dd P_1' \dd b_1 \dd b_2 \dd a$, and $y \dd q_t \dd Q \dd q_1 \dd v' \dd P_3 \dd a$,
 a contradiction. This proves that $y$ is adjacent to $a$. Since $G$ is $C_4$-free, it follows that $p_k$ is non-adjacent to $b_2$. Now, since
 by \eqref{q1P3*} $v' \neq b_3$, there is a theta with ends $p_k, b_2$ and paths $p_k \dd x \dd P_1' \dd b_1 \dd b_2$, $p_k \dd P \dd p_1 \dd b_2$ and $p_k \dd q_t \dd Q \dd q_1 \dd v' \dd P_3 \dd a \dd b_2$.
\end{proof}

Having proved Theorems \ref{loadedpyr1} and \ref{loadedpyr2}, we can give a proof of Theorem \ref{mainloadedpyr}, below.

\begin{proof}[Proof of Theorem \ref{mainloadedpyr}]
Suppose for a contradiction that there exists a loaded pyramid $\Pi=(\Sigma,P)$ in $G$ with $a, b_1,b_2,b_3,P_1,P_2,P_3, P$ as in the definition, and a connected component $D$ of $G\setminus N[b_2]$ for which neither $N[D]\cap (\Pi\setminus P_3)$ nor $N[D]\cap (P_3\setminus \{a\})$ is empty, such that neither of the two bullets of Theorem \ref{mainloadedpyr} hold. Also, let $a'$ be the neighbor of $a$ in $P_1$, and subject to the above properties and $a'$ being the neighbor of $a$ in $P_1$, let $|P_1|$ be minimal. We claim:
\\
\\
\sta{\label{gettype1or2} We have $N(p_k)\cap P_1$ is a clique.}

Suppose not. Traversing $P_1$ from $b_1$ to $a$, let $x$ and $y$ be the first and the last neighbor of $p_k$ in $P_1$, respectively. Then $x$ and $y$ are distinct and non-adjacent, and we have $y\neq a$. If $k=1$,
then $H=b_1 \dd P_1 \dd x \dd p_k \dd y \dd P_1  \dd a \dd P_3 \dd b_3 \dd b_1$
is a hole in $G$ and $b_2$ has exactly four neighbors in $H$, namely $a,b_1,b_3$ and $p_k$. But then $(H,b_2)$ is an even wheel in $G$, which is impossible. So $k>1$ and consequently $b_2$ is not adjacent to $p_k$. Now, replacing the path $P_1$ in the pyramid $\Sigma$ by the path $P_1'=b_1 \dd P_1 \dd x \dd p_k \dd y \dd P_1 \dd a$
we obtain a pyramid $\Sigma'$ in $G$, where $\Pi'=(\Sigma',P \setminus p_k)$ is a loaded pyramid in $G$ with $|P_1'|<|P_1|$ and $D$ is a connected component of $G\setminus N[b_2]$ for which neither $N[D]\cap (\Pi'\setminus P_3)$ nor $N[D]\cap (P_3\setminus \{a\})$ is empty, such that neither of the two bullets of Theorem \ref{mainloadedpyr} hold. This violates the choice of $\Pi$, and hence proves \eqref{gettype1or2}.
\\
\\
\sta{\label{b1anti} The vertex $b_1$ is anticomplete to $P$.}

Suppose not. Let $i$ be maximum such that $b_1$ is adjacent to $p_i$.
Also, traversing $P_1$ from $b_1$ to $a$, let $y$ be the last neighbor of $p_k$ in $P_1$. Then we have $y\in P_1\setminus \{a,b_1\}$. If $y$ is adjacent to $b_1$, then $N_{P_1}(p_k) \subseteq N_{P_1}(b_1)$ and 
$\Pi$ is a loaded pyramid of type 1, which together with Theorem \ref{loadedpyr1} implies that the one of the two bullets of Theorem \ref{mainloadedpyr} holds, a contradiction. Thus, $y$ is not adjacent to $b_1$.
Suppose that $i=1$.
Then $H=b_1 \dd p_1 \dd P \dd y \dd P_1 \dd a \dd P_3 \dd b_3 \dd b_1$
is a hole
in $G$ and $b_2$ has exactly four neighbors in $H$, namely $a,b_1,b_3$ and $p_1$. But then $(H,b_2)$ is an even wheel, which is impossible. So $i>1$,
and consequently $b_2$ is not adjacent to $p_i$. Now, replacing the path $P_1$ in the pyramid $\Sigma$ by the path $P_1'=b_1 \dd p_i \dd P \dd p_k \dd y \dd P_1 \dd a$
we obtain a pyramid $\Sigma'$ in $G$, where $\Pi'=(\Sigma', p_1 \dd P \dd p_{i-1})$ is a loaded pyramid in $G$ with $|P_1'|<|P_1|$, and $D$ is a connected component of $G\setminus N[b_2]$ for which neither $N[D]\cap (\Pi'\setminus P_3)$ nor $N[D]\cap (P_3\setminus \{a\})$ is empty, such that neither of the two bullets of Theorem \ref{mainloadedpyr} hold. This violates the choice of $\Pi$, and hence proves \eqref{b1anti}.
\\
\\
Now, if $p_k$ has exactly one neighbor $x$ in $P_1$, then
$x \neq b_1$ and so
by \eqref{b1anti}, $P_1\cup P \cup \{b_2\}$ is a theta with ends $b_2,x$ in $G$, a contradiction. Therefore, by \eqref{gettype1or2}, $p_k$ has exactly two neighbors in $P_1$, which are adjacent. Also, by \eqref{b1anti}, $b_1$ is anticomplete to $P$, and in particular $b_1$ is not adjacent to $p_k$. Therefore, we may assume that $N(p_k)\cap P_1=\{x,y\}\subseteq P_1\setminus \{a,b_1\}$
where $P_1$ traverses $b_1,x,y,a$ in this order.
But now $\Pi$ is loaded pyramid of type 2, which along with Theorem \ref{loadedpyr2} implies that the one of the two bullet of Theorem \ref{mainloadedpyr} holds, a contradiction. This completes the proof of Theorem \ref{mainloadedpyr}.
\end{proof}

We now prove Theorem~\ref{loadedcorner}.

\begin{proof}
  Suppose not, and let $D$ be a component of $G \setminus N[b_2]$
  such that $\Pi \subseteq N[D]$. By Theorem~\ref{mainloadedpyr} we deduce that
  there is a proper wheel $(H,b_2)$ in $G$ with two long sectors $\Gamma_1$ and
  $\Gamma_3$ such that $\Gamma_1^*$ contains the neighbor of $a$ in $P_1$ and $\Gamma_3=P_3$.
 But now we get a contradiction to Theorem~\ref{lemma:proper_wheel_forcer}.
  \end{proof}

From Theorem~\ref{wheelstarcutset} and Theorem~\ref{loadedcorner} we
deduce:

\begin{theorem}\label{hubstarcutset}
  Let $G \in \mathcal{C}$   and let $(H,v)$ be a proper  wheel  or a loaded
  pyramid in $G$.
  Then there is no component  $D$ of $G \setminus N[v]$
such that $H \subseteq N[D]$. 
\end{theorem}

\section{Tree strip systems}
\label{strips}
In this section we summarize results from \cite{bisimplicialnew}
that allow us to deal with cross-edges of
near-prisms in an even-hole-free graph. 
Recall that
an {\em extended near-prism} is a graph obtained from a near-prism by adding one extra edge, as follows. 
Let $P_1,P_2,P_3$ be as in the definition
of a near-prism, and let $a\in P_1^*$ and $b\in P_2^*$; and add an edge $ab$. 
We call $ab$ the {\em cross-edge} of the extended near-prism.
Next we explain a theorem from \cite{bisimplicialnew} that describes the structure of graphs with an extended near prism. We start with several definitions from
\cite{bisimplicialnew}.

Let $T$ be a tree with at least 3 leaves. 
A {\em leaf} of $T$ is a vertex of degree exactly one, and a {\em leaf-edge} is an edge incident with a leaf.
Let $(A',B')$ be a bipartition of $T$, and assume that 
for every $v\in V(T)$, there is at most one component $C$ of $T\setminus v$ such that $A'\cap C=\emptyset$,
and at most one such that $B'\cap C=\emptyset$. (Note that every component $C$ of $T\setminus v$ contains a leaf
of $T$ and therefore meets at least one of $A',B'$.) Since $|V(T)|\ge 3$, each leaf-edge
is incident with only one leaf; let $A$ be the set of leaf-edges incident with a leaf in $A'$, and define $B$
similarly.
Let $L(T)$ be the line-graph of $T$, that is the vertex set of $L(T)$ is the edge set of $T$, and two edges of $T$ are adjacent in $L(T)$ if they share an end in $T$. 
Add to $L(T)$ two more vertices $a,b$ and the edge $ab$, and make $a$ complete to $A$ and $b$ complete to $B$,
forming a graph $H(T)$ with vertex set $E(T)\cup \{a,b\}$.  We say that
$H(T)$ is an {\em extended tree line-graph}, and $ab$ is its {\em cross-edge}. 

Every extended near-prism is an extended tree line-graph, where the corresponding tree has four leaves and exactly two
vertices of degree three.

A {\em branch-vertex} of a tree is a vertex of degree different from two (thus, leaves are branch-vertices).
A {\em branch} of a tree $T$ is a path $P$ of $T$ with distinct ends $u,v$, both branch-vertices, 
such that every vertex of $P^*$ has degree
two in $T$. 
Every edge of $T$ belongs to a unique branch.

Let $T$ be a tree, and let $U$ be the set of branch-vertices of $T$;
and make a new tree $J$ with vertex set $U$  by making $u,v\in U$ adjacent in
$J$ if there
is a branch of $T$ with ends $u,v$. We call $J$ the {\em shape} of $T$. Thus $J$ has no vertices of degree two;
and $T$ is obtained from $J$ by replacing each edge by a path of positive length.

Let $A,B,C$ be subsets of $V(G)$, with $A, B\ne \emptyset$ and disjoint from $C$, and let $S=(A,B,C)$. A {\rm rung} of $S$, or
an {\em $S$-rung},
is a path $p_1\cc p_k$ of $G[A\cup B\cup C]$ such that $p_1\in A$, $p_k\in B$ and $p_2\ll p_{k-1}\in C$, and if
$k>1$ then $p_1\notin B$ and $p_k\notin A$. (If $A\cap B\ne \emptyset$, then $k=1$ is possible.)
If every vertex in $A\cup B\cup C$ belongs to an $S$-rung, we call $S$ a {\em strip}.

Let $J$ be a tree with at least three vertices.
$M$ is a {\em $J$-strip system} in a graph $G$ consists of:
\begin{itemize}
\item for each edge $e=uv$ of $J$,  a subset $M_{uv}=M_{vu}=M_e$ of $V(G)$; and
\item for each $v\in V(J)$,  a subset $M_v$ of $V(G)$
\end{itemize}
satisfying the following conditions:
\begin{itemize}
\item the sets $M_{e}\; (e \in E(J))$ are pairwise disjoint;
\item for each $u \in V(J)$, $M_u \subseteq \bigcup_{v \in N_J(u)} M_{uv}$;
\item for each $uv \in E(J)$, $(M_{uv}\cap M_u, M_{uv}\cap M_v, M_{uv}\setminus (M_u\cup M_v))$ is a 
strip;
\item if $uv,wx \in E(J)$ with $u,v,w,x$ all distinct, then there are
no edges between $M_{uv}$ and $M_{wx}$;
\item if $uv,uw \in E(J)$ with $v \ne w$, then $M_u \cap M_{uv}$ is complete to
$M_u \cap M_{uw}$, and there are no other edges between $M_{uv}$ and $M_{uw}$.
\end{itemize}
A rung of the strip $(M_{uv}\cap M_u, M_{uv}\cap M_v, M_{uv}\setminus (M_u\cup M_v))$ will be called an {\em $e$-rung}
or {\em $uv$-rung}. (the dependence on $M$ and $J$ is left implicit, for the sake of brevity.)
Let $V(M)$ denote the union of the sets $M_e\;(e\in E(J))$.

Let $J$ be a tree, let $M$ be a $J$-strip system in $G$, and let $(\alpha,\beta)$ be a partition of the set of 
leaves of $J$. We say an edge $ab$ of $G$ is a {\em cross-edge}
for $M$ with {\em partition $(\alpha,\beta)$} if:
\begin{itemize}
\item $J$ has no vertex of degree two, and at least three vertices;
\item for every vertex $s\in V(J)$, $s$ has at most one neighbor in $\alpha$, and at most one in $\beta$;
\item $a, b \not\in V(M)$; and 
\item $a$ is complete to $\bigcup_{u\in \alpha}M_u$, and $a$ has no other neighbors in $V(M)$;
$b$ is complete to $\bigcup_{u\in \beta}M_u$, and $b$ has no other neighbors in $V(M)$.
\end{itemize}
Under these circumstances, a leaf $v \in \alpha$ is called and {\em $a$-leaf}, and a leaf $v \in \beta$ is a {\em $b$-leaf}.

Let $M$ be a $J$-strip system in $G$ with cross-edge $ab$ and partition $(\alpha,\beta)$.
We say $X\subseteq V(M)\cup \{a,b\}$
is {\em local} if either:
\begin{itemize}
\item $X\subseteq M_e$ for some $e \in E(J)$; 
\item $X\subseteq M_u$ for some $u\in V(J)$; or
\item $X$ contains $a$ and not $b$, and $X\setminus \{a\}\subseteq M_{u}$ for some leaf $u\in \alpha$; 
  or $X$ contains $b$ and not $a$, and $X\setminus \{a\}\subseteq M_{u}$ for some leaf $u\in \beta$.
\end{itemize}

Next we describe two maximizations:
\begin{itemize}
\item
We start with an even-hole-free graph $G$, and an edge $ab$ of $G$, such that there is an extended tree line-graph $H(T)$
that is an induced subgraph of $G$, 
with cross-edge $ab$. Subject to this we choose $T$ with as many branches as possible, that is, such that its shape $J$
has $|E(J)|$ maximum.
\item Then we choose a $J$-strip system $M$ in $G$ with the same cross-edge $ab$, with $V(M)$ maximal.
\end{itemize}
In these circumstances $(J,M)$ is said to be {\em optimal for $ab$}.

We will need the following special case of Theorem~4.2 of \cite{bisimplicialnew} (here we have corrected a typo
that occurred in the statement in \cite{bisimplicialnew}):

\begin{theorem}[Chudnovsky, Seymour \cite{bisimplicialnew}]\label{treestructplusminus}
Let $ab$ be an edge of an even-hole-free graph $G$, and let $(J,M)$ be optimal for $ab$.
Assume that no vertex of $G$ is adjacent to both $a$ and $b$.
Then for every connected induced subgraph $F$ of $G\setminus (M \cup \{a,b\})$:
\begin{itemize}
\item if not both $a,b$ have neighbors in $V(F)$, then the set of vertices in $V(M) \cup \{a,b\}$
with a neighbor in $V(F)$ is local;
\item if both $a,b$ have neighbors in $V(F)$, then there exists a leaf $t$ of $J$ such that every vertex of $V(M)$
with a neighbor in $V(F)$ belongs to $M_t$.
\end{itemize}
\end{theorem}

Theorem~\ref{treestructplusminus} assumes that $G$ is even-hole-free,
rather than $G \in \mathcal{C}$. It is likely that the proof works
under the more general assumption, we but we have not verified the details.
The last result of this section is a slight strengthening of
Theorem~\ref{treestructplusminus}:

\begin{theorem}\label{treestructplus}
Let $ab$ be an edge of an even-hole-free graph $G$, and let $(J,M)$ be optimal for $ab$. Let $(\alpha, \beta)$ be the partition such that $ab$ is a cross-edge for $M$ with partition $(\alpha, \beta)$.
Assume that
\begin{itemize}
\item $|\beta| \geq 2$; 
  \item no vertex of $G$ is adjacent to both $a$ and $b$;
  \item  $a,b \not \in \Hub(G)$; and
  \item    $G \setminus N[a]$ is connected.
    \end{itemize}
Let $F \subseteq G \setminus (M \cup \{a,b\})$ be connected.
Then   $a$ is anticomplete to $F$,  and the set of vertices in
$M \cup \{b\}$
with a neighbor in $F$ is local.
\end{theorem}

\begin{proof}
  Let $F$ be a component of $G \setminus (M \cup \{a, b\})$; notice that it suffices to prove Theorem \ref{treestructplus} for such $F$. 
  If $a$ is anticomplete to $F$, the result follows from Theorem~\ref{treestructplusminus}.
  Thus suppose that $a$ has a neighbor in $F$.
  By Theorem~\ref{treestructplusminus} there exists a leaf
  $t$ of $J$ such that $N(F) \subseteq M_t \cup \{a,b\}$.
  Since $G \setminus N[a]$ is connected, it follows that
  $t$ is a $b$-leaf and $N(F) \cap M_t \neq \emptyset$.
    Let    $P$ be a path from $a$ to
  a vertex of $x \in M_t$ with $P^* \subseteq F$.
    Let $Q$ be a path of
    $J$ from  $t$ to an $a$-leaf $t'$.  Concatenating rungs corresponding to
    edges
    of $Q$ in $M$,
    we obtain a path $R$ from $x$ to a vertex $y \in M_t'$.
    Let $H$ be the hole $x \dd R \dd y \dd a \dd P \dd x$.
    Since $H \cup b$ is not a theta, it follows that
    $b$ has a neighbor in $P^*$. Since $(H,b)$ is not a proper wheel
    and $N(a) \cap N(b) = \emptyset$, it follows that
    $b$ has a unique neighbor $x' \in P^*$ and $x'$ is adjacent to $x$
    and non-adjacent to $a$. Let $t''$ be a $b$-leaf in $J$
    such that $t'' \neq t$, and let $S$ be a shortest path in $J$ from $t''$ to
    $Q$. Since $t$ is a leaf, it follows that $t \not\in S$. Concatenating rungs corresponding to
    edges    of $S$ in $M$,
    we obtain a path $T$ from $x'' \in M_{t''}$ to a
    vertex with two (consecutive) neighbors in $R$. But now $H \cup T \cup \{b\}$ is a loaded
    pyramid with loaded corner $b$ and apex $a$, contrary to the fact that $b \not \in \Hub(G)$. This proves Theorem~\ref{treestructplus}.
\end{proof}

\section{From local to global separators}
\label{sec:localglobal}

In this section we prove a theorem that is the heart of the proof of our main result. Qualitatively, the content of this theorem is the following. We
have a graph $G$ whose vertex set is partitioned into two subsets $D$ and $X$ where
$D$ is connected, and $X=N(D)$. We also have a dustinguished vertex $b \in D$.
We are   given a
collection of clique cutsets   separating  individual vertices   of $X$
from $b$ (and with additional properties). The theorem asserts that there is one
cutset, whose clique cover number is bounded from above by an absolute constant,
that separates a positive proportion of the vertices of $X$ from $b$.

Let us now delve into the details.
Let $D$ be graph and let $b \in D$.
Our first goal is to associate to each clique of $D$ a canonical separation.
First we handle cliques that do not contain $b$; in this case the definition
is similar in spirit to other papers in the series.
For every clique     $K  \subseteq D \setminus b$
    let $B(K)$ be the component of $D \setminus K$ with
    $b \in B(K)$.
        Let $C(K)=N(B(K))$ and 
        $A(K)=D \setminus (C(K) \cup B(K))$.
        We call $(A(K), C(K), B(K))$ the
        {\em $b$-canonical separation for $K$.}
                
        Now we  extend the definition of a $b$-canonical separation to all cliques of $D$. 
        Thus let $K \subseteq D$
    be a clique such that $b \in K$. 
    Let $B(K)$ be the union of the components $D'$ of $D \setminus K$ such
    that $b \in N[D']$.
        Let $C(K)=N(B(K))$ and let $A(K)=D \setminus (B(K) \cup C(K))$.
Note that $(A(K),C(K), B(K))=(A(K \setminus b), C(K \setminus b) \cup \{b\}, B(K \setminus b) \setminus \{b\})$.

    Let $\mathcal{K}$
        be the set of all cliques of $D$ with $A(K)$ maximal.
        Let $\beta(D,b)=\bigcap_{K \in \mathcal{K}} (B(K) \cup C(K)$).
        (Some readers may recognize $\beta(D,b)$ as the ``central bag''
        for a collection of separations defined in \cite{TWI}.)
        For every $v \in D \setminus \beta(D,b)$, let $F(v)$ be the
        component of $D \setminus \beta(D,b)$ such that $v \in F(v)$.
The next two lemmas describe some of the properties of $\beta(D,b)$.

        \begin{lemma} \label{laminar}
          We have $b \in \beta(D,b)$.
          For every component $F$  of $D \setminus \beta(D,b)$
          there exists $K \in \mathcal{K}$ such that $F \subseteq A(K)$.
In particular, $N(F(v))$ is a clique for every $v \in D \setminus \beta(D,b)$.
     \end{lemma}

     \begin{proof}
We prove the first assertion of the lemma first.
       \\
       \\
      \sta{$b \in \beta(D,b)$.  \label{binBC}}

      Let  $K \in \mathcal{K}$. Then  $b \in C(K) \cup B(K)$, and
         therefore \eqref{binBC} follows.
       \\
       \\
       \sta{Let $K_1,K_2 \in \mathcal{K}$. Then
         $C(K_1) \cap A(K_2)= C(K_2) \cap A(K_1)=\emptyset$.
         \label{CinB}}
Let $(A_i,C_i,B_i)$ be the canonical separation associated with $K_i$.
       Suppose that there exists $v \in C_1 \cap A_2$. Then $C_1 \cap B_2=\emptyset$.

       First we show that $b \not \in C_1 \cap C_2$. Suppose it is;
then $b \in K_2$, and $b$ has a neighbor in $A_2$, namely $v$, contrary to the definition of a canonical separation. This proves that $b \not \in C_1 \cap C_2$. Since $C_1 \subseteq C_2 \cup A_2$, and $b \in \beta(D,b)$ by
\eqref{binBC}, it follows that $b \not \in C_1$. By \eqref{binBC},  $b \in B_1$; see Figure \ref{fig:new} (left). 

\begin{figure*}[t!] 
    \centering
    \begin{subfigure}[t]{0.5\textwidth}
        \centering
        \includegraphics[height=2.5in]{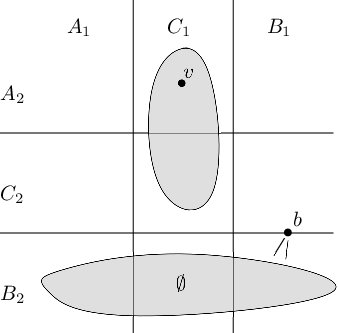}
    \end{subfigure}%
    ~ 
    \begin{subfigure}[t]{0.5\textwidth}
        \centering
        \includegraphics[height=2.5in]{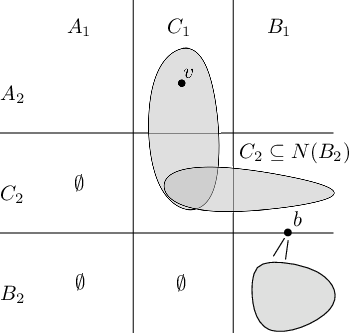}
    \end{subfigure} 
    \caption{Proof of Lemma \ref{laminar}.}
    \label{fig:new}
\end{figure*}

       Since $b \in B_1$, and $\{b\} \cup B_2$ is connected, and $(\{b\} \cup B_2) \cap C_1 = \emptyset$, it follows that $B_2 \subseteq B_1$; see Figure \ref{fig:new} (right). But $C_2 = N(B_2) \subseteq B_1 \cup C_1$, and so  $A_1 \subseteq A_2 \setminus \{v\}$, contrary to the definition of $\mathcal{K}$. This proves \eqref{CinB}.        \\
       \\
       Now let $F$ be a component of $D \setminus \beta(D,b)$.
       Let $K \in \mathcal{K}$ be such that $F \cap A(K) \neq \emptyset$.
       By \eqref{CinB}, $F \cap C(K) = \emptyset$. But now $F \subseteq A(K)$,
       as required. Since $N(F) \subseteq N(A(K)) \subseteq K$, the second assertion follows.
     \end{proof}
       
        \begin{lemma} \label{noncutpt}
          If $K$ is a clique cutset in the graph $\beta(D,b)$, then
          $b \in K$.
                    \end{lemma}

        \begin{proof}
          Suppose that $K$ is a clique cutset of $\beta(D,b)$ and $b \not \in K$. Let $D_1$ be the component of $\beta(D,b) \setminus K$,
          with  $b \in D_1$, and let $D_2=\beta(D,b) \setminus (K \cup D_1)$.
          Since by Lemma~\ref{laminar},  $N(F)$ is a clique for every component $F$ of $D \setminus \beta(D,b)$, it follows that $K$ is a clique cutset in
          $D$ and no component of $D \setminus K$ meets both $D_1$ and
          $D_2$. But then $D_2 \cap A(K) \neq \emptyset$.
          It follows that there exists $K' \in \mathcal{K}$
          such that $D_2 \cap A(K') \neq \emptyset$, contrary to the definition of $\beta(D,b)$.
          \end{proof}

Next we define a {\em breaker} in a graph.
Let $G$ be a graph. Let  $X \subseteq V(G)$ and let
$  X_1,X_2,X_3$ be subsets of $X$ such that $|X_1|=|X_2|=|X_3|$ 
and $X_1,X_2,X_3$ are pairwise disjoint and anticomplete to each other.
Write $D=G \setminus X$ and let $b \in D \setminus N[X]$.
Assume that  $N(D)=X$. 

For $x_1,x_2,x_3 \in X$ let  us say  that 
$x_1x_2x_3$ is {\em partitioned} if $x_i \in X_i$ for every $i \in \{1,2,3\}$.
For a partitioned triple $x_1x_2x_3$ we say that it is {\em $b$-separated}
and {\em $(i,j)$-active} if $b$ is non-adjacent to $x_1, x_2, x_3$ and
there is a clique $K \subseteq D$ such that one of the following holds
(For $i \in \{1,2,3\}$, let
   $D_i$ be the  union of components of $D \setminus K$ such that
  $N(x_i) \cap D_i \neq \emptyset$.)
  \label{secondpart}
  \begin{itemize}
  \item $b \not \in K \cup D_i \cup D_j$; or
\item   We have
  $b \not \in N[D_i]$. Moreover, there is a set $D_j' = D_j'(x_1x_2x_3)$ of vertices such that:
  \begin{itemize}
      \item $D_j'$ is not a clique; 
      \item $D_j'$ is complete
  to $K$; 
      \item There is a vertex $q= q(x_1x_2x_3)$ with the following property. Either $b \in K$ and $q = b$; or there exists $k \in \{1, 2, 3\} \setminus \{i\}$ such that $x_k$ is complete to $K \cup D_j'$ and $q = x_k$; 
      \item For every $v \in D_j'$, there is a path $P$ in $D \cup \{x_j\} $ from $b$ to $x_j$ such that $P \cap N(q) \neq \emptyset$ and the $(P, x_j)$-last vertex in $P \cap N(q)$ is $v$;
      and
      \item For every  path $P$ in $D \cup \{x_j\} $ from $b$ to $x_j$, we have $P \cap N(q) \neq \emptyset$, and the $(P, x_j)$-last vertex of $P \cap N(q)$ is complete to $K$. 
\end{itemize}
\end{itemize}

  Under these circumstances we say that
$K$ is a {\em witness} for $x_1x_2x_3$. We say that $x_1x_2x_3$ is of \emph{type 1} if the first bullet of \eqref{secondpart} holds; otherwise, we say that $x_1x_2x_3$ is of \emph{type 2}. We say an $(i,j)$-active triple $x_1x_2x_3$ is of \emph{type 2a} if $q(x_1x_2x_3) = b$; it is of \emph{type 2b} if $q(x_1x_2x_3) = x_j$; and it is of \emph{type 2c} if $q(x_1x_2x_3) \neq b, x_j$. 
A triple is $b$-{\em separated} if it is $b$-separated and $(i,j)$-active for some distinct $i, j \in \{1,2,3\}$.
Let $\delta \in (0,1]$.
We say that $X$ is a {\em $(\delta ,b)$-breaker} in $G$ if
          there exist at least
  $\delta |X|^3$ partitioned $b$-separated triples.

The main result of this section is the following:
  \begin{theorem}
    \label{localglobal}
    Let  $\delta \in (0,1]$ and let
      $\epsilon \leq \frac{\delta^2}{48 \times 192}$.
        Let $G$ be a $C_4$-free graph,
        and let $X$ be a $(\delta, b)$-breaker in $G$.
        Then 
                    there exists  $S \subseteq D \setminus \{b\}$
            with $\kappa(S) \leq (96/\delta)^2$ such that the component $D(b)$
            of $D \setminus S$ with $b \in D(b)$ 
is disjoint from $N(x)$ for
at least $\epsilon|X|$ vertices $x \in X$.
        \end{theorem}
         \begin{proof} We first show: 
    \begin{align*}
        {\left\lceil \frac{96}{\delta} \right\rceil+1 \choose 2} &= \left(\left\lceil \frac{96}{\delta} \right\rceil\right)\left(\left\lceil \frac{96}{\delta} \right\rceil+1\right)/2\\
        &\leq (96/\delta + 1)(96/\delta+2)/2\\
        &= (96/\delta)^2 + 3\cdot 96/(2\delta) + 1 - (96/\delta)^2/2\\
        &\leq (96/\delta)^2 + 3\cdot 96/(2\delta) + 1 - 48\cdot 96/\delta\\
        &\leq (96/\delta)^2. 
    \end{align*}
    Now suppose that the statement is false.
    For every clique     $K  \subseteq D$,
    let $(A(K), C(K), B(K))$ be the $b$-canonical separation
    for    $K$.
    
                 Let $x_1x_2x_3$ be a $b$-separated triple.
            We fix a witness  $K(x_1,x_2,x_3)$ for $x_1x_2x_3$, 
            where
            $K(x_1,x_2,x_3)$ is chosen such that $x_1x_2x_3$ is of type 1 if possible. 
 Let $\mathcal{F}$ be the  set of  components of $D \setminus \beta(D,b)$.
    For every $x \in X$, let $\mathcal{F}(x)$ be
the set of elements of $\mathcal{F}$ for which $N(x) \cap F \neq \emptyset$.
\\
\\
\sta{Let $x_1x_2x_3$ be a $b$-separated triple; write $K=K(x_1,x_2,x_3)$.
Assume that  the component $D(b)$ of $D \setminus (K \setminus b)$ with $b \in D(b)$
is anticomplete to $x_1$.  Then $x_1$ is anticomplete to
$\beta(D,b) \setminus K$. Moreover, 
let $F \in \mathcal{F}(x_1)$.
Then either $K \cap F \neq \emptyset$, or
$N(F) \subseteq K$. \label{whereisK}}

Since $\beta(D,b) \subseteq D(b) \cup K$, 
it follows that $x_1$ is anticomplete to
$\beta(D,b) \setminus K$.
Next, suppose that $F \cap K = \emptyset$ and that there is a vertex  $p \in N(F) \setminus K$.
By Lemma~\ref{noncutpt}, $b \in \beta(D,b)$ and $\beta(D,b) \setminus (K \setminus b)$ is connected. Let $P$ be a path from $p$ to $b$ with $P^* \subseteq \beta(D,b)$. Let $Q$ be a path from $N(x_1)$ to $p$ with interior in $F$.
Then $R=Q \dd p \dd P \dd p$ is a path from $N(x_1)$ to $b$ with
$R^* \cap K = \emptyset$, a contradiction. This proves \eqref{whereisK}.
\\
\\
It follows from \eqref{whereisK} that:
\\
\\
\sta{If $x_1x_2x_3$ is a $(1,2)$-active triple,
  then $x_1$ is anticomplete to $\beta(D,b) \setminus K(x_1,x_2,x_3)$. Moreover, if $x_1x_2x_3$ is of type 1, then $x_2$ is anticomplete to
  $\beta(D,b) \setminus K(x_1,x_2,x_3)$.
  \label{cleanbeta}}
\\
\\
Let $x \in X$. We define
the {\em projection} of $x$, denoted by $\Proj(x)$, to be
$N_{\beta(D,b)}(x) \cup \bigcup_{F \in \mathcal{F}(x)}N(F)$. Note that since $b$ is non-adjacent to $x_1, x_2, x_3$ for $b$-seperated triples $x_1x_2x_3$, and since $b$ is anticomplete to $A(K)$ for every $K \in \mathcal{K}$, it follows that $b \not\in \Proj(x)$ whenever $x$ is in a $b$-separated triple.
\\
\\
\sta{Let $x_1x_2x_3$ be a $(1,2)$-active triple; write $K=K(x_1,x_2,x_3)$. If $x_1x_2x_3$ is of type 1, then $\Proj(x_1) \cup \Proj(x_2)$
  is a clique. \label{completeproj}}

For $i\in \{1, 2\}$, let $F_i=\bigcup_{F \in \mathcal{F}(x_i)}F$.
Suppose first that $K \cap F_i= \emptyset$ for all $i \in \{1, 2\}$;
then by \eqref{whereisK} $N(F_i) \subseteq K$ for $i\in \{1, 2\}$.
Since by \eqref{cleanbeta}, $N_{\beta(D,b)}(x_i) \subseteq K$ for $i \in \{1, 2\}$,
\eqref{completeproj} holds. Thus we may assume that
there exists $F \in \mathcal{F}(x_1)$ such that $K \cap F \neq \emptyset$.
Let $F' \in (F_1 \cup F_2) \setminus F$. Since $K$ is a clique, it
follows that $F' \cap K = \emptyset$, and by \eqref{whereisK}, $N(F') \subseteq K$.  

Since $K \cap F \neq \emptyset$, we have that $K \cap \beta(D, b) \subseteq N(F)$, and so 
$N(F') \subseteq N(F)$.
Moreover, by \eqref{cleanbeta},  for $i \in \{1, 2\}$,  we have $N_{\beta(D,b)}(x_i) \subseteq K$, and so  $N_{\beta(D,b)}(x_i) \subseteq N(F)$.
$N(F)$ is a clique by Lemma \ref{laminar},  and \eqref{completeproj}
follows.
\\
\\
\sta{Let $x_1x_2x_3$ be a $(1,2)$-active triple which is of type 2; write $K=K(x_1,x_2,x_3)$. Then $\Proj(x_1) \subseteq K$. \label{completeproj2}}

By \eqref{cleanbeta}, we have $N_{\beta(D, b)}(x_1) \subseteq K$. By \eqref{whereisK}, we either have $N(F) \subseteq K$ or $K \cap F \neq \emptyset$ for every $F \in \mathcal{F}(x_1)$. If the former holds for all $F \in \mathcal{F}(x_1)$, then \eqref{completeproj2} holds; so we may assume that there exists $F \in \mathcal{F}(x_1)$ with $K \cap F \neq \emptyset$. By Lemma \ref{laminar}, it follows that $N(F)$ is a clique $K'$. From the definition of $\beta(D, b)$, it follows that $b \not\in K'$. We claim that $K'$ is a witness for $x_1x_2x_3$ that makes
$x_1x_2x_3$ be of type 1 (and therefore contradicts the choice of $K = K(x_1x_2x_3)$). Suppose not; let $P$ be a path from $b$ to $x_i$ for some $i \in \{1, 2\}$ such that $P^* \cap K' = \emptyset$. From the definition of a $b$-separated $(1,2)$-active triple, it follows that $P^* \cap (K \cup D_2'(x_1x_2x_3)) \neq \emptyset$. Since $D_2'(x_1x_2x_3)$ is complete to $K$, it follows that $(K \cup D_2'(x_1x_2x_3)) \setminus K' \subseteq F$. Therefore, $P^* \cap F \neq \emptyset$. Since $P^* \cap N(F) = \emptyset$, it follows that $P \subseteq N[F] \cup \{x_i\}$, contrary to the fact that $b \not\in N[F] \cup \{x_i\}$. This is a contradiction, and proves \eqref{completeproj2}. 

\vspace*{0.3cm}

By permuting the indices if necessary, we may assume that
$\frac{\delta}{6} |X|^3$ separated triples are $(1,2)$-active.
Now, for one of the four possible types (1, 2a, 2b, 2c),  there exist
$\frac{\delta}{24}|X|^3$ distinct $(1,2)$-active triples $x_1x_2x_3$ of this type with respect to $K(x_1x_2x_3)$. Let $l$ be the first entry of the list (1, 2a, 2b, 2c) for which this is the case; let us say that a triple $x_1x_2x_3$ is
{\em manageable} if it is $(1,2)$-active of type $l$.

   Let $Z_1 \subseteq X_1$ be the set of all vertices $y_1 \in X_1$
  such that $$|\{(y_2, y_3) \in X_2 \times X_3 : y_1y_2y_3 \textnormal{ is a manageable triple}\}| \geq \frac{\delta}{48}|X|^2.$$
  
  \sta{$|Z_1| \geq \frac{\delta}{48}|X|$. \label{Z1}}
  
    Suppose not. Each vertex $y_1 \in X_1 \setminus Z_1$ is in fewer than $\frac{\delta}{48}|X|^2$ manageable triples. Therefore, the total number of manageable triples is less than 
    $$|Z_1||X|^2 + |X_1 \setminus Z_1|\frac{\delta}{48}|X|^2 < \frac{\delta}{48}|X|^3 + \frac{\delta}{48}|X|^3 = \frac{\delta}{24}|X|^3,$$
    a contradiction. This proves \eqref{Z1}.
  \\
  \\
  Recall that by \eqref{completeproj},
  $\Proj(w)$ is a clique for every $w \in Z_1$.
  Since $|Z_1|>\epsilon |X|$, we have $\alpha(\bigcup_{w \in Z_1} \Proj(w)) >  \frac{96}{\delta}$.
  
Let $W_1$ be a minimal subset of $Z_1$ such that 
$\alpha(\bigcup_{w \in W_1} \Proj(w)) \geq  \frac{96}{\delta}$.
Then $\alpha(\bigcup_{w \in W_1} \Proj(w)) = \left\lceil \frac{96}{\delta} \right\rceil$.
  Let $S=\bigcup_{w \in W_1} \Proj(w)$.
By a theorem of \cite{wagon} (using that the complements of even-hole-free graphs contain no induced two-edge matching), we have $\kappa(S) \leq {\alpha(S)+1 \choose 2} \leq  {\left\lceil \frac{96}{\delta} \right\rceil+1 \choose 2}$.

   Let $J$ be a stable set of size $\left\lceil \frac{96}{\delta} \right\rceil$ in $S$.
    For every $j \in J$, let $x(j) \in W_1$ be such that
    $j \in \Proj(x(j))$. Since $J$ is a stable set, the elements
    $x(j)$ are pairwise distinct.

    Let $H$ be the bipartite graph with bipartition     $(\{x(j)\}_{j \in J}, X_2 \times X_3)$, such that $x(j)$ is adjacent to
$(y_2, y_3) \in X_2 \times X_3$ if  $x(j)y_2y_3$ is a manageable triple.
\\
\\
\sta{There exist  distinct $j_1,j_2 \in J$ such that
  $|N_H(x(j_1)) \cap N_H(x(j_2))| \geq \frac{\delta^2}{48 \times 192}|X|^2$.
  \label{N2}}

Suppose not. Let $j \in J$. Since $x(j) \in Z_1$, it follows that
$x(j)$ has at least $\frac{\delta}{48}|X|^2$ neighbors in $H$.
For every $j \in J$, let
$M(j)$ be the  set of vertices $y \in Y_2$ such that
$y$ is adjacent to $x(j)$ in $H$, and $y$ is not adjacent in
$H$ to any other vertex $x(j')$ for $j' \neq j$.
  Since  $|N_H(x(j_1)) \cap N_H(x(j_2))| < \frac{\delta^2}{48 \times 192}|X|^2$ for all distinct $j_1, j_2 \in J$,
and since each $x(j)$ has
at least $\frac{\delta}{48}|X|^2$ neighbors in $H$,
it follows that $$|M(j)| > \frac{\delta}{48}|X|^2 - \left(\left\lceil \frac{96}{\delta} \right\rceil-1\right)\frac{\delta^2}{48 \times 192}|X|^2 > \frac{\delta}{96}|X|^2$$  for each
$j \in J$.
But now $\bigcup_{j \in J}|M(j)|>|J| \frac{\delta}{96}|X|^2 \geq |X|^2$, a contradiction.
This proves \eqref{N2}.
\\
\\
Suppose that $l=1$.
Let $j_1,j_2$ be as in \eqref{N2}, and let $$Z_2=\{y_2: (y_2, y_3) \in N_H(x(j_1)) \cap N_H(x(j_2)) \textnormal{ for some } y_3 \in X_3\}.$$
Then $|Z_2| \geq \frac{\delta^2}{48 \times 192}|X|$. Since $|Z_2| \geq \epsilon |X|$, we deduce that
$\kappa (\bigcup_{y \in Z_2} \Proj(y))> {\left\lceil \frac{96}{\delta} \right\rceil+1 \choose 2} > 4$.
Therefore we can choose non-adjacent
$k_1,k_2 \in (\bigcup_{y \in Z_2} \Proj(y)) \setminus (\Proj(x(j_1)) \cup \Proj(x(j_2))).$
For $i \in \{1,2\}$, let $y(k_i) \in Z_2$ be such
$k_i \in \Proj(y_i)$.
It follows that for every $p,q \in \{1,2\}$ there exists
$y_3(p,q) \in Y_3$ such that $x(j_p),y(k_q),y_3(p,q)$ is a
manageable triple.
Now applying \eqref{completeproj} to
$x(j_p)y(k_q)y_3(p,q)$, we deduce  that $j_p$ is 
adjacent to $k_q$; consequently
$\{j_1,j_2\}$ is complete to $\{k_1,k_2\}$.
But then
$j_1 \dd k_1 \dd j_2 \dd k_2 \dd j_1$ is a $C_4$ in $G$, a contradiction.

This proves that $l$ is one of 2a, 2b, 2c. Let $j_1,j_2$ be as in \eqref{N2} and let
$(y_2, y_3) \in N_H(x(j_1)) \cap N_H(x(j_2))$. Note that since $x(j_1)y_2y_3$ and $x(j_2)y_2y_3$ are both $(1, 2)$-active and of the same type in 2a, 2b, 2c, we have $q(x(j_1)y_2y_3) = q(x(j_2)y_2y_3)$. Let us define $q = q(x(j_1)y_2y_3)$. For $i \in \{1, 2\}$, let $K_i = K(x(j_i)y_2y_3)$. By \eqref{completeproj2} we have $\Proj(x(j_1)) \subseteq K_1$ and so $j_1 \in K_1$; likewise, $j_2 \in K_2$.

Let $p_1, p_2 \in D_2'(x(j_1)y_2y_3)$ be non-adjacent. For $i \in \{1, 2\}$, let $P_i$ be a path from $b$ to $y_2$ in $D \cup \{y_2\}$ such that the $(P_i, y_2)$-last vertex in $P_i \cap N(q)$ is $p_i$. It follows that $p_1, p_2$ are complete to $K_1, K_2$. But now $j_1 \dd p_1 \dd j_2 \dd p_2$ is a $C_4$ in $G$, a contradiction.
\end{proof}

  \section{Handling dangerous triples} \label{sec:dangerous}
  Let $G$ be a graph, let $a \in G$ and write $X=N(a)$
and  $D=G \setminus N[a]$.
  Let $X$ be partitioned into three equal-size subsets $X_1,X_2,X_3$ pairwise anticomplete to each other.
We say that the triple $x_1x_2x_3$ with $x_i \in X_i$ is
{\em dangerous with center $x_2$} if
the edge $ax_2$ is a cross-edge of an extended near-prism in the graph
$D \cup \{x_1,x_2,x_3, a\}$.

The goal of this section is to prove that if $G$ contains many dangerous triples with a fixed center, then we can bypass the main argument of Section~\ref{sec:sepab_nohubs} and obtain the desired conclusion directly. The details of this are explained in Section~\ref{sec:sepab_nohubs}. 

We need the following definition:  Let $G=D \cup X \cup \{a\}$ be a graph  where $D$ is connected, $a$ is complete to $X$ and anticomplete to  $D$, and $N(D)=X$. Let $b \in D$. Let us say that $X' \subseteq X$ is {\em pure} if there does not exist a hole $H \subseteq D \cup X' \cup  \{a\}$ such that $a,b \in H$.
\begin{theorem} \label{dangerous}
  Let $\delta \in (0,1]$ and let $\epsilon \leq \frac{1}{8}\delta$.
  Let $G=D \cup X \cup \{a\}$ be a graph  where $D$ is connected, $a$ is complete to $X$ and anticomplete to
  $D$, and $N(D)=X$.
  Assume that $X \cap \Hub(G)= \emptyset$.
  Assume that $X$ is partitioned into three pairwise anticomplete
  sets $X_1,X_2,X_3$ of equal size. Suppose that some  $x_2 \in X_2$
  is a center of  $\delta |X|^2$ dangerous triples.
  Assume also that there is no clique of size $\epsilon |X|$ in $X$.
   Let $b \in D$. Then one of the following holds: 
    \begin{itemize}
 \item there exists $S \subseteq D \setminus \{b\} $ with $\kappa(S) \leq 4$
such that  the component $D(b)$    of $D \setminus S$ with $b \in D(b)$
 is disjoint from $N(x)$ for
 at least $\epsilon|X|$ vertices $x \in X$;
 or
 \item $X$ is not pure and there exists $X' \subseteq X$ with $|X'| \geq \frac{1-4 \epsilon}{2} |X|$
 such that $X'$ is pure.
 \end{itemize}
\end{theorem}

\begin{proof}
  Suppose not.
  Let $x_2 \in X_2$.
For $i=1,3$ let $Y_i \subseteq X_i$ be the set of all $y \in X_i$
such that there exist at least $\frac{1}{2} \delta |X|$ elements $z \in X_{4-i}$ for which $yx_2z$ is a dangerous triple with center $x_2$.
\\
\\
\sta {For $i=1,3$,   $|Y_i| \geq \frac{1}{2}\delta|X|$. \label{bigYi}}

  Suppose that $|Y_1| < \frac{1}{2}\delta |X|$. 
  The number of dangerous triples with center $x_2$ and using  an element of $Y_1$ is at
  most $|Y_1| |X| \leq \frac{1}{2} \delta |X|^2$.
  The number of dangerous triples with center $x_2$ and  not using an element of $Y_1$ is less than $|X_1 \setminus Y_1| \times \frac{1}{2}\delta|X| \leq \frac{1}{2} \delta |X|^2$. It follows that the number of dangerous triples with center $x_2$ is
  less than  $\delta|X|^2$, a contradiction.
  This proves~\eqref{bigYi}.
  \\
  \\
  Let $G'=D \cup X_1 \cup X_3 \cup \{a,x_2\}$. Then $ax_2$ is a cross-edge
  of an extended near-prism in $G'$ and  no vertex of $G'$ is
  adjacent to both $a$ and $x_2$.
  Since $N(D)=X$, it follows from Theorem~\ref{hubstarcutset} that $a \not \in \Hub(G)$.
  Applying Theorem~\ref{treestructplus} to $ax_2$ and $G'$,
  we obtain a $J$-strip system $M$ with cross-edge $ax_2$ such that
  for every connected induced subgraph $F$ of $G'\setminus (M \cup \{a,x_2\})$, we have that
  $a$ is anticomplete to $F$,  and the set of vertices in $M \cup \{x_2\}$
with a neighbor in $F$ is local.
\\
\\
\sta{For every
$x \in X_1 \cup X_3$, there exists an $a$-leaf $t$ 
 such that $x \in M_t$. \label{whereisX}}

Since $X_1 \cup X_3 \subseteq N(a)$, Theorem~\ref{treestructplus} implies that $X_1 \cup X_3 \subseteq V(M)$. Now \eqref{whereisX} follows from
the fact that $ax_2$ is a cross-edge for $M$.
\\
\\
Let $v \in V(J)$. Let $e$ be an edge of $J$ incident with $v$. We say that
$e$ is {\em special for $v$} if either $v$ is a leaf, or
the set $M_v \cap M_e$ is not a clique.
Since $G'$ is $C_4$-free and $M_v \cap M_e$ is complete to $M_v \cap M_{e'}$ for distinct edges $e, e'$ incident with $v$, it follows that for every $v \in V(J)$, there is at most one
special edge for $v$.

For $v \in V(J)$, let $F_v$ be, the union  of components $F$ of
$G' \setminus (M \cup \{a,x_2\})$ such that $N(F) \cap M \subseteq M_v$.
For $e \in E(J)$ with ends $u, v$ let $F_e$ be the union of components $F$ of
$G' \setminus (M \cup \{a,x_2\})$ such that $N(F) \cap M \subseteq M_e$, and such that
$N(F) \not\subseteq M_u$ and $N(F) \not\subseteq  M_v$. Since $D$ is connected and disjoint from $N(a)$, it follows that $F_t= \emptyset$ for every $a$-leaf $t$. 

Let $e=uv$ be an edge of $J$. Define $\mu(e)$ as follows.
If $e$ is not special for either $u$ or $v$, let
$\mu(e)=F_e \cup M_e \setminus (M_u \cup M_v)$. If $e$
is special for $u$ and not for $v$, let
$\mu(e)=F_e \cup (M_e \setminus M_v) \cup F_u$. If $e$
is special for $v$ and not for $u$, let
$\mu(e)=F_e \cup (M_e \setminus M_u) \cup F_v$.
If $e$ is special for both $u$ and $v$,
let $\mu(e)=M_e \cup F_e \cup F_u \cup F_v$.
In all cases let $\nu(e)=\emptyset$.
Next, let $v \in V(J)$.
If there is a special edge $e$ for $v$,
let $\mu(v)=M_v \setminus M_e$ and $\nu(v)=\emptyset$.
If no edge is special for $v$,
let $\mu(v)=M_v$ and $\nu(v)=F_v$. It follows that $\mu(v)$ is a clique for every $v \in V(J)$. 
Note that each vertex of $V(G') \setminus \{x_2, a\}$ is in at least one set in $\{\mu(x), \nu(x)\}_{x \in V(J) \cup E(J)}$; the only vertices of $V(G') \setminus \{x_2, a\}$ that are in two such sets are vertices in $M_u \cap M_v$ where $uv$ is not special for $u$ and not special for $v$; they are in both $\mu(v)$ and $\mu(u)$. 
\\
\\
\sta{If $e$ is an edge of $J$ with ends $u,v$ and $b \in \mu(e)$, then
  $|(X_1 \cup X_3) \setminus (\mu(e) \cup \mu(v) \cup \mu(u)) | \leq
  \epsilon |X|$. \label{binedge}} 
Suppose $x \in (X_1 \cup X_3) \setminus (\mu(e) \cup \mu(v) \cup \mu(u))$.
Then $N(x) \cap \mu(e) = \emptyset$. Since
$L=\mu(v) \cup \mu(u)$ separates $\mu(e)$ from
$D \setminus (\mu(e) \cup L)$ in $D$, it follows that
the component of $D \setminus L$ that contains $b$ is disjoint
from $N(x)$. Since $\kappa(L) \leq 2$,  there are at most $\epsilon|X|$ such
vertices $x$, and \eqref{binedge} follows.
\\
\\
\sta{Suppose that $v \in V(J)$ and $b \in \mu(v)$.
  Let $e=vu \in E(J)$ such that $b \in M_e$. Moreover: 
  \begin{itemize}
      \item Define $L_1$ and $C_1$ as follows. Let $f = vw$ be a special edge at $v$ if one exists; in this case, let $L_1 = \mu(w) \cup \mu(f)$ and $C_1 = \mu(w)$; otherwise, $L_1 = C_1 = \emptyset$. 
      \item Define $L_2$ and $C_2$ as follows. If $b \in M_u \cap M_v$ and there is a special edge $uz$ at $u$ with $z \neq v$, let $L_2 = \mu(uz) \cup \mu(z)$ and $C_2 = \mu(z)$. If $b \in M_u \cap M_v$ and there is no special edge at $u$ except possibly $uv$, we let $L_2 = \nu(u)$ and $C_2 = \emptyset$. Finally, if $b \not\in M_u$, let $L_2 = C_2 = \emptyset$.
  \end{itemize}
  Then
    $|(X_1 \cup X_3) \setminus (\mu(e) \cup \mu(v) \cup L_1 \cup L_2\cup \mu (u) \cup \nu(v))| \leq \epsilon|X|$. \label{binpotato}}

Suppose that $x \in (X_1 \cup X_3) \setminus (\mu(e) \cup \mu(v) \cup L_1 \cup L_2\cup \mu (u) \cup \nu(v))$. 
Let $L=((\mu(v) \cup \mu(u) \cup C_1 \cup C_2) \setminus \{b\}) \cap D$.
Then $L$ separates $(\{b\} \cup \mu(e) \cup \nu(v) \cup (L_1 \setminus C_1) \cup (L_2 \setminus C_2)) \cap D$
from the rest of $D$, and in particular, from $D \cap (N(x) \setminus L)$.  Since $\kappa(L) \leq 4$, there are at most $\epsilon|X|$ such vertices $x$, and \eqref{binpotato} follows.
\\
\\
\sta{Suppose that $v \in V(J)$ and $b \in \nu(v)$.
    Then
    $|(X_1 \cup X_3) \setminus (\mu(v) \cup \nu(v)) |   \leq \epsilon|X|$. \label{binnode}}

Suppose that $x \in (X_1 \cup X_3) \setminus (\mu(v) \cup \nu(v))$.
Then $N(x) \cap \nu(v) =\emptyset$.
Let $L=\mu(v)$.
Then $L$ separates $\nu(v)$ from $D \setminus (\nu(v) \cup \mu(v))$.
We deduce  that the component of $D \setminus L$ that contains $b$ is disjoint
from $N(x)$. Since $\kappa(L) =1$, there are at most $\epsilon|X|$ such vertices $x$, and \eqref{binnode} follows.
\\
\\
It follows from \eqref{whereisX}, 
\eqref{binedge}, \eqref{binpotato} and \eqref{binnode} that
there is an $a$-leaf $t$  with $N_J(t)=\{t'\}$ such that
$b \in \mu(t't) \cup M_{t'} \cup \nu(t')$. If $b \in \nu(t')$, then by \eqref{whereisX} and \eqref{binnode}, it follows that $|X \cap \mu(t')| \geq (1-\epsilon) |X|$. But $\mu(t')$ is a clique, and $\epsilon \leq 1/8$, and $X$ contains no clique of size $\epsilon|X|$, a contradiction. It follows that $b \in \mu(t't) \cup M_{t'}$. 

From Theorem \ref{treestructplus}, it follows that: 

\sta{If $x \in N(x_2) \cap M$, then there is a $b$-leaf $q$ such that $x \in M_q \cup F_q$. \label{wherearenbrsofx2}}

Next, we show: 

\sta{Let $t$ be an $a$-leaf, and let $t'$ be the unique neighbor of $t$ in $J$. Let $Z = \mu(tt') \cup \nu(t')$. 
Suppose that $y_1 \in X_1 \cap Z$ and $y_3 \in X_3 \cap Z$. Then $y_1x_2y_3$ is not a dangerous triple. \label{lessdangerous}}

Suppose not. Then, by the definition of a dangerous triple, there exists a path $R$ from $y_1$ to $y_3$ with $x_2 \in R$ and such that
$R \setminus \{y_1,x_2,y_3\} \subseteq D$. By \eqref{wherearenbrsofx2}, it follows that there exist $b$-leaves $q_1, q_3$ such that the neighbor of $x_2$ on the subpath of $R$ from $x_2$ to $y_i$ is in $M_{q_i} \cup F_{q_i}$ for $i \in \{1, 3\}$. Since $\mu(t')$ separates
$Z$ from $F_r \cup M_r$ for all $r \in V(J) \setminus \{t, t'\}$, it follows that the interiors of
both the paths $y_1 \dd R \dd x_2$ and $x_2 \dd R \dd y_3$ meet $\mu(t')$,
contrary to the fact that $\mu(t')$ is a clique. This proves \eqref{lessdangerous}. 
\\

\sta{Let $t$ be an $a$-leaf, and let $t'$ be the unique neighbor of $t$ in $J$. Let $t''$ be a neighbor of $t$ with $t'' \neq t$. Let $Z = \mu(tt') \cup \mu(t't'') \cup \mu(t') \cup \mu(t'')\cup \nu(t') \cup \nu(t'')$. 
Then $|X \cap Z| < (1-\epsilon)|X|. $ \label{oneleaf}}

From the definition of a $J$-strip structure with cross-edge $ax_2$, it follows that $t''$ is not an $a$-leaf. 

Now,
since $X$ contains no clique of size $\epsilon |X|$, it follows that $|X \cap \mu(t') \cup \mu(t'')| \leq 2 \epsilon|X|$. Furthermore, \eqref{whereisX} implies that $X \cap \nu(t') = \emptyset$, $X \cap \nu(t'') = \emptyset$ and $X \cap \mu(t't'') = \emptyset$ (as $t'', t'$ are not $a$-leaves). 
Therefore, $|(X_1 \cup X_3) \cap  \mu(tt')|  \geq (1-3 \epsilon)|X|$. Since $\epsilon < \frac{1}{8} \delta$, it follows from \eqref{bigYi} that
there exist $y_1 \in Y_1 \cap (Z \setminus \mu(t'))$ and $y_3 \in X_3  \cap (Z \setminus \mu(t'))$
such that $y_1x_2y_3$ is a dangerous triple, contrary to \eqref{lessdangerous}. This proves \eqref{oneleaf}. 
\\

To finish the proof of Theorem \ref{dangerous}, we consider three cases. Suppose first that either:
\begin{itemize}
    \item $b \in \mu(tt')$; or
    \item $b\in M_{tt'} \cap M_{t'}$ and no edge is special at $t'$ (except possibly $tt'$). 
\end{itemize}
Let $Z = \mu(tt') \cup \mu(t')$ if $b \in \mu(tt')$, and let $Z = \mu(tt') \cup \mu(t') \cup \nu(t')$ otherwise. By \eqref{binedge} and \eqref{binpotato} and since $\mu(t) = \nu(t) = \emptyset$, it follows that $|X \cap Z| \geq (1-\epsilon)|X|$. This contradicts \eqref{oneleaf} (choosing $t''$ arbitrarily).

Now suppose that $b \in (M_{t'} \cap M_{tt'}) \setminus \mu(tt')$. We may assume that we are not in the first case, and so it follows that there is an edge $t't''$ with $t'' \neq t$ which is special at $t'$. It follows that $t''$ is not an $a$-leaf. Let $Z = \mu(tt') \cup \mu(t') \cup \mu(t't'') \cup \mu(t'') \cup \nu(t')$. 
By \eqref{binpotato} and since $\nu(t) = \mu(t) = \emptyset$, it follows that $|X \cap Z| \geq (1-\epsilon)|X|$. This contradicts \eqref{oneleaf}.

It follows that $b \in M_{t'} \cap M_{t't''}$ for some neighbor $t''$ of $t'$ with $t'' \neq t$. Then $t''$ is not an $a$-leaf. 
Suppose first that either:
\begin{itemize}
    \item $b \not\in M_{t''}$; or
    \item $b \in M_{t''} \cap M_{t'}$ and there is no special edge at $t''$ except possibly $t''t'$. 
\end{itemize}
 Let
$Z= \mu(tt') \cup \mu(t't'') \cup \mu(t') \cup \mu(t'')\cup \nu(t') \cup \nu(t'')$.
 Using \eqref{binpotato} (with $t' = v$; $t = w$; $t'' = u$), we conclude that $|X \cap Z| \geq (1-\epsilon )|X|$. Again, this contradicts \eqref{oneleaf}.

It follows that $b \in M_{t'} \cap M_{t''}$, and there is a special edge $t''t'''$ at $t''$ with $t''' \neq t'$. Let
$Z= \mu(tt') \cup \mu(t't'') \cup \mu(t') \cup \mu(t'')\cup \nu(t') \cup \nu(t'') \cup \mu(t''t''') \cup \mu(t''')$. Using \eqref{binpotato} (with $t' = v$; $t = w$; $t'' = u$; $t''' = z$), we conclude that $|X \cap Z| \geq (1-\epsilon )|X|$. From \eqref{whereisX}, it follows that $X \cap (\nu(t') \cup \nu(t'')) = \emptyset$.  Since $X$ contains no clique of size at least $\epsilon|X|$, it follows that $|X \cap (\mu(t') \cup \mu(t'') \cup \mu(t'''))| < 3 \epsilon |X|$. Neither $t'$ nor $t''$ is a leaf, and so $X \cap \mu(t't'') = \emptyset$ by \eqref{whereisX}.

It follows that $|X \cap (\mu(tt') \cup \mu(t''t'''))| > (1-4\epsilon)|X|$. If $X \cap \mu(t''t''') = \emptyset$, then as before, there exist $y_1 \in \mu(tt') \cap Y_1$ and $y_3 \in \mu(tt') \cap Y_3$ such that $y_1x_2y_3$ is a dangerous triple, contrary to \eqref{lessdangerous}. So $X \cap \mu(t''t''') \neq \emptyset$. It follows that $t'''$ is an $a$-leaf. There is symmetry (switching $t, t', t'', t'''$ with $t''', t'', t', t$), and so $X \cap \mu(tt') \neq \emptyset$. 

Let $x \in X \cap \mu(tt')$ and let $x' \in X \cap \mu(t''t''')$. Then $x \in M_t$ and $x' \in M_{t'''}$ by \eqref{whereisX}. Let $R$ be a $tt'$-rung containing $x$, and let $R'$ be a $t''t'''$-rung containing $x'$. Then $x \dd R \dd b \dd R' \dd x' \dd a \dd x$ is a hole in $D \cup X \cup \{a\}$ containing $a$ and $b$, and so $X$ is not pure. 

By symmetry, we may assume that $|X \cap \mu(tt')| \geq \frac{1-4\epsilon}{2}|X|$. Write $X' = X \cap \mu(tt')$. We claim that $X'$ is pure (and so the second outcome of the theorem holds). Suppose not; let $H$ be a hole containing $a$ and $b$ with $H \setminus a \subseteq X' \cup D$. Then $H$ contains two internally disjoint paths from $b$ to $a$, say $P_1$ and $P_2$. Since $Y = M_{t'} \cap M_{tt'}$ separates $b$ from $X' \setminus Y$ in $D \cup X'$, it follows that $P_1^*$ and $P_2^*$ each contain a vertex in $Y$ (and in particular, $tt'$ is special at $t'$). Since $\mu(t')$ is complete to $Y$, it follows that $H \cap \mu(t') = \{b\}$. Consequently, $H \setminus b \subseteq \mu(tt')$. It follows that $R' \setminus b$ is anticomplete to $H \setminus b$. But now $H \cup R'$ is a theta in $G$ with ends $a, b$, a contradiction. This concludes the proof. 
\end{proof}

    \section{Separating a pair of vertices: the hub-free case}
    \label{sec:sepab_nohubs}

    In this section we set $\epsilon=\frac{1}{4 \times 17^6 \times 48 \times 192}$, $\gamma=\frac{\epsilon(1-4 \epsilon)}{2}$ 
    and
$C=96^2 \times 4 \times 17^6+4$.
The goal of this section is to prove the following:
\begin{theorem}
  \label{ablogn}
  Let $G \in \mathcal{C}$ with $|V(G)|=n$, and let $a,b \in V(G)$ be
  non-adjacent. Assume that $N(a) \cap \Hub(G)=\emptyset$.
  Then there is a set $Z \subseteq V(G) \setminus \{a,b\}$
  with $\kappa(Z) \leq -C \frac{1}{\log(1-\gamma)}  \log n$ and such that every component of
  $G \setminus Z$ contains at most one of $a,b$.
  \end{theorem}

We need the following result from \cite{prismfree}.
\begin{lemma}[Abrishami, Chudnovsky, Dibek, Vu\v{s}kovi\'c \cite{prismfree}]\label{minimalconnected}
Let $x_1, x_2, x_3$ be three distinct vertices of a graph $G$. Assume that $H$ is a connected induced subgraph of $G \setminus \{x_1, x_2, x_3\}$ such that $V(H)$ contains at least one neighbor of each of $x_1$, $x_2$, $x_3$, and that $V(H)$ is minimal subject to inclusion. Then, one of the following holds:
\begin{enumerate}[(i)]
\item For some distinct $i,j,k \in  \{1,2,3\}$, there exists $P$ that is either a path from $x_i$ to $x_j$ or a hole containing the edge $x_ix_j$ such that
\begin{itemize}
\item $V(H) = V(P) \setminus \{x_i,x_j\}$; and
\item either $x_k$ has two non-adjacent neighbors in $H$ or $x_k$ has exactly two neighbors in $H$ and its neighbors in $H$ are adjacent.
\end{itemize}

\item There exists a vertex $a \in V(H)$ and three paths $P_1, P_2, P_3$, where $P_i$ is from $a$ to $x_i$, such that 
\begin{itemize}
\item $V(H) = (V(P_1) \cup V(P_2) \cup V(P_3)) \setminus \{x_1, x_2, x_3\}$;  
\item the sets $V(P_1) \setminus \{a\}$, $V(P_2) \setminus \{a\}$ and $V(P_3) \setminus \{a\}$ are pairwise disjoint; and
\item for distinct $i,j \in \{1,2,3\}$, there are no edges between $V(P_i) \setminus \{a\}$ and $V(P_j) \setminus \{a\}$, except possibly $x_ix_j$.
\end{itemize}

\item There exists a triangle $a_1a_2a_3$ in $H$ and three paths $P_1, P_2, P_3$, where $P_i$ is from $a_i$ to $x_i$, such that
\begin{itemize}
\item $V(H) = (V(P_1) \cup V(P_2) \cup V(P_3)) \setminus \{x_1, x_2, x_3\} $; 
\item the sets $V(P_1)$, $V(P_2)$ and $V(P_3)$ are pairwise disjoint; and
\item for distinct $i,j \in \{1,2,3\}$, there are no edges between $V(P_i)$ and $V(P_j)$, except $a_ia_j$ and possibly $x_ix_j$.
\end{itemize}
\end{enumerate}
\end{lemma}

We also need the following; for a proof see, for example, \cite{TWX}:

\begin{theorem}
  \label{center}
  Let $(T,\chi)$ be a tree decomposition of a graph $G$. Then there exist
  a vertex $t_0 \in T$ such that 
  $|D| \leq \frac{1}{2} |V(G)|$ for every component $D$ of $G \setminus \chi(t_0)$.
  \end{theorem}

We start with a lemma.

\begin{lemma}
  \label{chordal}
  Let $n$ be an integer.
  Let $G$ be a  chordal graph with $n-\epsilon n$ vertices, and
  assume that $G$ has no clique of size $\epsilon n$.
  Then there is $Z \subseteq V(G)$
  such that 
  \begin{itemize}
    \item $\kappa(Z) \leq 2$, and 
\item    there exist subsets $X_1,X_2,X_3$ of $V(G) \setminus Z$,  
  pairwise disjoint and anticomplete to each other, and
  such that $|X_i|=\left\lceil\frac{1}{17} n\right\rceil $ for every $i \in \{1,2,3\}$.
\end{itemize}
  \end{lemma}

\begin{proof}
  Since $G$ is chordal, there is a tree decomposition $(T, \chi)$ of $G$
  such that $\chi(t)$ is a clique for every $t \in T$ \cite{gavril}. 
  By Theorem~\ref{center}, there exists a vertex $t_0 \in T$ such that
  $|D| \leq \frac{n}{2}$ for every component $D$ of $G \setminus \chi(t_0)$.
  Let $X$ be a minimal set of components of $G \setminus \chi(t_0)$
  such that $|\bigcup_{D \in X}D| \geq (\frac{1}{4} -2\epsilon)n$.
  \\
  \\
  \sta{$|G \setminus (\bigcup_{D \in X}D \cup \chi(t_0))| \geq \frac{1}{4}n$.
    \label{otherside}}

  Suppose not. Then $|\bigcup_{D \in X}D \cup \chi(t_0)| > (\frac{3}{4}- \epsilon)n$.
  Since $\chi(t_0)$ is a clique, it follows that $|\chi(t_0)|< \epsilon n$,
  and so $|\bigcup_{D \in X}D| \geq (\frac{3}{4} -2 \epsilon)n$.
  Let $D_0 \in X$. Then $|D_0| \leq \frac{1}{2}n$, and so
  $|\bigcup_{D \in X \setminus \{D_0\}}D| \geq (\frac{1}{4}-2 \epsilon)n$,
  contrary to the minimality of $X$. This proves~\eqref{otherside}.
  \\
  \\
  Let $X_1= \bigcup_{D \in X}D$ and let $Z_1=\chi(t_0)$.
  Let $Y$ be a subset of $G \setminus (X_1 \cup Z_1)$
  with $|Y|=\frac{1}{4} n$, and let $G'=G[Y]$.
  By Theorem~\ref{center}, there exists a vertex $t_0' \in T$ such that
  $|D| \leq \frac{n}{8}$ for every component $D$ of $G' \setminus \chi(t_0')$.
  Let $X'$ be a minimal set of components of $G' \setminus \chi(t_0')$
  such that $|\bigcup_{D \in X'}D| \geq (\frac{1}{16} - 2 \epsilon)n$.
  Write $X_2=\bigcup_{D \in X'}D$ and $Z_2=\chi(t_0')$.
  Let $X_3=G' \setminus (X_2 \cup Z_2)$.
  By \eqref{otherside}, $|X_3| \geq \frac{1}{16}n$.
  Let $Z=Z_1 \cup Z_2$. Then $\kappa(Z)=2$, the sets
  $X_1,X_2,X_3$ are pairwise disjoint and anticomplete to each other, and
  $|X_i| \geq (\frac{1}{16}- 2 \epsilon)n=\frac{1}{17}n$. Now the conclusion
  of the lemma follows.
  \end{proof}

Next we prove the following, which immediately implies Theorem~\ref{ablogn}.

\begin{theorem}
  \label{ablogn_induction}
  Let $G \in \mathcal{C}$ with $|V(G)|=n$, and let $a,b \in V(G)$ be
  non-adjacent. Assume that $N(a) \cap \Hub(G)=\emptyset$ and $N(a) \neq \emptyset$.
  Then there is a set $Z \subseteq V(G) \setminus \{a,b\}$
  with $\kappa(Z) \leq -C \frac{1}{\log(1-\gamma)} (\max (1, \log |N(a)|))$ and such that every component of
  $G \setminus Z$ contains at most one of $a,b$.
\end{theorem}

\begin{proof}
  We may assume that $|N(a)| > -C \frac{1}{\log(1-\gamma)}$.
  Let $D$ be the component of $G \setminus N[a]$
  such that $b \in D$. We may assume that $G=D \cup N[a]$.
  Write $X=N(a)$.
  Our first goal is to  show the following:
  \\
  \\
  \sta{Assume that either $X$ is pure, or  there does not
    exist $X' \subseteq X$ with $|X'| \geq \frac{1-4 \epsilon}{2} |X|$  such that $X'$ is pure. Then there exists  $S \subseteq X \cup (D \setminus b)$
            with $\kappa(S) \leq C-2$ such the component $D(b)$
            of $D \setminus S$ with $b \in D(b)$ meets $N(x)$ for 
                        at most $(1-\epsilon) |X|$ vertices
            $x \in X \setminus S$.
\label{onestep}}
  \\
  \\
  The proof proceeds in several steps.
  \\
  \\
  \sta{$X$ is chordal. \label{Xchordal}}

    Suppose that there is a hole $H \subseteq X$. Then $(H,a)$ is a wheel.
    But $D$ is connected and $H \subseteq N(D)$, contrary to
    Theorem~\ref{hubstarcutset}. This proves~\eqref{Xchordal}.
     \\
    \\
    \sta{     If there is a clique
      $K$ of size $\epsilon |X|$ in $X$, then \eqref{onestep} holds. \label{noclique}}.

    Suppose such a clique $K$ exists. Now setting
    $S=K$, it follows that \eqref{onestep} holds. This proves~\eqref{noclique}.
    \\
    \\
    Let $Z'=N(b) \cap X'$. Since $G$ is $C_4$-free, it follows that
$Z'$ is a clique. By \eqref{noclique}, we can apply Lemma~\ref{chordal} to $X \setminus Z'$; 
    let $Z,X_1,X_2,X_3 \subseteq X \setminus Z'$ as in the conclusion of the lemma.    

    Let $\delta=\frac{1}{6 \times 17^2}$.
    A triple $x_1x_2x_3$ is {\em partitioned} if $x_i \in X_i$.
    We remind the reader that dangerous triples were defined in Section~\ref{sec:dangerous}.
    \\
    \\
    \sta{Under the assumptions in \eqref{onestep}, if for some $i \in \{1,2,3\}$, some  $x_i \in X_i$
      is a center of  $\delta |X|^2$ dangerous triples, then \eqref{onestep} holds.
      \label{fewdangerous}}
    
    Since $\epsilon < \frac{1}{8}\delta$,
    it follows from \eqref{noclique} and Theorem~\ref{dangerous} that
  there exists    $S' \subseteq D \setminus b$ with $\kappa(X) \leq 4$ such that
  the component $D(b)$
            of $D \setminus S'$ with $b \in D(b)$ 
is disjoint from $N(x)$ for
at least $\epsilon|X|$ vertices $x \in X$.
Setting $S=S'  \cup Z'$, \eqref{onestep} holds. This
proves~\eqref{fewdangerous}.
\\
\\
In view of \eqref{fewdangerous}, in order to prove \eqref{onestep}  we may assume that  for every $i \in \{1,2,3\}$,
every $x_i \in X_i$
is a center fewer that   $\delta |X|^2$ dangerous triples.
\\
\\
\sta{At least $\frac{1}{2}\left\lceil\frac{1}{17} |X|\right\rceil^3$ of the partitioned triples
  are not dangerous. \label{manysafe}}

Since there are $\left\lceil\frac{1}{17} |X|\right\rceil^3$ partitioned triples, 
it is enough to prove that at most  $\left\lceil\frac{1}{17} |X|\right\rceil^3/2$ of the partitioned triples are  dangerous.
Let $i \in \{1,2,3\}$.
By \eqref{fewdangerous}, the number of dangerous triples with center
in $X_i$ is at most
$$|X_i| \times  \delta |X|^2 \leq =\left\lceil\frac{1}{17} |X|\right\rceil \cdot \frac{1}{6 \times 17^2}|X| \leq \frac{1}{6}\left\lceil\frac{1}{17} |X|\right\rceil^3.$$
Since every dangerous triple has a center in one of the sets
$X_1,X_2,X_3$, it follows that
the total number of dangerous triples is at most
$$3 \times \frac{1}{6}\left\lceil\frac{1}{17} |X|\right\rceil^3 \leq  \frac{1}{2}\left\lceil\frac{1}{17} |X|\right\rceil^3,$$
as required. This proves \eqref{manysafe}.
\\
\\
\sta{Every partitioned triple that is not dangerous is $b$-separated.
  \label{septriple}}

Let $x_1x_2x_3$ be a partitioned triple that is not dangerous.
Let $H=D \cup \{x_1,x_2,x_3,a\}$.
Let $F$ be a minimal connected subgraph of $D$ such that each of
$x_1,x_2,x_3$ has a neighbor in $F$. Since $X \cap \Hub(G)=\emptyset$,
it follows from Lemma~\ref{minimalconnected}
that $\Sigma=F \cup \{x_1,x_2,x_3,a\}$  is a  pyramid with apex $a$.
For $i \in \{1,2,3\}$, let $Q_i=a \dd x_i$.
Since $x_1x_2x_3$ is not a dangerous triple, we deduce that $a$
is not contained in a cross-edge of an extended near-prism in  $H$. Now
by Theorem~\ref{pyramid_separate} applied to $\Sigma, Q_1,Q_2,Q_3,$ and $H$,
it follows that the triple $x_1x_2x_3$ is $b$-separated, and
\eqref{septriple} follows.
\\
\\
  \sta{$X$ is a $(\frac{1}{2 \times 17^3},b)$-breaker in $G$.
    \label{getbreaker}}

By \eqref{manysafe}, at least $\frac{1}{2}\left\lceil\frac{1}{17} |X|\right\rceil^3$ of the partitioned triples  are not dangerous. Now by \eqref{septriple} at least 
$\frac{1}{2}\left\lceil\frac{1}{17} |X|\right\rceil^3 \geq \frac{1}{2 \times 17^3}|X|^3$ of the partitioned triples
$x_1x_2x_3$ are $b$-separated, and \eqref{getbreaker} follows.
\\
\\

By Theorem~\ref{localglobal},
                    there exist  $S' \subseteq D \setminus b$
            with $\kappa(S) \leq C-4$ such that the component $D(b)$
            of $D \setminus S$ with $b \in D(b)$ 
is disjoint from $N(x)$ for
at least $\epsilon|X|$ vertices $x \in X$. Setting $S= Z' \cup S$, 
\eqref{onestep} follows.
\\
\\
We complete the proof of Theorem~\ref{ablogn_induction}
by induction on $|N_G(a)|$.
If $X$ is pure, let $X_0=X$.
If $X$ is not pure and there does not
exist $X' \subseteq X$ with $|X'| \geq \frac{1-4 \epsilon}{2} |X|$  such that $X'$ is pure, let $X_0=X$.
If $X$ is not pure and there exists $X' \subseteq X$ with $|X'| \geq \frac{1-4 \epsilon}{2} |X|$ such that $X'$ is pure, let $X_0$ be a pure subset of $X$ with
$|X_0| \geq \frac{1-4 \epsilon}{2} |X|$.
Let $G_0=G \setminus (X \setminus X_0)$.
Note that $X_0=N_{G_0}(a)=N_{G_0}(D)$.
Apply \eqref{onestep} in $G_0$ (and with   $X=X_0)$.
Let $S$ be as in \eqref{onestep}.
Let  $D_1$ be the component of
$D \setminus S$ with $b \in D_1$, let $X_1=N_{G_0}(D_1)$ and let
$G_1=D_1 \cup X_1 \cup \{a\}$. By \eqref{onestep}
$|N_{G_1}(a)| \leq (1-\epsilon)|N_{G_0}(a)| \leq (1-\epsilon)|X_0|$.
Let $G_2=G[V(G_1) \cup (X \setminus X_0)]$.
Since  $N_{G_2}(a) = (X  \setminus X_0) \cup N_{G_1}(a)$,
it follows that
$$|N_{G_2}(a)| \leq (1-\epsilon)|X_0|+|X|-|X_0| \leq |X| - \epsilon |X_0|
\leq (1-\gamma)|X|.$$ 
If $|N_{G_2}(a)| \leq 1$, let $Z_1=N_{G_2}(a)$. Otherwise, $\log |N_{G_2}(a)| \geq 1$ and inductively,  there is a set $Z_1 \subseteq V(G_2) \setminus \{a,b\}$
  with $\kappa(Z) \leq -C \frac{1}{\log (1-\gamma)}(\log |N_{G_2}(a)|)$ and such that every component of
  $G_2 \setminus Z_1$ contains at most one of $a,b$.
In the latter case, $\kappa(Z_1) \leq -C \frac{1}{\log (1-\gamma)} \log (|N_G(a)|) - C$.
Since $\kappa(S) \leq C$, the set
$Z_1 \cup S$ satisfies the conclusion of the theorem. In the former case, $|Z_1 \cup S| \leq C$ and $Z_1 \cup S$ satisfies this conclusion of the theorem.  
\end{proof}

\section{Stable   sets  of safe hubs} \label{sec:centralbag_banana}
As we discussed in Section~\ref{sec:outline}, the in the course of the proof of
Theorem~\ref{bananaclique}, we will repeatedly   decompose the graph by star cutsets
arising from a stable set of appropriately chosen hubs (using Theorem~\ref{hubstarcutset}).
In this section we prepare the tools for handling one such step: one
stable set of safe hubs.

Let $d$ be an integer.
In this section we again set $\epsilon=\frac{1}{4 \times 17^6 \times 48 \times 192}$, $\gamma=\frac{\epsilon(1-4 \epsilon)}{2}$ and 
$C=96^2 \times 4 \times 17^6+4$.
Let $a,b \in V(G)$ be non-adjacent and such that no subset $Z$ of $G$
with $\kappa(Z) \leq  -C \frac{1}{\log(1-\gamma)}  \log n +d$  separates $a$ from $b$.
Following \cite{TWXV}, we say that
    a vertex $v$ is {\em $d$-safe} if $|N(v) \cap \Hub(G)| \leq d$.
        As in \cite{TWXV}, the goal of the next  lemma is to classify $d$-safe vertices into
``good ones'' and ``bad ones'', and show that the bad ones are rare.
        A vertex $v \in G$ is $ab$-cooperative if there exists a
    component $D$ of $G \setminus N[v]$ such that $a,b \in N[D]$.

    \begin{lemma}
      \label{abnoncoop}
      If $v \in G$  is  $d$-safe and not $ab$-cooperative, then
      $v$ is adjacent to both $a$ and $b$. In particular, the set of
      vertices that are not $ab$-cooperative is a clique.
    \end{lemma}

    \begin{proof}
      Suppose $v$ is non-adjacent to $b$ and $v$ is not $ab$-cooperative.
      Let $D$ be the component of $G \setminus N[v]$ such that $b \in D$.
      Let $X=N(D) \setminus \Hub(G)$. 
      Let $G'=D \cup X \cup \{v\}$. Then $a \not \in G'$.
      We apply Theorem~\ref{ablogn} to the vertices $v,b$ in $G'$ to
      obtain a subset
            $Z \subseteq G' \setminus \{v,b\}$ 
      with $\kappa(Z) \leq  -C \frac{1}{\log(1-\gamma)}  \log n$ 
      and such that every component of
  $G' \setminus Z$ contains at most one of $v,b$.
  We claim that $Z'=Z \cup (N(v) \cap \Hub(G))$ separates $a$ from $b$
  in $G$.
  Suppose that $P$ is a path from $a$ to $b$ with $P^* \cap Z'=\emptyset$.
  Since $b \in D$ and $a \not \in D$, there is a vertex $x \in P$ such that
  $b \dd P \dd x \subseteq D$ and $x \not \in D$. Then $x \in X$.
But now
$b \dd P \dd x \dd v$ is a path from
$b$ to $v$ in $G' \setminus Z$, a contradiction.
This proves the claim that $Z'$ separates $a$ from $b$ in $G$.
But
$\kappa(Z') \leq  -C \frac{1}{\log(1-\gamma)}  \log n +d$,
a contradiction. This  proves Lemma~\ref{abnoncoop}.
\end{proof}

    Let $S'$ be a stable set of hubs of $G$ with $S' \cap \{a,b\}=\emptyset$, and assume that every $s \in S'$ is $d$-safe.
    Let $S'_{bad}=S' \cap N(a) \cap N(b)$. Since $G$ is $C_4$-free, it follows that $|S'_{bad}| \leq 1$. Let $S=S' \setminus S'_{bad}$.
    By  Lemma \ref{abnoncoop}, every vertex in $S$ is $ab$-cooperative.

  A {\em separation} of  $G$ is a triple $(X,Y,Z)$ of pairwise disjoint
  subsets of $G$ with $X \cup Y \cup Z = G$ such that $X$ is anticomplete to
  $Z$. We are now ready to move on to star cutsets.  
  As in other papers on the subject, we associate a certain unique star separation to every  vertex of $S$. The choice of the separation is the same as in \cite{TWXV}.

  Let $v \in S$. Since $v$ is $ab$-cooperative,
    there is a  component
  $D$ of $G \setminus N[v]$ with $a,b \in N[D]$.
  Since $v \not \in S'_{bad}$, it follows that
  $v$ is not complete to $\{a,b\}$; consequently $D \cap \{a,b\} \neq \emptyset$, and so the component $D$ is unique.
  Let $B(v) = D$, let $C(v)=N(B(v)) \cup \{v\}$, and let $A(v)=G \setminus (B(v) \cup C(v))$. Then $(A(v), C(v),B(v))$ is the 
    {\em canonical star separation of $G$ corresponding to $v$}.
    
    As in \cite{TWXV}, we observe:
\begin{lemma} \label{nohub}
  The vertex $v$ is not a hub of $G\setminus A(v)$.
    \end{lemma}

    \begin{proof}
      Suppose
      that $(H,v)$ is a proper wheel or a loaded pyramid in $G \setminus A(v)$.
      Then $H \subseteq N[B(v)]$, contrary to Theorem \ref{hubstarcutset}.
               \end{proof}

Let $\mathcal{O}$ be a linear order on $S \cap \Hub(G)$.  Following \cite{wallpaper}, we say
     that two vertices of $S \cap \Hub(G)$ are {\em star twins} if $B(u)= B(v)$, $C(u) \setminus \{u\} = C(v) \setminus \{v\}$, and
    $A(u) \cup \{u\} = A(v) \cup \{v\}$. 

Let $\leq_A$ be a relation on $S \cap \Hub(G)$ defined as follows: 
\begin{equation*}
\hspace{2.5cm}
x \leq_A y \ \ \ \text{ if} \ \ \  
\begin{cases} x = y, \text{ or} \\ 
\text{$x$ and $y$ are star twins and $\mathcal{O}(x) < \mathcal{O}(y)$, or}\\ 
\text{$x$ and $y$ are not star twins and } y \in A(x).\\
\end{cases}
\end{equation*} 
Note that if $x \leq_A y$, then either $x = y$, or $y \in A(x).$

The following two results were  proved in \cite{TWXV} with a slightly different setup: The set of hubs is defined differently. However, the proofs do not use the definition of the set of hubs. 
  
\begin{lemma}[Chudnovsky, Gartland, Hajebi, Lokshtanov, Spirkl \cite{TWXV}, Lemma 4.8] \label{Aorder}
      $\leq_A$ is a partial order on $S \cap \Hub(G)$. 
\end{lemma}

Let  $\Core(S')$ be a the set of all $\leq_A$-minimal elements of $S \cap \Hub(G)$.

  \begin{lemma}[Chudnovsky, Gartland, Hajebi, Lokshtanov, Spirkl \cite{TWXV}, Lemma 4.9]\label{looselylaminar}
    Let $u,v \in \Core(S')$. Then
    $A(u) \cap C(v)=C(u) \cap A(v)=\emptyset$.
  \end{lemma}

  As in \cite{TWXV}, we define the \emph{central bag}
 $$\beta^A(S')=\left(\bigcap_{v \in \Core(S')} (B(v) \cup C(v))\right) \setminus S'_{bad}.$$
 The next result describes  important properties of
 $\beta^A(S')$.

\begin{theorem} 
  \label{A_centralbag}
    The following hold:
    \begin{enumerate}
    \item For every $v \in \Core(S')$, we have $C(v) \subseteq \beta^{A}(S')$. \label{A-1}
        \item For every component $D$ of $G  \setminus (\beta^{A}(S') \cup S'_{bad})$, there exists $v \in \Core(S')$ such that $D \subseteq A(v)$. Further, if $D$ is a component of $G \setminus (\beta^A(S') \cup S'_{bad})$ and $v \in \Core(S')$ such that $D \subseteq A(v)$, then $N(D) \subseteq C(v) \cup S'_{bad}$. \label{A-3}
      \item $S' \cap \Hub(\beta^A(S'))=\emptyset$. \label{A-4}
    \end{enumerate}
\end{theorem}

\begin{proof}
    \eqref{A-1} is immediate
    from Lemma  \ref{looselylaminar}.

Next we prove  \eqref{A-3}. Let $D$ be a component of
$G \setminus (\beta^A(S') \cup S'_{bad})$. Since  $G \setminus (\beta^A(S') \cup S'_{bad})=\bigcup_{v \in \Core(S')}A(v)$, there exists
    $v \in \Core(S')$ such that $D \cap A(v) \neq \emptyset$.
    If $D \setminus A(v) \neq \emptyset$, then, since $D$ is connected, it follows that $D \cap N(A(v)) \neq \emptyset$; but then $D \cap C(v) \neq \emptyset$, contrary to \eqref{A-1}. Since $N(D) \subseteq \beta^A(S')  \cup S'_{bad}$ and $N(D) \subseteq A(v) \cup C(v)  \cup S'_{bad}$, it follows that $N(D) \subseteq C(v)  \cup S'_{bad}$. 
    This proves \eqref{A-3}.

    To prove \eqref{A-4}, let $u \in S' \cap \Hub(\beta^A(S'))$.
    Since $\beta^A(S') \cap S'_{bad}=\emptyset$, we deduce that $u \not \in S'_{bad}$,
    and so   $u \in S \cap \Hub(G)$.
       By  Lemma \ref{nohub}, it follows that
$\beta^A(S') \not \subseteq B(u) \cup C(u)$, and therefore $u \not \in \Core(S')$.
But then $u \in A(v)$ for some $v \in \Core(S')$, and so $u \not \in \beta^A(S')$,
a contradiction. This proves \eqref{A-4} and completes the proof of
 Theorem \ref{A_centralbag}.
  \end{proof}

In the course of the proof of Theorem \ref{bananaclique}, we will inductively obtain
a small  cutset separating $a$ from $b$ in
$\beta^A(S')$, using that the vertices in $S'$ are not hubs in $\beta^A(S')$. The next theorem lets us lift this cutset
into a cutest that separates $a$ from $b$ in $G$.

\begin{theorem}
  \label{smallsepinG}
Let $(X,Y,Z)$ be a separation of $\beta^A(S')$ such that $a \in X$ and
$b \in Z$. Then there exists a set $Y' \subseteq V(G)$ 
such that 
\begin{enumerate}
  \item $Y'$
    separates $a$ from $b$ in $G$, and
  \item $\kappa(Y') \leq \kappa(Y)+|Y \cap \Core(S')|(-C \frac{1}{\log(1-\gamma)}    \log n +d) + 1$.
 
    \end{enumerate}
\end{theorem}

\begin{proof}
  Let $s \in Y \cap \Core(S')$.
    Let $X=C(s)  \setminus \Hub(G)$. 
    Let $G'=B(s) \cup X \cup \{s\}$.
    Suppose first that $s$ is non-adjacent to $b$.
      We apply Theorem~\ref{ablogn} to the vertices $s,b$ in $G'$ to
      obtain a subset
            $Z(s) \subseteq G' \setminus \{s,b\}$ 
      with $\kappa(Z) \leq -C \frac{1}{\log(1-\gamma)} \log n$ and such that every component of
  $G' \setminus Z(s)$ contains at most one of $s,b$.
  Now suppose that $s$ is adjacent to $b$. Since $s$ is cooperative,
  it follows that $s$ is non-adjacent to $a$.
Now we apply Theorem~\ref{ablogn} to the vertices $s,a$ in $G'$ to
      obtain a subset
            $Z(s) \subseteq G' \setminus \{s,a\}$ 
      with $\kappa(Z(s)) \leq -C \frac{1}{\log(1-\gamma)} \log n$ and such that every component of
  $G' \setminus Z(s)$ contains at most one of $s,a$.
      We deduce:
      \\
      \\
      \sta{For every $s \in \Core(S')$, $Z(s)$ separates $s$ from at least one of $a,b$. \label{Z(s)}}
        \\
        \\
  Now let
  $$Y'=Y \cup \bigcup_{s \in Y \cap \Core(S')}Z(s) \cup \bigcup_{s \in Y \cap \Core(S')}
  (N(s) \cap \Hub(G)) \cup S'_{bad}.$$
  Since every vertex of $S'$ is safe, we have that
  $\kappa(Y') \leq \kappa(Y)+|Y \cap \Core(S')| (-C \frac{1}{\log(1-\gamma)}\log n +d)$.

  We show that $Y'$ separates $a$ from $b$ in $G$.
  Suppose not.
  Let $D_b$ be the component of $\beta^A(S') \setminus Y$ such that
  $b \in D_b$, an let $D_a$ be the component of $\beta^A(S') \setminus Y$ such that
  $a \in D_a$. Let $D_b'=D_b \cup \bigcup_{s \in D_b \cap \Core(S')}A(s)$,
  and let $D_a'=D_a \cup \bigcup_{s \in D_a \cap \Core(S')}A(s)$.
Let $P$ be a path from $b$ to $a$ in $G \setminus Y'$.
  Since $b \in D_b'$ and $a \not \in D_b'$, there is $x \in P$
  such that $b \dd P \dd x \subseteq D_b'$ 
  and the neighbor $y$ of $x$ in the path $x \dd P \dd a$ does
  not belong to $D_b'$.
  Since for every  $s \in D_b \cap \Core(S')$, $N_{G}(A(s)) \cap \beta^A(S') \subseteq C(s) \subseteq D_b \cup Y$, and since  $Y \subseteq Y'$, it follows that $y \not \in \beta^A(S')$.
  Let 
   $D'$ be the component of   $G \setminus (\beta^A(S') \cup S'_{bad})$
  such that $y \in D'$.
    By Theorem~\ref{A_centralbag}\eqref{A-3} there is an
  $s \in \Core(S')$ such that $D' \subseteq A(s)$ and
  $N(D') \subseteq C \cup S'_{bad}$; consequently by
  Theorem~\ref{A_centralbag}\eqref{A-1},
  $N(D') \subseteq  N_{\beta^A(S')}(s) \cup S'_{bad}$.
  In particular, $x \in N_{\beta^A(S')}(s)$. Since $y \not \in D'_b$, it follows that
  $s \not \in D_b$.
Since 
$N_{\beta^A(S')}(s) \cap D_b \neq \emptyset$ and $s \not \in D_b$,
it follows that $s \in Y$. Consequently, $P \cap Z(s)=\emptyset$.
Since $y \in A(s)$ and $a \in \beta^A(S') \subseteq B(s) \cup C(s)$, there exists $x' \in y \dd P \dd a$ such that $x' \in C(s)$
and $s$ has no other neighbors in the path $x' \dd P \dd a$.
Now $b \dd P \dd x \dd s$ and $a \dd P \dd x' \dd s$ are paths
from $b$ to $s$ and from $a$ to $s$, respectively, and both are
disjoint from $Z(s)$, contrary to \eqref{Z(s)}.
  \end{proof}

  \section{Bounding the number of non-hubs}
  \label{sec:boundhubs}

For $X \subseteq V(G)$, a component
$D$ of $G \setminus X$ is {\em full for $X$} if $N(D)=X$.
$X \subseteq V(G)$ is a {\em minimal separator} in $G$ if
there exist two distinct full components for $X$.
In this section we again set $\epsilon=\frac{1}{4 \times 17^6 \times 48 \times 192}$, $\gamma=\frac{\epsilon(1-4 \epsilon)}{2}$  and $C=96^2 \times 4 \times 17^6+4$, and let
$D=-C \frac{1}{\log(1-\gamma)}$.
    Let $a,b \in V(G)$ be non-adjacent.
  The goal of this section is to start with a minimal separator in $G$ separating $a$ from $b$, and turn it into a separator that interfaces well with
  Theorem~\ref{smallsepinG}.
  Let $d$ be an integer and let $S_1 \subseteq V(G) \setminus \Hub(G)$ be a stable set of $d$-safe vertices.
    For  a set $U \subseteq V(G)$ we denote by
  $\mu_d(U)$ the set $U \cap S_1$.
  We will prove the following:
  \begin{theorem}
    \label{boundhubs}
    Let $d$ be an integer.
    Let $G$ be an even-hole-free graph  and let $a,b \in V(G)$ be non-adjacent.
    Let $|V(G)|=n$.
    Let $Y$ be a minimal separator in $G$ such that $a$ and $b$ belong
    to different components of $G \setminus Y$.
    Then there exists a set $Y' \subseteq V(G) \setminus \{a,b\}$ such that
    \begin{itemize}
    \item $Y'$ separates $a$ from $b$; 
 \item $|\kappa(Y' \setminus Y)| \leq D (D\log n+d) \log n$; and
         \item $|\mu_d(Y')| \leq  D(D \log n+d) \log n$. 
           \end{itemize}
    \end{theorem}

    The main ingredient of the proof is  the    following:

    \begin{lemma}
      \label{onestep_sep}
      Let  $d$ be an integer and let $G$ be an even-hole-free graph
      where $G=D_1 \cup D_2 \cup X$, $X$ is a minimal
    separator in $G$, and $D_1, D_2$ are full components for $X$. Let $a \in D_1$ and $b \in D_2$.
    Assume that $X$ is a stable set, $X \cap \Hub(G)=\emptyset$, and that every vertex of $X$ is
    $d$-safe. Then there exists $Z \subseteq V(G) \setminus \{a,b\}$
with $\kappa(Z) \leq  D \log n + d$ 
    such that either
    \begin{itemize}
\item the component $D(b)$
            of $D_2 \setminus Z$ with $b \in D(b)$ 
contains a neighbor of $x$ for at most $(1 - \epsilon)|X|$
          vertices $x \in X \setminus Z$,
or
\item the component $D(a)$
            of $D_1 \setminus Z$ with $a \in D(a)$ 
contains a neighbor of $x$ for at most $(1 - \epsilon)|X|$
          vertices $x \in X \setminus Z$.
\end{itemize}
\end{lemma}

\begin{proof}
We may assume that $|X|> D \log n + d$.
\\
\\
\sta{If some $v \in D_1$ has $\epsilon|X|$ neighbors in $X$, then the theorem holds. \label{onevertex}}

  Suppose such $v$ exists. Let $G'=D_2 \cup N_{G}[v]$. Apply Theorem~\ref{ablogn} to $v,b$ in
  $G'$ (where $v$ plays the role of $a$) to obtain a set $Z$ as in the conclusion of
  the theorem. Let $D(b)$ be the component of $D_2 \setminus Z$ with $b \in D(b)$.
  Then $N_{G'}[v] \setminus Z$ is anticomplete to $D(b)$ in $G'$ and therefore in $G$.
  Since $|N_{G'}(v)| \geq \epsilon|X|$, \eqref{onevertex} follows.
  \\
  \\
  In view of \eqref{onevertex} (using the symmetry between $a$ and $b$) from now on we assume that  no $v \in D_1 \cup D_2$ has $\epsilon|X|$ neighbors in $X$. 
  \\
  \\
  \sta{If for some $x \in X \setminus N(a)$ at least $\epsilon |X|$ vertices of $X$ are anticomplete to
    the component $D(a)$  of $D_1 \setminus N(x)$ with $a \in D(a)$, then the theorem holds.
    \label{nbrsofone}}

 Suppose that such a vertex $x$ exists. Let $D'=D(a)$, let $X'=(N(x) \cap N(D(a)) \setminus \Hub(G)$ and let
  $G'=D' \cup X' \cup \{x\}$. Apply Theorem~\ref{ablogn} to $x,a$ in
  $G'$ (where $x$ plays the role of $a$, and $a$ plays the role of $b$) to obtain a set $Z$ as in the conclusion of
  the theorem. Let $Z'=Z \cup (N(x) \cap \Hub(G))$. Since $x$ is $d$-safe, it follows that
  $\kappa(Z') \leq D\log n + d$. Let $D'(a)$ be the component of $D_1 \setminus Z'$     with $a \in D'(a)$.
  Since $Z$ separates $x$ from $a$ in $G'$, it follows that
  $N[x] \setminus Z'$ is anticomplete to $D(a) \setminus Z'$,
  and therefore  $D'(a) \subseteq D(a)$. 
  
  Consequently, if $x' \in X$ is anticomplete to $D(a)$, then $x'$ is anticomplete to $D'(a)$.  Since
  least $\epsilon |X|$ vertices of $X$ are anticomplete to $D(a)$, \eqref{nbrsofone} follows.
\\
\\
  In view of \eqref{nbrsofone}, from now on we assume that for every $x \in X \setminus N(a)$ fewer than  $\epsilon |X|$ vertices of $X$ are anticomplete to
  the component $D(a)$  of $D_1 \setminus N(x)$ with $a \in D(a)$, and
  for every $x \in X \setminus N(b)$ fewer than  $\epsilon |X|$ vertices of $X$ are anticomplete to
  the component $D(b)$  of $D_2 \setminus N(x)$ with $b \in D(b)$.
    Let $X'=X \setminus (N(a) \cup N(b))$. 
  Then $|X'| \geq (1-2\epsilon)|X|$. Since $X$ is a stable set, we can choose
  disjoint and anticomplete subsets  $X_1,X_2,X_3$  of $X'$ such that
  $|X_1|=|X_2|=|X_3|=\left\lceil\frac{1}{4}|X|\right\rceil$.
  
  Let us say that a partitioned triple $x_1x_2x_3$  is {\em $b$-triangular} if there is a
  minimal connected subgraph $H$ of $D_2$ containing neighbors of $x_1,x_2,x_3$ such that
  either
  \begin{itemize}
    \item $H$
      satisfies the third outcome to Lemma~\ref{minimalconnected}, or
    \item (possibly with the roles of $x_1,x_2,x_3$ exchanged) $H$ is the interior of a path from  $x_1$ to $x_3$, and $x_2$ has exactly two neighbors in $H$
      and they are adjacent.
      \end{itemize}
  We define {\em $a$-triangular} triples similarly.
  \\
  \\
  \sta{Every partitioned triple is either $a$-triangular or $b$-triangular.
        \label{onlygoodtriples}}

  Let $x_1x_2x_3$ be a partitioned triple, and suppose that it is neither
  $b$-triangular nor $a$-triangular. For $i \in \{1,2\}$,  let $H_i$
    be a minimal connected induced subgraph of $D_i$ such that
    each of $x_1,x_2,x_3$ has a neighbor in $H_i$. We apply
    Lemma~\ref{minimalconnected} to $H_1$ and $H_2$.
    If the second outcome of the lemma holds for both $H_1$ and $H_2$,
    then $H_1 \cup H_2 \cup \{x_1,x_2,x_3\}$ is a theta, a contradiction.
    Thus we may assume that $H_1$ is the interior of a path from $x_1$
    to $x_3$ and $x_2$ has at least two non-adjacent neighbors in $H_1$.
If $H_2$ is  the interior of a path from $x_1$
to $x_3$ and $x_2$ has a  neighbor in $H_2$,
then $(H_1 \cup H_2, x_2)$ is a wheel in $G$, contrary to the fact that
$x_2 \not \in \Hub(G)$. Thus $H_2$ is not  the interior of a path from $x_1$
to $x_3$ such that $x_2$ has a  neighbor in $H_2$.
Next suppose that the second outcome of
Lemma~\ref{minimalconnected} holds for $H_2$. Let $P_1,P_2,P_3$ be as in
the  the second outcome of
Lemma~\ref{minimalconnected}, and let $a_2=P_1 \cap P_2 \cap P_3$.
Then $x_2$ is non-adjacent to $a_2$, and so $H_1 \cup H_2 \cup \{x_1,x_2,x_3\}$
contains a theta with ends $a_2,x_2$, a contradiction. 
It follows that
the first outcome of Lemma~\ref{minimalconnected} holds for $H_2$, and, by symmetry between $x_1$ and $x_3$, we may assume that
$H_2$ is the interior of a path from $x_1$ to $x_2$, and $x_3$ has two
non-adjacent neighbors in $H_2$. Now we get a theta with ends
$x_2,x_3$ and paths $x_2 \dd H_1 \dd x_3$, $x_2 \dd H_2 \dd x_3$, and
$x_2 \dd H_1 \dd x_1 \dd H_2 \dd x_3$, again a contradiction.
 This proves~\eqref{onlygoodtriples}.
   \\
   \\
   In view of \eqref{onlygoodtriples},
     by switching the roles of $D_1$ and $D_2$ if necessary, 
 we may assume that at least
  $\frac{1}{2}\left\lceil\frac{1}{4}|X|\right\rceil^3$ of the partitioned triples are $b$-triangular.
  \\
  \\
 Let us say that the triple $x_1x_2x_3$ is {\em acceptable} if
 it is partitioned and 
for every $\{i,j,k\}=\{1,2,3\}$ there is path $P_{ij}$
from $x_i$ to $x_j$ with interior in $D_1 \setminus N(x_k)$.
\\
\\
\sta{At most $ 6 \epsilon \left\lceil\frac{1}{4}|X|\right\rceil^2 |X|$ partitioned  triples are not
  acceptable.
  \label{manyacceptable}}

Let  $x_1 \in X_1$ and let $D(a)$ be the component of $D_1 \setminus N(x_1)$
such that $a \in D(a)$. If $x_1x_2x_3$ is a partitioned triple such that
there is no path from $x_2$ to $x_3$  in $D_1 \setminus N(x_1)$, then
at least one of $x_2,x_3$ is anticomplete to $D(a)$.
By the assumption following \eqref{nbrsofone},  there are fewer than
$2\epsilon|X| \times \left\lceil\frac{1}{4}|X|\right\rceil$ such  pairs $x_2x_3$,
and therefore at most
$$\left(2\epsilon|X| \times\left\lceil\frac{1}{4}|X|\right\rceil\right) \times\left \lceil\frac{1}{4}|X|\right\rceil  \leq  2 \epsilon \left\lceil\frac{1}{4}|X|\right\rceil^2 |X|$$
such triples.
Repeating this argument with $x_2$ and $x_3$ playing the role of $x_1$,
we get that there are at most $ 6 \epsilon \left\lceil\frac{1}{4}|X|\right\rceil^2 |X|$
partitioned triples that are not acceptable, and
\eqref{manyacceptable} follows.
\\
\\
  Let $G'=D_2 \cup X$. Our next goal is to show the following:
  \\
  \\
  \sta{$X$ is a $(\frac{1}{256},b)$-breaker in $G'$.
    \label{getbreaker2}}

  By \eqref{manyacceptable}, there are fewer than
  $ 6 \epsilon \left\lceil\frac{1}{4}|X|\right\rceil^2 |X|$ partitioned triples that are not acceptable.
   Since the total number of $b$-triangular triples at least   $\frac{1}{2}\left\lceil\frac{1}{4}|X|\right\rceil^3$,
   we deduce that 
  there are at least $\left(\frac{1}{2} - 24 \epsilon\right)\left\lceil\frac{1}{4}|X|\right\rceil^3$ triples that are both
  acceptable and $b$-triangular.

  Let $x_1x_2x_3$ be a $b$-triangular acceptable triple.
  Let $H$ be a minimal connected subgraph of $D_2$ as in the definition of
  $b$-triangular.
 Then there exists a triangle  $h_1h_2h_3$ in $H \cup \{x_1,x_2,x_3\}$ and three paths $P_1, P_2, P_3$, where $P_i$ is a path from $h_i$ to $x_i$ (possibly of length zero), such that:
\begin{itemize}
\item $V(H) = (V(P_1) \cup V(P_2) \cup V(P_3)) \setminus \{x_1, x_2, x_3\} $; 
\item the sets $V(P_1)$, $V(P_2)$ and $V(P_3)$ are pairwise disjoint; and
\item for distinct $i,j \in \{1,2,3\}$, there are no edges between $V(P_i)$ and $V(P_j)$, except $h_ih_j$.
\end{itemize}
  If the path $P_{i}$ has  even length, let $q_{i}=2$, and if the path
  $P_{i}$ has  odd length, let $q_{i}=3$.
  Let $G''$ be the graph obtained from $D_2 \cup \{x_1,x_2,x_3\}$
  by adding  a new vertex $v$ and paths $Q_i$ from $v$ to $x_i$
  such that
  \begin{itemize}
  \item $Q_i$ has length $q_i$.
    \item $(Q_1 \cup Q_2 \cup Q_3) \setminus \{x_1, x_2, x_3\}$ is anticomplete to $D_2$.
  \item       $Q_1 \setminus v$, $Q_2 \setminus v$ , $Q_3 \setminus v$
    are pairwise disjoint and anticomplete to each other.
  \end{itemize}

  \sta {$G''$ is even-hole-free. \label{G''ehf}}
    
  Suppose not, and let $H$ be an even hole in $G''$. Since $H$ is not an
  even hole in $G$, and since all internal vertices of each $Q_i$ have
  degree two, we may assume that $Q_1 \cup Q_2 \subseteq H$.
  Then $H \setminus (Q_1^* \cup Q_2^*)$ is a path from $x_1$ to
  $x_2$ with interior in $D_2 \cup \{x_3\}$
  and whose length has the same parity as $q_1+q_2$.
  Since the triple $x_1x_2x_3$ is acceptable, there is a path $R$
  from $x_1$ to $x_2$ with $R^* \subseteq D_1 \setminus N(x_3)$.
  Since $(H \setminus (Q_1 \cup Q_2)) \cup R$ is not an even hole in
  $G$, it follows that the length of $R$ has the same parity as $q_1+q_2+1$.
  But now $x_1  \dd P_1 \dd h_1 \dd h_2 \dd P_2 \dd x_2 \dd R \dd x_1$
  is an even hole in $G$, a contradiction. This proves
  \eqref{G''ehf}.
  \\
  \\
 By \eqref{G''ehf} and since $X \cap \Hub(G) = \emptyset$, the assumptions of Theorem~\ref{pyramid_separate} are satisfied. Now
      Theorem~\ref{pyramid_separate} applied in $G''$ implies that 
      the triple      $x_1x_2x_3$ is $b$-separated in $G'$.

 Since there are
 at least $\left(\frac{1}{2} - 24 \epsilon\right)\left\lceil\frac{1}{4}|X|\right\rceil^3$ acceptable $b$-triangular triples, and since $\epsilon \leq 1/96$, 
\eqref{getbreaker2} follows.
\\
\\
By Theorem~\ref{localglobal} applied in $G'$,
                    there exists  $Z \subseteq D_2 \setminus b$
            with $\kappa(Z) \leq 96^2 \cdot 256^2 \leq C - 4$ such the component $D(b)$
            of $D \setminus S$ with $b \in D(b)$ 
is disjoint from $N(x)$ for
at least $\epsilon|X|$ vertices $x \in X$.
This completes the proof.
    \end{proof}

Next we prove the following, which immediately implies Theorem~\ref{boundhubs}.

 \begin{theorem}
    \label{boundhubs_induction}
    Let $d$ be an integer.
    Let $G$ be an even-hole-free graph  and let $a,b \in V(G)$ be non-adjacent.
    Let $|V(G)=n$.
        Let $Y$ be a minimal separator in $G$ such that $a$ and $b$ belong
        to different full components $D_1$ and $D_2$ of $G \setminus Y$.
        Assume that $\mu_d(Y) \neq \emptyset$.
        Then there exists a set $Y' \subseteq V(G) \setminus \{a,b\}$ such that
    \begin{itemize}
    \item $Y'$ separates $a$ from $b$; 
      \item $\kappa(Y' \setminus Y) \leq  D (D \log n+d) \max (1, \log |\mu_d(Y)|)$; and 
      \item $|\mu_d(Y')| \leq  D (D \log n+d) \max (1, \log |\mu_d(Y)|)$. 
              \end{itemize}
            \end{theorem}

 \begin{proof}
   We may assume that $|\mu_d(Y)| \geq  D (D \log n+d)$.
              By Lemma~\ref{onestep_sep} applied to the graph $D_1 \cup D_2 \cup \mu_d(Y)$
              with $X=\mu_d(Y)$,
           we may assume that there exists $Z \subseteq V(G) \setminus \{a,b\}$
           with $\kappa(Z) \leq D \log n + d$ such that 
          the component $D(b)$
            of $D_2 \setminus Z$ with $b \in D(b)$ 
meets $N(x)$ for at most $(1-\epsilon)|X|$ vertices $x \in X \setminus Z$.
            Let $G'=G \setminus Z$.
            Let $Y_1=N_{G'}(D(b))$. Then $Y_1 \subseteq Y \setminus Z$,
            and $\mu_d(Y_1) \subseteq X \setminus Z$; consequently
            $|\mu_d(Y_1)| \leq (1-\epsilon)\mu_d(Y)$.
Let $D(a)$ be the component of $G' \setminus Y_1$ such that
$a \in D(a)$. Let $Y_2=N_{G'}(D(a))$. Then $Y_2 \subseteq Y_1$,
and therefore  $|\mu_d(Y_2)| \leq (1-\epsilon)|\mu_d(Y)|$.
Let $G''=D(a) \cup D(b) \cup Y_2$.
Then $Y_2$ is a minimal separator in $G''$ where $D(a)$ and $D(b)$ are full components for $Y_2$. If $|Y_2| \leq 1$, let $Y''=Y_2$.
Now  assume that $|Y_2|>1$.
Inductively,  there is a set $Y'' \subseteq V(G'') \setminus \{a,b\}$ such that
\begin{itemize}
    \item $Y''$ separates $a$ from $b$; 
      \item $|\kappa(Y'' \setminus Y_2)| \leq D (D\log n+d) \log |\mu_d(Y_2)|$; and
      \item $|\mu_d(Y'') | \leq D (D\log n+d) \log |\mu_d(Y_2)|$
              \end{itemize}

In both cases, let  $Y'=Y'' \cup Z$. Then
$Y' \setminus Y \subseteq (Y'' \setminus Y_2) \cup Z$
and $\mu_d(Y') \subseteq  \mu_d(Y'') \cup Z$.
Since $|\mu_d(Y_2)| \leq (1-\epsilon) |\mu_d(Y)|$  
and $\kappa(Z) \leq   D \log n + d$, it follows that $|Y'|$ satisfies the conclusion of the theorem.
\end{proof}

            \section{Separating a pair of vertices: a bound using clique size}
              \label{sec:bananaclique}
              We can now prove our first main result. We follow the
              outline of   the proof of Theorem 1.3  from          \cite{TWX}, using bounds from earlier sections of the present paper.
              As in the earlier sections, let
              $\epsilon=\frac{1}{4 \times 17^6 \times 48 \times 192}$,
$\gamma=\frac{\epsilon(1-4 \epsilon)}{2}$ 
              and $C=96^2 \times 4 \times 17^6+4$, and let
$D=-C \frac{1}{\log(1-\gamma)}$.
\begin{theorem}
  \label{bananaclique}
Let $t$ be an integer.
  Let $G$ be an even-hole-free graph  with $|V(G)|=n$ and with no clique of size $t+1$, and let $a,b \in V(G)$ be
  non-adjacent. 
  Then there is a set $Z \subseteq V(G) \setminus \{a,b\}$
  with
  $$\kappa(Z) \leq D \log n +2D (D \log n+ 8t)^2 2t \log^2 n$$
  and such that every component of
    $G \setminus Z$ contains at most one of $a,b$.
  \end{theorem}

We will need the main result of \cite{bisimplicialnew}.

\begin{theorem}[Chudnovsky, Seymour \cite{bisimplicialnew}] \label{bisimplicialnew}
  Every even-hole-free graph has a vertex $v$ such that
  $\kappa(N(v)) \leq 2$. 
\end{theorem}

Following the proof of Theorem 7.1 of \cite{TWIII}, using Theorem~\ref{bisimplicialnew} we deduce:

\begin{theorem}
  \label{logncollections}
  Let $t \in \mathbb{N}$, and let $G$ be an even-hole-free graph with no
  clique of size $t+1$ and    with $|V(G)|=n$.
  There exist 
    a partition
  $(S_1, \dots, S_k)$   of $V(G)$ with the following properties:
  \begin{enumerate}
  \item $k \leq   2t \log n$.
  \item $S_i$ is a stable set for every $i \in \{1, \dots, k\}$.
  \item For every $i \in \{1, \dots, k\}$ and $v \in S_i$ we have
    $\deg_{G \setminus \bigcup_{j <i}S_j}(v)  \leq 8t$. \label{hubsequence-3}
  \end{enumerate}
\end{theorem}

For the remainder of this section, let us fix $t \in \mathbb{N}$.
Let $G$ an even-hole-free  graph with no clique of size $t+1$, and let $a,b \in V(G)$. A
{\em hub-partition with respect to $ab$} of $G$ is a partition
$S_1, \dots, S_k$ of $\Hub(G) \setminus \{a,b\}$ as in
Theorem \ref{logncollections};
we call $k$ the {\em order} of the partition.
We call the {\em hub-dimension} of $(G,ab)$ (denoting it by
$\hdim(G,ab)$) the smallest $k$ such that $G$ has a hub-partition of order $k$
with respect to $ab$.

Since, in view of Theorem~\ref{logncollections}, we have 
$\hdim(G,ab) \leq  2t \log n$ for every $a,b \in V(G)$,
Theorem~\ref{bananaclique} follows immediately from the next result:

\begin{theorem}
  \label{diminduction_banana}
   Let $G$ be an even-hole-free graph  with $|V(G)|=n$ and with no clique of size $t+1$, and let $a,b \in V(G)$ be
  non-adjacent. 
  Then there is a set $Z \subseteq V(G) \setminus \{a,b\}$
  with
  $$\kappa(Z) \leq D \log n +2D (D \log n+ 8t)^2  \hdim(G,ab)  \log n $$
  and such that every component of
    $G \setminus Z$ contains at most one of $a,b$.
  \end{theorem}

\begin{proof}
  Let $a,b \in G$ be non-adjacent and suppose that no such set $Z$ exists.
  We will get a contradiction by induction on $\hdim(G,ab)$.
  Suppose that  $\hdim(G,ab)=0$. Then $\Hub(G) \subseteq \{a,b\}$ and  by Theorem \ref{ablogn},  there is a set $Z \subseteq V(G) \setminus \{a,b\}$
  with $\kappa(Z) \leq D  \log n$ and such that every component of
  $G \setminus Z$ contains at most one of $a,b$.
      Thus we may assume $\hdim(G,ab)>0$.
    
    Let $S_1, \dots, S_k$ be a hub-partition of $G$ with respect to
    $ab$ and with $k=\hdim(G,ab)$.
  We now use notation and terminology from Section \ref{sec:centralbag_banana}.
Write $d=8t$.
  It follows from the definition of $S_1$ that every vertex in $S_1$
is $d$-safe. Let 
$\beta^A(S_1)$ be as in Section \ref{sec:centralbag_banana}; then
$a,b \in \beta^A(S_1)$.
By Theorem \ref{A_centralbag}\eqref{A-4}, we have that   $S_1 \cap \Hub(\beta^A(S_1))=\emptyset$ and
   $S_2 \cap \Hub(\beta^A(S_1)), \dots, S_k \cap \Hub(\beta^A(S_1))$ is a hub-partition of
   $\beta^A(S_1)$ with respect to $ab$. 
It follows that $\hdim(\beta^A(S_1),ab) \leq k-1$.
Inductively  there exists
a set $Y_1 \subseteq \beta^A(S_1) \setminus \{a,b\}$
  with
  $$\kappa(Y_1) \leq D \log n +2D (D \log n+ d)^2  (k-1) \log n $$
  and such that every component of
    $\beta^A(S_1) \setminus Y_1$ contains at most one of $a,b$.
  Let $D(b)$ be the component of $G \setminus Y_1$ such that $b \in D$,
  and let $D(a)$ be the component of $\beta^A(S_1) \setminus N(D(b))$ with $a \in D(a)$.
  Write $N_{\beta^A(S_1)}(D(a))=Y_2$. Then $Y_2 \subseteq Y_1$ and $Y_2$ is a minimal separator
  in $\beta^A(S_1)$ where $D(a)$ and the component of $\beta^A(S_1) \setminus Y_2$ containing $D(b)$ are two distinct full
  components for $Y_2$.
  By Theorem~\ref{boundhubs} applied in $\beta^A(S_1)$ and using $S_1$
  to define $\mu_d(Y_2)$, 
    there exists a set $Y \subseteq \beta^A(S_1) \setminus \{a,b\}$ such that
    \begin{itemize}
    \item $Y$ separates $a$ from $b$ in $\beta^A(S_1)$, and
 \item $|\kappa(Y \setminus Y_2)| \leq D (D\log n+d) \log n$, and
         \item $|\mu_d(Y)| \leq  D(D \log n+d) \log n$. 
           \end{itemize}
        It follows that
    \begin{align*}
        \kappa(Y) &\leq D \log n +2D (D \log n+ d )^2  (k-1) \log n  + D ( D \log n+ d)  \log n\\
        &\leq D (\log n) (1 + 2(D \log n + d)^2(k-1) + D \log n + d).
    \end{align*}
    
        Since $\Core(S_1) \cap Y \subseteq  \mu_d(Y)$,
                we deduce that
     $|\Core(S_1) \cap Y| \leq D(D\log n+d) \log n$.
     Now  applying Theorem~\ref{smallsepinG} to $Y$ we obtain a set $Y'$ such that
     \begin{itemize}
 \item $Y'$
    separates $a$ from $b$ in $G$; and
    \item $\kappa(Y') \leq \kappa(Y)+|Y \cap \Core(S')|(D \log n +d)+1$.
     \end{itemize}
     Consequently,
     \begin{align*}
         \kappa(Y') &\leq \kappa(Y)+|Y \cap \Core(S')|(D \log n +d)+1\\
         &\leq D (\log n) (1 + 2(D \log n + d)^2(k-1) + D \log n + d) +|Y \cap \Core(S')|(D \log n +d)+1\\
         &\leq D (\log n) (1 + 2(D \log n + d)^2(k-1/2) + D \log n + d) +1\\
         &\leq D (\log n) (1 + 2(D \log n + d)^2k).
     \end{align*}
     as required.
\end{proof}

\section{The proof of Theorem~\ref{banana}}
              \label{sec:banana}

              We can finally prove Theorem~\ref{banana}.
    We will need a theorem from \cite{KLSSX}.
    
    \begin{theorem}  [Korchemna, Lokshtanov, Saurabh, Surianarayanan, Xue \cite{KLSSX}]
      \label{cliqueMenger}
      Let $G$ be a graph with $|V(G)|=n$, $A,B \subseteq V(G)$, $\mathcal{F}$ a family of
      cliques of $G$, and $f$ an integer. Then one of the following holds:
      \begin{itemize}
      \item There exists $S \subseteq V(G)$ such that $S$ separates $A$ from $B$
        and $\kappa(S) \leq f \log^2 n$.
      \item There exist an integer $t \leq 2 \log (|\mathcal{F}|)$ and
        $t \times f$ paths $P_1, \dots, P_{t \times f}$ from $A$ to $B$ such that for every
        $K \in \mathcal{F}$, $K \cap P_i \neq \emptyset$ for fewer than
        $4t$ values of $i$.
      \end{itemize}
      \end{theorem}

    We also need the following result of \cite{Alekseev} and \cite{Farber}:
    \begin{theorem}[Alekseev \cite{Alekseev}, Farber \cite{Farber}]
      \label{cliques}
      An $n$-vertex $C_4$-free graphs has at most $n^2$ maximal cliques.
      \end{theorem}

    We now prove Theorem~\ref{banana}.
    \begin{proof}
Let $D$ be as in Section~\ref{sec:boundhubs} and let $c=16 \times 256 D^3$.
      Let $a,b \in G$ and let $G'=G \setminus \{a,b\}$.
      Let $A=N_G(a)$ and $B=N_G(b)$.
      Let $\mathcal{F}$ be the set of all maximal cliques in $G'$.
      By Theorem~\ref{cliques}, we have  $|\mathcal{F}| \leq n^2$.
      Let $f=c \log^6 n$. We apply Theorem~\ref{cliqueMenger}
      to $G',A,B, \mathcal{F}$ and $f$.
      We may assume that the statement of the first bullet does not hold,
      and so there exists  an integer $t \leq 2 \log (|\mathcal{F}|)$ and
        $t \times f$ paths $P_1, \dots, P_{t \times f}$ from $A$ to $B$ such that for every
        $K \in \mathcal{F}$, $K \cap P_i \neq \emptyset$ for at most
      $4t$ values of $i$.
      We may assume that the paths $P_1, \dots, P_{t \times f}$ are induced.
      Let $G''=\bigcup_{i=1}^{t \times f}P_i \cup \{a,b\}$.
      Since every clique of $G'' \setminus \{a,b\}$ is contained in an
      element of $\mathcal{F}$, we deduce that
      $G''$ does not have a clique of size $8t+1 \leq 1+  16\log n$.
      By Theorem~\ref{bananaclique} there exists a
       set $Z \subseteq V(G'') \setminus \{a,b\}$
  with
  \begin{align*}
      \kappa(Z) &\leq D \log n +2D (D \log n+ 8\times 16 \log n)^2 2 \times 16 \log n \log^2 n \\
      &< 256 D^3 \log^5 n. 
  \end{align*}
  and such that every component of
  $G \setminus Z$ contains at most one of $a,b$.
  Since   for every
  $K \in \mathcal{F}$, $K \cap P_i \neq \emptyset$ for at most $4t$ values of $i$,
  it follows that the number of values of $i$ for which $K \cap P_i \neq \emptyset$ is less
     than $4 t \kappa(Z) \leq c \log^6 n$.
   But this contradicts the fact that $Z$ separates
  $a$ from $b$ in $G''$.
    \end{proof}

    \section{From pairs of vertices to tree decompositions.}
    \label{sec:domsep}
    
    In this section we prove our main result, following the outline of the
    last few sections of \cite{TI2}.
    We need a theorem from  \cite{TWXV}:
    \begin{theorem}[Chudnovsky, Gartland, Hajebi, Lokshtanov, Spirkl \cite{TWXV}]
      \label{domsep}
There is an integer  $d$ with the following property.
Let $G \in \mathcal {C}$ and let $w$ be a normal weight function on $G$.
Then there exists $Y \subseteq V(G)$ such that
\begin{itemize}
\item $|Y| \leq d$, and
\item $N[Y]$ is a $w$-balanced separator in $G$.
\end{itemize}
    \end{theorem}

    We also need some terminology and two results from \cite{TI2}.
        Let $L,d,r$ be integers. We say that an $n$-vertex graph $G$ is {\em $(L,d,r)$-breakable}
  if
  \begin{enumerate}
    \item
      for every two disjoint and anticomplete cliques 
      $H_1,H_2$ of $G$ with $|H_1| \leq r$ and $|H_2| \leq r$,  there is a set
  $X \subseteq G \setminus (H_1 \cup H_2)$ with $\alpha(X) \leq L$ separating $H_1$ from $H_2$, and
  \item  for every normal weight function $w$ on $G$ and for every induced subgraph
    $G'$ of $G$ there  exists a set $Y \subseteq V(G')$ with $|Y| \leq d$
    such that for every component $D$ of $G'  \setminus N[Y]$,
    $w(D) \leq \frac{1}{2}$.
\end{enumerate}

  \begin{theorem} [Chudnovsky, Hajebi, Lokshtanov, Spirkl \cite{TI2}]
  \label{domtosmall}
  For every  integer $d>0$ there is an integer $C(d)$ with the
  following property.
  Let $L, n, r >0$ be integers such that $r \leq d(2+\log n)$ and let $G$ be an
  $n$-vertex $(L,d,r)$-breakable
  theta-free graph.
  Then there exists a $w$-balanced separator $Y$ in $G$ such that
  $\alpha(Y) \leq C(d)  \lceil \frac{d (2+\log n)} {r} \rceil (2+\log n) L$.
\end{theorem}

  \begin{lemma} [Chudnovsky, Hajebi, Lokshtanov, Spirkl \cite{TI2}]
    \label{lemma:bs-to-treealpha} 
  Let $G$ be a graph, let $c \in [\frac{1}{2}, 1)$, and let $d$ be a positive integer. If for every normal weight function $w$ on $G$, there is a
    $(c,w)$-balanced separator $X_w$ with $\alpha(X_w) \leq d$, then 
        the tree independence number of $G$ is at most $\frac{3-c}{1-c}d$. 
\end{lemma} 
  We now prove:

  \begin{theorem}
    \label{smallsep}
    There exists an integer $M$ with the following property.
    Let $G$ be an even-hole-free graph with $n$ vertices and let $w$ be a normal function on $G$.
     Then there exists a $w$-balanced separator $Y$ in $G$ such that
  $\alpha(Y) \leq M \log^{10} n$.
\end{theorem}

  \begin{proof}
    Let $d$ be as in Theorem~\ref{domsep} and let $c$ be as in
    Theorem~\ref{banana}.
    By Theorem~\ref{domsep} and Theorem~\ref{banana}, it follows that
    $G$ is $(c \log^8n,d,1)$-breakable.
    Now the result follows from Theorem~\ref{domtosmall}.
    \end{proof}
  Theorem~\ref{main} now follows immediately from Theorem~\ref{smallsep}
  and Lemma~\ref{lemma:bs-to-treealpha}.





\end{document}